\documentclass[a4paper,twosided, 11pt]{article}
\usepackage[utf8]{inputenc} % encodage des caractères
\usepackage[T1]{fontenc} % encodage de la fonte
\usepackage[a4paper]{geometry} % la gestion de la géométrie de la page
\usepackage{amsmath, amssymb, mathtools, amsthm}
\usepackage{mathtools} % pour tous les ams[...]
\usepackage{mathrsfs}  % \mathscr
\usepackage{graphicx} % pour \includegraphics{monJPG}
\usepackage{xcolor}   % pour gérer les couleurs
\usepackage[english]{babel} % gestion des langues
\usepackage{bbm}
\usepackage[title]{appendix}
\newcommand{\eps}{\varepsilon}
\usepackage{caption}%    Personnalisation des légendes des flottants
\usepackage{subcaption}% Sous figures / sous tableaux
\usepackage{comment}
\usepackage{geometry}
\usepackage{cite}
\usepackage{float}
\usepackage[makeroom]{cancel}
\usepackage{ulem}
\usepackage{booktabs}
\usepackage{varwidth}
\usepackage{enumitem} %for labeled enumitem
\usepackage[pagewise]{lineno}
\usepackage[colorlinks]{hyperref} % les liens hypertextes

\newcommand \commentout[1] {}

\newcommand{\R}{\mathbb{R}}

\newcommand {\vp} {\varphi}
\newcommand {\bv} {\mathbf{v}}

\newcommand {\Chi} {{\bf \raise 2pt \hbox{$\chi$}} }
\newcommand {\dt}   {{\Delta t}}
\newcommand {\dx}   {{\Delta x}}

\newcommand {\Div}  { {\rm div} }
\newcommand {\dv}  { {\rm div} }

\newcommand {\f}   {\frac}
\newcommand {\p}   {\partial}

\newcommand*{\dd}{\mathop{\kern0pt\mathrm{d}}\!{}}

\newcommand*{\DD}{\mathop{\kern0pt\mathrm{D}}\!{}}

\newcommand*{\vecu}{\mathbf{v}}

\newcommand*{\vecsig}{\pmb{\sigma}}
\newcommand*{\vectau}{\pmb{\tau}}

\newcommand{\duality}[2]{\langle #1, #2 \rangle}

\newcommand{\corr}[2]{\textcolor{red}{\sout{#1}} \textcolor{blue}{#2}}

\def\onedot{$\mathsurround0pt\ldotp$}

\def\cddot{% two dots stacked vertically
  \mathbin{\vcenter{\baselineskip.67ex
    \hbox{\onedot}\hbox{\onedot}}%
  }}%

\DeclarePairedDelimiter{\abs}{\lvert}{\rvert}
\DeclarePairedDelimiter{\norm}{\|}{\|}

\newcommand*{\Id}{\mathbb{I}}

\DeclareMathOperator*{\spann}{span}

\theoremstyle{plain}
\newtheorem*{thm*}{Theorem}
\newtheorem{thm}{Theorem}[section]

\newtheorem{lemma}[thm]{Lemma}
\newtheorem{proposition}[thm]{Proposition}

\theoremstyle{remark}
\newtheorem{remark}[thm]{\bf Remark}
\newtheorem{definition}[thm]{\bf Definition}

\newcommand{\ie}{\textit{i.e.}\;}

\newcommand{\eg}{\textit{e.g.}\;}
\newcommand{\etal}{\textit{et al.}\;}
\newcommand{\etc}{\textit{etc.}\;}
\newcommand{\apriori}{\textit{a priori}\;}

\newcommand{\beq}{\begin{equation}}
\newcommand{\eeq}{\end{equation}}
\newcommand{\bea} {\begin{array}{rl}}
\newcommand{\eea} {\end{array}}
\newcommand{\bepa}{\left\{ \begin{array}{l}}
\newcommand{\eepa} {\end{array}\right.}
\newcommand{\diff}{\mathop{}\!\mathrm{d}}

\numberwithin{equation}{section}

\title{Analysis and numerical simulation of a generalized compressible Cahn-Hilliard-Navier-Stokes model with friction effects}

\author{Charles Elbar\thanks{Sorbonne Universit\'{e}, CNRS, Universit\'{e} de Paris, Inria, Laboratoire Jacques-Louis Lions (LJLL), F-75005 Paris, France } \thanks{Corresponding author. Email: charles.elbar@sorbonne-universite.fr} 
\and Alexandre Poulain\thanks{Univ. Lille, CNRS, UMR 8524 - Laboratoire Paul Painlevé, F-59000 Lille, France} \thanks{Email: alexandre.poulain@univ-lille.fr}
}

\date{\today}

\begin{document}

\maketitle

\begin{abstract}
We propose a new generalized compressible diphasic Navier-Stokes Cahn-Hilliard model that we name G-NSCH. 
This new G-NSCH model takes into account important properties of diphasic compressible fluids such as possible non-matching densities and contrast in mechanical properties (viscosity, friction) between the two phases of the fluid. the model also comprises a term to account for possible exchange of mass between the two phases.
Our G-NSCH system is derived rigorously and satisfies basic mechanics of fluids and thermodynamics of particles. 
Under some simplifying assumptions, we prove the existence of global weak solutions.  
We also propose a structure preserving numerical scheme based on the scalar auxiliary variable method to simulate our system and present some numerical simulations validating the properties of the numerical scheme and illustrating the solutions of the G-NSCH model. 
\end{abstract}
\vskip .7cm

\noindent{\makebox[1in]\hrulefill}\newline

\section{Introduction}

We derive, analyze and simulate numerically the generalized compressible Navier-Stokes-Cahn-Hilliard variant (\textit{G-NSCH} in short)
\begin{align}
    &\f{\p \rho}{\p t} + \dv\left(\rho \vecu\right)=0,\label{eq:main1}\\
    &\frac{\p (\rho c)}{\p t} + \dv\left(\rho c \vecu \right) = \dv\left(b(c) \nabla \mu \right) + F_c,\label{eq:main2}\\ 
    &\rho \mu = -\gamma\Delta c  + \rho \f{\p \psi_0}{\p c},\label{eq:main3}\\
     &\f{\p (\rho \vecu)}{\p t} + \dv\left(\rho \vecu \otimes\vecu\right) = \begin{multlined}[t][10cm] -\left[\nabla p+ \gamma\dv\left( \nabla c \otimes \nabla c - \frac{1}{2}\abs{\nabla c}^2 \Id \right) \right] + \dv\left(\nu(c)\left(\nabla \vecu + \nabla \vecu^T \right) \right)\\
   -\f23 \dv\left(\nu(c)\dv\left(\vecu\right)\Id\right) + \dv\left(\eta(c) \dv\left(\vecu\right) \Id \right)  - \kappa(\rho,c) \vecu\label{eq:main4},
    \end{multlined}
\end{align}
stated in $(0,T) \times \Omega$, where $T > 0$ is finite time horizon, and $\Omega \subset \R^d$ $(d=1,2,3)$ is an open bounded domain with a smooth boundary $\partial \Omega$. 

Interested by the modeling of invasive growth of tumors in healthy tissues, we motivate the different terms of the model with this biological application in mind.  However, we emphasize that the model is a general compressible diphasic fluid model that could be used for other applications. 

System~\eqref{eq:main1}--\eqref{eq:main4} models the motion of a diphasic fluid composed of two immiscible components, \ie two different cell types (\eg tumor and healthy cells), and comprises viscosity effects, surface tension, and friction on rigid fibers representing the extracellular matrix (\textit{ECM} in short). In System~\eqref{eq:main1}--\eqref{eq:main4}, $\rho$ is the total density of the mixture (\ie the sum of the two partial densities), $c$ is the relative mass fraction of one component (\eg the cancer cells), $\vecu$ is the mass averaged total velocity, $\mu$ is called the chemical potential, $p$ is the pressure. The coefficient $\gamma$ is related to the surface tension and is equal to the square of the width of the diffuse interface existing between the two populations. The friction coefficient $\kappa(\rho, c)$ is a non-negative function of the density and the mass fraction, and takes into account the possible difference of friction strength between the two populations. We use this friction term to model possible adhesive effects of the cells on the ECM. The coefficients $\nu(c)$ and $\eta(c)$ represents the viscosity coefficients (shear and dilatational, respectively) of the mixture. Possible differences in viscosities could be considered for the two populations. The function $\psi_0$ represents the separation of the two components of the mixture and phenomenologically models the behavior of cells (\ie cells tend to form aggregates of the same cell type). The function $F_c(\cdot)$ accounts for the possible proliferation and death of cells and these two effects are assumed to be modelled as an exchange of mass between the populations. 
The non-negative function $b(\cdot)$ models the mobility of cells. This function models the probability for a cell of any of the two populations to find an available neighboring spot to which it can move. More details about the general assumptions and precise forms of the different functions will be given in the next sections.

The motivation of our model stands from the modeling of tumor progression and invasion in healthy tissues. Indeed, as explained in Appendix~\ref{sec:assumptions-functions-2pop}, under suitable choices of functionals, our model can be viewed as a representation of a proliferating population of cells, \ie the tumor cells, in a domain filled with a non-proliferating population, \ie the healthy cells and the rest of tissue (ECM, extracellular fluid, \etc). The proliferation of cells happens by consuming mass from the other phase (we are not injecting mass in the system). Both cell populations move in an ECM constituted of rigid fibers on which they can adhere. As we only focus on the mechanical effects generated by the properties of the cells that could play a role during invasion, we do not consider in the model other effects that are known to be important in tumor progression: \eg angiogenesis, digestion of the ECM by proteolic enzymes, role of helping cells located in the stroma. 

We emphasize that this article only concerns the analysis and the numerical simulation of the G-NSCH model~\eqref{eq:main1}--\eqref{eq:main4}. This latter comprises effects that are negligible in biological situations, \eg inertia effects. We propose here an analysis of the model and a structure preserving numerical scheme for the G-NSCH model.

\paragraph{Literature review}
% Compressible NSCH and its use
%% Generalities
The motion of a binary mixture of two immiscible and compressible fluids can be described by the Navier-Stokes equation coupled to the Cahn-Hilliard model: the Navier-Stokes-Cahn-Hilliard model (\textit{NSCH model} in short). 
The well-known incompressible variant of the compressible NSCH model has been denominated model H (see \eg~\cite{hohenberg_1977_critical,gurtin_1996_binary}). 
%% Analysis
Model H has been proposed to represent viscous fluid flow in an incompressible binary mixture undergoing phase separation. This model assumes matching densities, \ie $\rho_1 = \rho_2$ and, hence, constant total density $\rho$. To consider non-matching densities, Lowengrub and Truskinovsky~\cite{Lowengrub-quasi-incompressible-1998} proposed the compressible NSCH model. Expanding the divergence term in the mass balance equation, the authors found a relation denoting the quasi-compressible nature of the fluid. Concomitantly, Anderson, FcFadden, and Wheeler~\cite{Anderson-1998-Diffuse} proposed a similar system. In the present work, we use a similar system. We also remark that a very recent work~\cite{eikelder_2023_unified} proposed a unified framework for the incompressible NSCH system and shows that the different NSCH models found in the literature only differ from their general modelling framework by specific constitutive hypotheses. 

Under some simplifying assumptions compared to the system proposed in~\cite{Lowengrub-quasi-incompressible-1998} but being closer to the system in~\cite{Anderson-1998-Diffuse}, the analysis of the compressible NSCH model with no-flux boundary conditions has been realized by Abels and Feireisl~\cite{abels_diffuse_2008}. Their analysis requires to simplify the model proposed in~\cite{Lowengrub-quasi-incompressible-1998} to avoid zones with zero density which would make this analysis a lot more difficult since the control from certain estimates would be lost.
% FOR CHARLES: Some words on the techniques they use 
In another article, for the same system, Abels proved the existence of strong solutions for short times~\cite{Abels_2012_strongwell}.
Considering the same assumptions and dynamic boundary conditions, Cherfils \etal~\cite{Miranville-2019-compressible} proved the well-posedness of the compressible NSCH model with these special boundary conditions. These latter allow to model the interaction of the fluid components and the walls of the domain.

Results on the analysis of the incompressible variant of the NSCH model, \ie the model H, are numerous and we here mention only a few of them since a complete review would be out of the scope of the present article. With a non-degenerate mobility coefficient and a physically relevant choice of potential, the well-posedness and regularity analysis of model H has been performed by Abels~\cite{abels2009diffuse} using tools both from the analysis of Navier-Stokes model and the Cahn-Hilliard model.
It is worth mentioning that the non-degeneracy of the mobility coefficient leads to non-physical effects, \ie Ostwald ripening effects (see~\cite{Abels_2012_thermo}). For this reason, Abels, Depner and Garcke studied model H with a degenerate mobility~\cite{abels2013incompressible}. Their analysis relies on a regularization of the mobility and singular potential into, respectively, a non-degenerate and non-singular potential. Then, suitable \textit{a-priori} estimates uniform in the regularization parameter allow to pass to the limit in the regularization and show the existence of weak solutions to the degenerate model H. 

% CH variants for tumor growth
We now review partially the extensive literature about the Cahn-Hilliard equation and its use for the modelling of tumors. 
The Cahn-Hilliard equation has been initially used to represent the phase separation in binary mixtures and has been applied to the spinodal decomposition of binary alloys under a sudden cooling~\cite{cahn_free_1958,cahn_spinodal_1961}. 
The model represents the two phases of the fluids as continua separated by a diffuse interface. 
This equation has been used later in many different applications and we do not intend here to give an overview of all these. However, we refer the reader interested in the topic to the presentation of the Cahn-Hilliard equation and its applications to the review book~\cite{miranville-book-2019}. We are interested here in the application of the Cahn-Hilliard framework to tumor modelling (see \eg~\cite{lowengrub_analysis_2012,lowengrub_analysis_2013}). Latter, different variants of the Cahn-Hilliard model appeared: \eg (without giving a complete overview) its coupling to Darcy's law~\cite{garcke_2016_CHdarcy}, Brinkman's law~\cite{Ebenbeck_2021_CH}, chemotaxis~\cite{ROCCA2023530}.  
Recently, a variant of the CH equation has been used to better represent the growth and organization of tumors. The main change is the use of a single-well logarithmic degenerate potential instead of a double-well potential~\cite{chatelain2011emergence,Agosti-CH-2017,poulain_relaxation_2019}. This type of potential has been proposed in~\cite{ambrosi_closure_2002} to represent the action of the cells depending only on the local density, \ie attraction at low cell density and repulsion for large cell density representing the tendency of cells to avoid overcrowding. 
The Cahn-Hilliard framework has also been utilized in systems representing invasive growth of tumors. The interested reader can find a lot of information about phase-field type systems modelling tumor growth and invasion in the very recent survey paper~\cite{Fritz_tumor_Tumor} and references therein.

% What about its numerical simulation
We now review some of the literature about the numerical simulation of NSCH models.
The numerical simulation of Model H for binary fluids with non-matching densities has been the subject of numerous works (see \eg~\cite{Hosseini-Isogeometric-2017} and references therein). However, in part due to its complexity, the numerical simulation of the compressible NSCH system has been less explored. 
A $C^0$ finite element numerical scheme for a variant of the quasi-compressible NSCH model proposed in~\cite{Lowengrub-quasi-incompressible-1998} has been proposed in~\cite{Guo-quasiNSCH-2014}. Around the same time, Giesselmann and Pryer~\cite{aki_quasi_2014, Giesselmann_2015_energy} designed a discontinuous Galerkin finite element scheme to simulate the quasi-incompressible NSCH system which preserves the total mass and the dissipation of energy.    
A numerical method has also been proposed in~\cite{Quaolin-2020-compressible} in the case of constant mobility $b(c)$ and smooth polynomial potential $\psi(c)$. However, the system simulated in~\cite{Quaolin-2020-compressible} is a simplification of the compressible NSCH system since the pressure does not appear in the definition of the chemical potential $\mu$ in their system.  

The previous works we presented for the simulation of the compressible or quasi-compressible NSCH systems deal with constant mobility combined with a smooth polynomial potential. We aim to simulate the compressible NSCH model with choices of mobility and potential relevant for biology (but also relevant for material sciences and fluid mechanics), \ie degenerate mobility combined with a logarithmic potential. We now review briefly some relevant discretization methods for the Cahn-Hilliard equation with degenerate mobility and singular potentials. Considering a degenerate mobility and a double-well logarithmic potential, we mention the work of Barrett, Blowey and Garcke~\cite{barrett_finite_1999}. In this article the authors proposed a finite element scheme with a variational inequality to preserve the bounds of the solution. 
Based on these ideas, Agosti \etal~\cite{Agosti-CH-2017} proposed a similar finite element scheme for the single-well logarithmic potential case. The difficulty in this latter case lies in the fact that the degeneracy and the singularity sets do not coincide and, considering an order parameter that must remain within the bounds $[0,1)$, negative solutions can appear if a standard discretization method is used. The numerical scheme proposed in~\cite{Agosti-CH-2017} solves this issue but does not preserve the mass. In a more recent work, Agosti~\cite{agosti_discontinuous_2019} proposed a discontinuous Galerkin finite element scheme that preserves the bounds $[0,1)$ and preserves the exact mass. 
However, the main drawback of the previously mentioned methods is that they are computationally expensive: they solve a strongly coupled nonlinear system and resort to the use of iterative algorithms. 

Since the Cahn-Hilliard equation is a gradient flow (see \eg~\cite{Lisini-CH-gradient-flow}), a structure-preserving linear scheme can be constructed using the Scalar Auxiliary Variable (\textit{SAV} in short) method~\cite{shen_2018_sav}. 
The SAV method is a very powerful tool to design unconditionally energy-stable numerical schemes for models possessing a gradient-flow (see \eg~\cite{Shen_SAVapproaches_2020, Yanrong_2022_generalizedSAV} and references therein) or Hamiltonian structure (see \eg~\cite{Antoine_2021_SAV, Poulain_2022_SAV} and references therein). The SAV method has evolved during the past 6 years starting from the original SAV method~\cite{shen_2018_sav, shen_new_2019} to improved variants such as the generalized version GSAV (see \eg~\cite{Yanrong_2022_generalizedSAV, Fukeng_newclassSAV_2022}) and the relaxed RSAV method~\cite{Maosheng_2022_improvingSAV}. In our work, we use the GSAV method that has been already used in~\cite{huang-SAV-fourth} for the Cahn-Hilliard equation. In this latter work, the scheme is structure-preserving from the use of a scalar variable that represents the discrete energy, and an additional equation is solved to ensure dissipation at the discrete level. The bounds of the order parameter are ensured using a transformation that maps $\mathbb{R}$ to the physical relevant interval ($(0,1)$ in the case of a double-well potential).  
Hence, compared to other techniques, the SAV method has the advantage to allow for the design of a linear, efficient, structure-preserving scheme and can easily be used for our G-NSCH system. We also emphasize that the SAV method has been used for the simulation of the incompressible NSCH model with positive mobility and polynomial potential in~\cite{Xiaoli_SAV_2020}. In the present work, we use the GSAV method to design a numerical scheme for the G-NSCH model. Our numerical scheme allows to use degenerate mobility and singular potential functionals which is more physically relevant. 
To the best of our knowledge, our G-NSCH model is new because it comprises the friction force term and the exchange between the two phases of the fluid. Moreover, the use of the GSAV method for a compressible NSCH system is new, especially with our choice of functionals (\ie degenerate mobility and singular potential). 

\paragraph{Objectives of our work}
The first objective of our work is to study the well-posedness of the G-NSCH model under some simplifying assumptions (\ie smooth potential and positive mobility). The second objective is the design of an efficient and structure-preserving numerical scheme for the G-NSCH model with singular double-well potential and degenerate mobility. The third focus of the present work concerns the rigorous derivation of the G-NSCH model that is presented in the Appendix.  

\paragraph{Outline of the paper} 
Section~\ref{sec:notation} presents the notations, functional spaces and assumptions we use in our work for the analytical part but also for the numerical part.  
Section~\ref{sec:existence} concerns the proof of the existence of weak solutions for the G-NSCH system~\eqref{eq:main1}--\eqref{eq:main4} under simplifying assumptions. A structure preserving numerical scheme based on the GSAV method is then proposed in Section~\ref{sec:scheme} and some numerical results are presented in Section~\ref{sec:results}.   
Our model's equations come from a thermodynamically consistent derivation of the compressible Navier-Stokes-Cahn-Hilliard model including friction effects and source terms. The derivation is described in Appendix~\ref{sec:derivation-model-2pop}. From the general model, we propose in Appendix~\ref{sec:assumptions-functions-2pop} two reductions: The G-NSCH studied and simulated in the present work and one biologically relevant reduction that will be the focus of a forthcoming work.

\section{General assumptions, notations and functional setting} \label{sec:notation}
\begin{sloppypar}
The equations are set in a domain $\Omega_T = \Omega \times (0,T)$ with $\Omega$ an open and bounded subset of $\R^d$ ($d=1,2,3$). We assume that the boundary $\p \Omega$ is sufficiently smooth. We indicate the usual Lebesgue and Sobolev spaces by respectively $L^p(\Omega)$, $W^{m,p}(\Omega)$ with ${H^m(\Omega) := W^{m,2}(\Omega)}$, where $1 \le p \le +\infty$ and $m \in  \mathbb{N}$. \iffalse We denote the corresponding norms by $||\cdot||_{m,p,\Omega}$, $||\cdot||_{m,\Omega}$ and semi-norms by $|\cdot|_{m,p,\Omega}$, $|\cdot|_{m,\Omega}$. The standard $L^2$ inner product will be denoted by $(\cdot,\cdot)_\Omega$ and the duality pairing between $(H^1(\Omega))'$ and $H^1(\Omega)$ by $<\cdot,\cdot>_\Omega$. \fi For $q\in [1,+\infty]$, we indicate the Bochner spaces by $L^q(0,T;X)$ (where $X$ is a Banach space). Finally, $C$ denotes a generic constant that appears in inequalities and whose value can change from one line to another. This constant can depend on various parameters unless specified otherwise.
\end{sloppypar}

\subsection{Assumptions on functionals}\label{subsec:assumptions}
We divide the assumptions on the different terms appearing in system ~\eqref{eq:main1}--\eqref{eq:main4} into two parts: analytical and numerical assumptions. Indeed we are not able to prove the existence of weak solutions in the general setting used for the numerical simulations. For instance, the case of the usual logarithmic double-well potential in the Cahn-Hilliard equation is not treated but can be implemented in our numerical scheme. However, we can analyze our system with a polynomial approximation of the double well. We also consider non-degenerate mobilities to obtain estimates on the chemical potential $\mu$ directly. The case of degenerate mobility, see for instance~\cite{elliott_cahn-hilliard_1996}, seems unavailable as we do not have anymore the classical “entropy” estimates of the Cahn-Hilliard equation that provide bound on second-order derivatives of the mass fraction $c$.

\textbf{Framework for numerical simulations}
We assume that the viscosity $\nu(c), \eta(c)$ and permeability $\kappa(\rho,c)$ coefficients are smooth non-negative functions. The mobility is a non-negative function of the order parameter (mass fraction) $c$. Hence, we assume that 
\begin{equation}
\begin{aligned}
    b \in C^1([0,1];\R^+),\quad &\text{and}\quad b(c)\ge 0\quad \text{for}\quad 0\le c\le1.
    \end{aligned}
\end{equation}

In agreement with the literature (see e.g \cite{Miranville-2019-compressible}), the homogeneous free energy $\psi_0(\rho,c)$ is assumed to be of the form
\begin{equation}
    \psi_0(\rho,c) = \psi_e(\rho) + \psi_\text{mix}(\rho,c),
\end{equation}
with $ \psi_\text{mix}(\rho,c)= H(c)\log \rho + Q(c)$ 
and $Q(c)$ is a double-well (or single-well) potential. Then, using the constitutive relation for the pressure, we have
\begin{equation}
    p(\rho,c) = \rho^2\f{\p \psi_0}{\p\rho} = p_e(\rho) + \rho H(c), \label{eq:def-pressure}
\end{equation}
where $p_e = \rho^2 \psi'_e(\rho)$ and is assumed to satisfy
\begin{equation}
    p_1\rho^{a-1}-p_2\le p'_e(\rho) \le p_3(1+\rho^{a-1}), \quad \text{for}\quad a>3/2,\quad p_1,p_2,p_3>0.
\end{equation}

We assume that the exchange term $F_c$ (that can depend on the mass fraction and the density) is bounded,
\begin{equation}
\left|F_c(\rho,c)\right| + \left|\f{F_{c}(\rho,c)}{\rho}\right| \le C,\, \forall(\rho,c)\in \R^{2}.  \label{eq:assumpt-source}
\end{equation}
%We have in mind examples of the form $F_c(\rho,c)= \rho c G(p)$ with $G$ depending on the pressure such that $G(p)=0$ for $|p|\ge p_H$ where $p_H$ is the homeostatic pressure which is the lowest level of pressure that prevents cell multiplication due to contact-inhibition.  

\begin{remark}[Double-well logarithmic potential]
    In the present work, we aim to use a double-well logarithmic potential in the definition of the mixing potential. A relevant example of potential is 
    \begin{equation}
    \psi_\text{mix} = \f12 \left(\alpha_1(1-c)\log(\rho(1-c)) + \alpha_2c\log(\rho c)\right) - \f\theta 2 (c-\frac 1 2 )^2.
\end{equation}
This potential gives 
\[
    H(c) = \frac{1}{2} \left(\alpha_1(1-c) +  \alpha_2c\right), \quad Q(c) = \f12 \left(\alpha_1(1-c)\log(1-c) + \alpha_2c\log(c)\right) - \f\theta 2 (c-\frac 1 2 )^2,
\]
where $\theta>1$. 

\end{remark}

\paragraph{Additional assumptions for the existence of weak solutions and analysis of the numerical scheme.}
Concerning the existence of weak solutions and analysis of the numerical scheme, we need to strengthen our assumptions. The viscosity coefficients $\nu(c), \eta(c)$ are assumed to be bounded from below by a positive constant and the friction coefficient $\kappa(\rho,c)$ is assumed to be nonnegative. Moreover, $\nu(c), \eta(c)$ and $\kappa(\rho,c)$ are functions bounded in $L^{2}(0,T;L^{2}(\Omega))$ whenever $c$ is bounded in $L^{\infty}(0,T;H^{1}(\Omega))$ and $\rho$ is smooth (for instance $C(0,T;C^{2}(\overline{\Omega}))$. We consider $a>2$ the exponent of the pressure law. In the numerical simulations, we take degenerate mobilities of the form $b(c)=c(1-c)^{\alpha}$. However, in the analysis, we consider a non-degenerate mobility by truncating the previous mobility. For instance, using a small parameter $0<\varepsilon_{b}<<1$, we approximate the mobility $b(\cdot)$ by 
\[
b_{\eps_{b}}(c) = \begin{cases}
    b(1-\eps_{b}),\quad \text{if } c \ge 1-\eps_{b},\\
    b(\eps_{b}),\quad \text{if } c \le \eps_{b},\\
    b(c),\quad \text{otherwise},
\end{cases}
\]
and consider the case of a fixed $\eps_{b}$. Dropping the $\eps_b$ subscript, we obtain that 
\begin{equation}
\begin{aligned}
    b \in C^1(\R;\R^+),\quad &\text{and}\quad b(c)\ge C>0\quad \forall  c\in\R.
    \end{aligned}
    \label{eq:cond-mot}
\end{equation}

Concerning the functionals appearing in the definition of the free energy $\psi_{0}$ we assume that $H$ and $H'$ are bounded and that $Q$ is a polynomial approximation of the double well potential. More precisely we take 
\begin{equation}
    \begin{aligned}
        &H_1\le H'(c),\, H(c) \le H_2,\quad c\in\R,\quad H_1,H_2>0,\\
        &Q(c)=\f{1}{4}c^{2} (1-c)^{2}. 
    \end{aligned}
    \label{eq:assumptionHQ}
\end{equation}

The case of the double-well logarithmic potential has not been tackled yet even though this is the main motivation for the decomposition of $\psi_{mix}$ as in the works~\cite{abels_diffuse_2008} and~\cite{Miranville-2019-compressible}.

Also, to make the computations simpler, we assume that
\begin{itemize}
    \item  $a>6$ where $a$ is the pressure exponent,
    \item $\psi_{e}(\rho)=\f{\rho^{a-1}}{a-1}$ and therefore $p_e(\rho)=\rho^{a}.$
\end{itemize}
These two assumptions are not necessary {and could be removed} but simplify the analysis. We refer for instance to~\cite{abels_diffuse_2008,feireisl-book} for the more general setting. For instance, the condition $a>6$ is used to not introduce another parameter in the approximating scheme which would make the article longer. Note that the assumptions on $\psi_{0}$ imply in particular the following lemma which is essential to obtain estimates on the energy dissipation: 

\begin{lemma}\label{lem:psi_0_estim}
There exists a constant $C$ such that 
\begin{equation*}
\left|\rho \f{\p\psi_0}{\p c}\right|\le C \rho \psi_{0} + C.   
\end{equation*}
\end{lemma}
Its proof uses the assumption on $H$ and the fact that for $c$ large, $Q'(c)\approx c^3\le c^4+1\approx Q(c)+1$.

\section{Existence of weak solutions} \label{sec:existence}
We now turn to the proof of the existence of weak solutions for the G-NSCH model~\eqref{eq:main1}--\eqref{eq:main4} subjected to boundary conditions
\begin{equation}
\label{eq:boundary}
\vecu=0,\,\f{\p c}{\p\mathbf{n}}=b(c)\f{\p\mu}{\p\mathbf{n}}=0, \quad \text{on $\p\Omega$},
\end{equation}
and initial conditions
\begin{equation}\label{init_cond}
    \rho(0,x) = \rho_0\ge 0 \in L^a(\Omega),\quad   c(0,x) = c_0 \in H^1(\Omega) \quad \rho_0 \vecu(0,x) = \mathbf{m}_0,\, \text{with } \frac{\abs{\mathbf{m}_0}^2}{\rho_0} \in L^1(\Omega). 
\end{equation}

Also, we suppose $\rho_0\ne 0$. 
In this section we take $d=3$.
The proof of the result is quite long and technical. Therefore, when possible and for the sake of clarity, we omit some proofs and give instead appropriate references. 

\paragraph{Outline of the analysis} For readability reasons, we here present the outline of the analysis of the G-NSCH model. We first start with the analysis of a "truncated" version of G-NSCH model in the sense that the double-well is truncated for large values of $c$ with a parameter $\eps_{Q}$. Then, for this fixed truncation, we prove the existence of weak solutions using the ideas of~\cite{Lions-1998-fluids,feireisl-book,abels_diffuse_2008,Miranville-2019-compressible}. Then, we pass to the limit $\eps_{Q}\to 0$. Namely, recalling that $Q(c)=\f{1}{4}c^{2}(1-c)^{2}$ we first consider $Q_{\eps_{Q}}(c)$ a smooth truncated approximation of $Q$ that satisfies 

\iffalse
\[
Q_{\eps_{Q}}''(c) = \begin{cases}
    Q''(1-\eps_{Q}),\quad \text{if } c \ge 1-\eps_{Q},\\
    Q''(\eps_{Q}),\quad \text{if } c \le \eps_{Q},\\
    Q''(c),\quad \text{otherwise}.
\end{cases}
\]
\fi

\begin{equation}\label{ass:Q_reg}
    |Q_{\eps_{Q}}|, |Q'_{\eps_{Q}}|, |Q''_{\eps_{Q}}|\le C\left(\f{1}{\eps_{Q}}\right).  
\end{equation}

 In the first subsections, we drop the $\varepsilon_{Q}$ notation and work with the regularized problem. We will use the $\varepsilon_{Q}$ notation when we pass to the limit. For the moment, we benefit from the properties of the regularization. 

\subsection{Energy estimates}

The G-NSCH system comes with an energy structure which is useful to obtain first \apriori estimates. 

\begin{proposition}\label{prop:energy_smooth}
Smooth solutions of the system~\eqref{eq:main1}--\eqref{eq:main4} satisfy the following energy relation
\begin{equation}
\label{eq:energy-diff}
\f{d}{dt}E+D=\int_{\Omega}\mu F_c\diff x,
\end{equation}

where $E$ is the energy, and $D$ is the dissipation defined as 

\begin{align}
E&=\int_{\Omega}\rho \f{|\bv|^{2}}{2}+\rho \psi_{0}+ \f{\gamma}{2}|\nabla c|^{2}\diff \mathbf{x}\label{eq:energy},\\
D&=\int_{\Omega}\f{\nu(c)}{2}\left|\nabla \mathbf{v}+ \nabla \mathbf{v}^{T}-\f{2}{3}\Div(\mathbf{v})\mathbb{I}\right|^{2} + \eta(c) \abs{\dv\left( \vecu\right)\Id}^2 +b(c)|\nabla\mu|^{2}+\kappa(\rho,c)|\bv|^{2}\diff\mathbf{x}.\label{eq:diffusion}
\end{align}

This yields a priori estimates on the solution \ie there exists a positive constant $C$ such that

\begin{equation*}
E(t)+\int_{0}^{t}D(s)\diff s\le C+ CE(0). 
\end{equation*}
\end{proposition}

Note that the energy is bounded from below since $\rho\log\rho H(c)$ is bounded from below with~\eqref{eq:assumptionHQ}. Also, the purpose of the assumptions $\nu(c), \eta(c)$ and $b(c)$ bounded from below by a positive constant becomes clear, they are crucial to obtain estimates on the $H^{1}(\Omega)$ norm of $\mu$ and $\vecu$. 

\begin{proof}
We recall the formula 
\begin{equation}\label{eq:formulaDelta}
\nabla c\Delta c=\mbox{div}(\nabla c\otimes \nabla c)-\frac{1}{2}\nabla |\nabla c|^{2}.
\end{equation}
We denote by $\mathbb{T}$ the tensor 
\begin{equation}\label{def:tensorT}
\mathbb{T}=\nu(c) (\nabla \vecu+ \nabla \vecu^{T}-\f{2}{3}\Div(\vecu)\mathbb{I}) +  \eta(c)\dv\left(\vecu \right) \Id.
\end{equation}
Then we multiply Equation~\eqref{eq:main1} by $\f{|\vecu|^2}{2}$ and sum it with the scalar product of Equation~\eqref{eq:main4} with $\vecu$. We obtain 
\begin{multline*}
\f{\p}{\p t}\left(\rho\f{|\vecu|^{2}}{2}\right) + \Div\left(\f{1}{2}\rho |\vecu|^{2} \vecu + p(\rho,c)\vecu-\mathbb{T}\cdot\vecu\right) + \mathbb{T}:\nabla\vecu+\kappa(\rho,c)\vecu^2=p(\rho,c)\Div (\vecu) \\+\gamma \Div(\f{1}{2} |\nabla c|^{2}\mathbb{I} -( \nabla c\otimes \nabla c))\cdot\bv, 
\end{multline*}
which is equivalent to 
\begin{equation}
\label{eq:apriori1}
\f{\p}{\p t}\left(\rho\f{|\mathbf{v}|^{2}}{2}\right) + \Div\left(\f{1}{2}\rho |\mathbf{v}|^{2} \mathbf{v} + p(\rho,c)\mathbf{v}-\mathbb{T}\cdot\mathbf{v}\right) + \mathbb{T}:\nabla\mathbf{v}+\kappa(\rho,c)\mathbf{v}^2=p(\rho,c)\Div \mathbf{v}  -\gamma \Delta c\nabla c\cdot\mathbf{v}.
\end{equation}
Then, we multiply Equation~\eqref{eq:main2} by $\mu$ and obtain using also~\eqref{eq:main1}
\begin{equation*}
\rho\mu(\p_{t}c+\mathbf{v}\cdot\nabla c)=\Div(b(c)\nabla\mu)\mu+\mu F_c. 
\end{equation*}
And, using~\eqref{eq:main3} we obtain 
\begin{equation*}
\rho\f{\p\psi_{0}}{\p c}(\p_{t}c+\mathbf{v}\cdot\nabla c)=\Div(b(c)\nabla\mu)\mu+\gamma \Delta c(\p_{t}c+\mathbf{v}\cdot\nabla c) + \mu F_c.   
\end{equation*}
The previous equation can be rewritten using the chain rule as
\begin{equation*}
\begin{aligned}
\p_{t}(\rho\psi_{0})+\Div(\rho\psi_{0}\mathbf{v})&-\psi_{0}(\p_{t}\rho+\Div(\rho\mathbf{v}))-\rho\f{\p\psi_{0}}{\p\rho}(\p_{t}\rho+\mathbf{v}\cdot\nabla \rho)\\
&=\Div(b(c)\nabla\mu)\mu+\gamma \Delta c(\p_{t}c+\mathbf{v}\cdot\nabla c) + \mu F_c.  
\end{aligned}
\end{equation*}
We have $\rho\f{\p\psi_{0}}{\p\rho}(\p_{t}\rho+\mathbf{v}\cdot\nabla \rho)=\rho\f{\p\psi_{0}}{\p\rho}(-\rho\Div(\vecu))=-p\Div(\mathbf{v})$ (see Equation~\eqref{eq:def-pressure} for the definition of the pressure). Moreover, we know that $\Delta c\p_{t}c=\Div(\p_{t}c\nabla c)-\p_{t}\left(\f{|\nabla c|^{2}}{2}\right)$ and, hence, 
\begin{multline}
 \label{eq:apriori2}
\p_{t}(\rho\psi_{0})+\Div(\rho\psi_{0}\mathbf{v})+p\Div(\mathbf{v})=\Div(b(c)\nabla\mu)\mu+\gamma \left[\Div(\p_{t}c\nabla c)-\p_{t}\left(\f{|\nabla c|^{2}}{2}\right)+\Delta c\mathbf{v}\cdot\nabla c \right]\\+\mu F_c.    
\end{multline}
Summing~\eqref{eq:apriori1} and~\eqref{eq:apriori2} we obtain
\begin{multline*}
\f{\p}{\p t}\left(\rho \f{|\vecu|^{2}}{2}+\rho \psi_{0}+ \f{\gamma}{2} |\nabla c|^{2}\right)+\Div \left(\rho\psi_{0}\mathbf{v}+\f{1}{2}\rho |\bv|^{2}\bv+p(\rho,c)\bv - \mathbb{T}:\bv - \gamma \p_{t}c\nabla c\right) - \Div(b(c)\nabla\mu)\mu  \\+ \mathbb{T}:\nabla\bv + \kappa(\rho,c)|\bv|^{2}= \mu F_c .
\end{multline*}
Now we use the fact that 

\begin{equation}
\label{eq:tensorT}
\mathbb{T}:\nabla \bv=\f{\nu(c)}{2}\left|\nabla \mathbf{v}+ \nabla \mathbf{v}^{T}-\f{2}{3}\Div(\mathbf{v})\mathbb{I}\right|^{2} + \eta(c) \left|\dv\left(\vecu \right) \Id\right|^2. 
\end{equation}

Integrating in space and using the boundary conditions~\eqref{eq:boundary} ends the proof of the first part of the proposition. To prove the second part, we integrate the equation in time and control the right-hand side. 
Indeed, due to the assumption on the source term~\eqref{eq:assumpt-source}, we have
 $$
\left| \int_{0}^{t}\int_{\Omega}\mu F_c\diff x\diff t\right|\le C\int_{0}^{t}\int_{\Omega}|\mu|.  
 $$

We want to use Lemma~\ref{lemma:PW} to control the $L^{1}$ norm of $\mu$. Integrating the equations on $\rho$ to obtain $\int_{\Omega}\rho\diff x=\int_{\Omega}\rho_{0}\diff x> M_{0}$ we satisfy the first assumption of the lemma. For the second, we notice that we can consider a variant of this lemma such that instead of asking $\rho$ to be in $L^{6/5}$ we have the inequality 
$$
\left\|\mathbf{u}- \f{1}{|\Omega|}\int_{\Omega}\rho\mathbf{u}\right\|_{L^{2}}\le C \norm{\nabla \mathbf{u}}_{L^{2}}+\norm{\rho}_{L^{6/5}}. 
$$

Using Young's inequality, the fact that in the energy $\rho\psi_{0}$ contains a term of the form $\rho^{a+1}$ we obtain for $\widetilde{C}$ small enough
$$
\int_{0}^{t}\int_{\Omega}|\mu|\diff x \le C + \widetilde{C}\int_{0}^{t}\int_{\Omega}|\mu|^2\diff x \le C + C E(t)+ \f{\inf_{c}b(c)}{2}\int_{\Omega}|\nabla \mu|^{2}\diff x + C \left|\int_{\Omega}\rho\mu\diff x\right|.    
 $$
Since the energy dissipation controls the third term of the right-hand side, it remains to control the last term of the right-hand side. We recall that $\rho\mu=\rho\f{\p \psi_{0}}{\p c}-\gamma\Delta c$. Using the Neumann boundary conditions on $c$, it remains to control $\left|\int_{\Omega}\rho\f{\p \psi_{0}}{\p c} \right|$. Using Lemma~\ref{lem:psi_0_estim}, we obtain 

$$
\left|\int_{\Omega}\rho\f{\p \psi_{0}}{\p c}\diff x\right|\le C+ CE(t).
$$
We conclude using Gronwall's lemma.  
\end{proof}

\subsection{Existence of weak solutions for fixed $\eps_{Q}$}

%We first prove existence of weak solutions with regularized $Q=Q_{\eps_{Q}}:=Q_{1,\eps_{Q}}+Q_{2}$ and then we pass to the limit when $\eps_{Q}$ tends to 0.  

The weak solutions of system~\eqref{eq:main1}--\eqref{eq:main4} are defined as follows

\begin{definition}\label{def:existence_sol}
We say that $(\rho,\bv,c, \mu)$ is a weak of system~\eqref{eq:main1}--\eqref{eq:main4} provided:
\begin{itemize}
\item $\rho\ge 0$ and we have the regularity \begin{align*}
 &\rho\in L^{\infty}(0,T;L^{a}(\Omega)),\\ 
 & \bv\in L^{2}(0,T;H_{0}^{1}(\R^{3})), \quad \sqrt{\rho}\bv\in L^{\infty}(0,T;L^{2}(\Omega;\R^{3})), \quad \mathbb{T}:\nabla\bv\in L^{1}(0,T;L^{1}(\Omega)),\\
 & c\in L^{\infty}(0,T;H^{1}(\Omega)),\\
 &\mu\in L^{2}(0,T;H^{1}(\Omega)). 
\end{align*}

\item Equations~\eqref{eq:main1}--\eqref{eq:main4} are satisfied in the distributional sense.
\item The initial conditions~\eqref{init_cond} are satisfied a.e. in $\Omega$. 
\item The boundary conditions~\eqref{eq:boundary} are satisfied. 
\end{itemize}

\end{definition}

We state our main theorem about the existence of weak solutions 

\begin{thm}[Existence of weak solutions]\label{thm:existence_weak_sol}
There exist $(\rho,\bv,c, \mu)$ weak solutions of~\eqref{eq:main1}--\eqref{eq:main4} in the sense of Definition~\ref{def:existence_sol}.
\end{thm}

In order to prove the existence of weak solutions, we use an approximating scheme with a small parameter $\eps>0$ borrowing the idea from~\cite{MR1637634,feireisl-book}. 
More precisely,  let $X_{n}=\spann\{\eta_{i}\}_{i=1,...,n}$ be the set of the first $n$ vectors of a basis of $H_{0}^{1}(\Omega;\R^{3})$ such that $X_{n}\subset C^{2}(\overline{\Omega};\R^{3})$. We consider the following problem for $(\rho, \bv_{n},c)$ with $\bv_{n}\in X_{n}$ (with coordinates depending on time):
\begin{equation}
\label{eq:density_reg}
\p_{t}\rho+\Div(\rho \bv_{n})=\eps\Delta\rho,
\end{equation}
and for every $\eta\in X_{n}$, 
\begin{multline}
\label{eq:velocity_reg}
\int_{\Omega}\rho\bv_{n}(t)\cdot\eta\diff x- \int_{\Omega}\bold{m}_{0}\cdot\eta\diff x- \int_{0}^{t}\int_{\Omega}\rho \bv_{n}\otimes\bv_{n}:\nabla\eta\diff x\diff s- \int_{0}^{t}\int_{\Omega}p(\rho,c)\Div(\eta)\diff x\diff s \\
+ \eps\int_{0}^{t}\int_{\Omega}(\nabla\bv_{n}\nabla\rho)\cdot\eta\diff x\diff s+\int_{0}^{t}\int_{\Omega}\mathbb{T}:\nabla\eta\diff x\diff s+ \gamma\int_{0}^{t}\int_{\Omega}(\f{1}{2} |\nabla c|^{2}\mathbb{I} -( \nabla c\otimes \nabla c))
:\nabla\eta\diff x\diff s\\ + \int_{\Omega}\int_{0}^{t}\kappa(\rho,c)\bv_{n}\cdot\eta\diff x\diff s=0.
\end{multline}
And for the equation on the mass fraction 
\begin{equation}
\label{eq:mass_fraction_reg}
\p_{t}c+\bv_{n}\cdot\nabla c=\f{1}{\rho}\Div(b(c)\nabla\mu)+ \frac{F_c}{\rho},\quad \mu=\f{\p\psi_{0}}{\p c}-\gamma\f{\Delta c}{\rho}.
\end{equation}
We consider Neumann boundary conditions
\begin{equation}
\label{eq:neumann_reg}
\nabla \rho\cdot\bold{n}=b(c)\nabla\mu\cdot\bold{n}=\nabla c\cdot\bold{n}=0\quad \text{on $\p\Omega$},
\end{equation}
and the Dirichlet boundary condition for $\bv_{n}$ is included in the definition of $X_{n}$. 
Finally, we consider the initial conditions
\begin{equation}
\label{eq:init_cond_reg}
\rho(0,\cdot)=\rho_{0,\eps}>0, \quad c(0,\cdot)=c_{0,\eps}, \quad \rho \bv_{n}(0,\cdot)=\bold{m}_{0},    
\end{equation}
where $\rho_{0,\eps}$, $c_{0,\eps}$ satisfy the Neumann boundary conditions and they are smooth approximations of $\rho_0, c_0$ (when $\eps\to 0$). 

We now comment on the scheme used above and detail the strategy of the proof. We add the artificial diffusion in~\eqref{eq:density_reg} with the parameter $\eps>0$. Here, $\bv_{n}$ is fixed and we can conclude the global in time existence of classical solutions to~\eqref{eq:density_reg} which are positive since the initial condition is positive (and using maximum principle). Using this positivity, we conclude the existence of a strong solution to Equation~\eqref{eq:mass_fraction_reg} which is in fact a fourth-order parabolic equation. Having obtained $c$, we focus on Equation~\eqref{eq:velocity_reg} and we prove existence for a small time with Schauder's fixed point theorem. Note the presence of the additional term $\eps\int(\nabla\bv_n \nabla \rho)\cdot\eta$ which is useful to cancel energy terms introduced by $\eps\Delta\rho$ in~\eqref{eq:density_reg}. Having obtained existence on a short time interval we compute the energy of the system and obtain global existence. Then, we pass to the limit $n\to\infty$. It remains to send $\eps$ and $\eps_{Q}$ to 0 and obtain solutions of system~\eqref{eq:main1}--\eqref{eq:main4}.

We first turn our attention to Equation~\eqref{eq:density_reg}. From~\cite{feireisl-book}, we obtain the following proposition, and lemma

\begin{proposition}\label{prop:density_reg}
Let $\Omega\subset \R^{3}$ be a bounded domain of class $C^{2+\beta}$ for some $\beta>0$. For a fixed $\bv_{n}\in X_{n}$, there exists a unique solution to Equation~\eqref{eq:density_reg}
with Neumann boundary conditions~\eqref{eq:neumann_reg} and initial data conditions~\eqref{eq:init_cond_reg}. 
Furthermore, the mapping $\bv_{n}\mapsto \rho[\bv_{n}]$, that assigns to any $\bv_{n}\in X_{n}$ the unique solution of~\eqref{eq:density_reg}, takes bounded sets in the space $C(0,T;C^{2}_{0}(\overline{\Omega},\R^{3}))$ into bounded sets in the space
\begin{equation*}
V:=\{\p_{t}\rho \in C(0,T;C^{\beta}(\overline{\Omega})) , \,\rho\in C(0,T;C^{2+\beta}(\overline{\Omega})) \}.  
\end{equation*}
\end{proposition}

\begin{lemma}\label{lem:maximum_principle}
The solutions of~\eqref{eq:density_reg} satisfy

\begin{equation*}
\begin{aligned}
(\inf_{x\in\Omega}\rho(0,x))\exp\left(-\int_{0}^{t}\norm{\Div\,  \bv_{n}(s)}_{L^{\infty}(\Omega)}\diff s\right)&\le  \rho(t,x)\\
&\le  (\sup_{x\in\Omega}\rho(0,x))\exp\left(\int_{0}^{t}\norm{\Div\,  \bv_{n}(s)}_{L^{\infty}(\Omega)}\diff s\right),
\end{aligned}
\end{equation*}
for all $t\in[0,T]$ and $x\in\Omega$.  
\end{lemma}
Using the latter lemma, if the velocity field is in $W^{1,\infty}$, the density is bounded from below by a positive constant (provided the initial condition is positive). We now focus on Equation~\eqref{eq:mass_fraction_reg}. 

\begin{proposition}\label{prop:mass_fraction_reg}
Let $\rho$ be given such that $\rho\in C(0,T;C^{2}(\overline{\Omega}))$ and $\rho\ge \underline{\rho}>0$. Then Equation~\eqref{eq:mass_fraction_reg} with Neumann boundary conditions~\eqref{eq:neumann_reg} admits a strong solution. Moreover, the mapping $\bv_{n}\mapsto c[\bv_{n}]$ takes bounded sets in the space $C(0,T;C^{2}_{0}(\overline{\Omega},\R^{3}))$ into bounded sets in the space
\begin{equation}\label{eq:set_W}
W:=\{ c\in L^{\infty}(0,T;H^{1}(\Omega))\cap L^{2}(0,T;H^{3}(\Omega))\}.  
\end{equation}
\end{proposition}
The existence of a strong solution is based on the remark that the highest order term of this equation is $-\gamma\f{b(c)}{\rho}\Delta^{2}c$. Using $b(c),\rho\ge C>0$ we obtain a fourth-order parabolic equation with smooth coefficients and with zero Neumann boundary conditions. Therefore, we can admit the global in time strong solution than can be achieved through a Galerkin scheme and we focus on the estimates~\eqref{eq:set_W}. In the proof, we need the following two lemmas

\begin{lemma}[Lemma 3.1 in~\cite{MR2384572}]
\label{lemma:PW}
Let $\Omega\in\R^{3}$ be  a bounded Lipschitz domain and let $M_{0}>0$, $K>0$. Assume that $\rho$ is a nonnegative function such that
\begin{equation*}
    0<M_{0}\le \int_{\Omega}\rho\diff x, \int_{\Omega}\rho^{a}\diff x\le K, \quad\text{ with ${a>\f{6}{5}}$}. 
\end{equation*}
Then, there exists a positive constant $C=C(M_{0},K,a)$ such that the inequality
\begin{equation*}
\left\|\bold{u}-\f{1}{|\Omega|}\int_{\Omega}\rho\bold{u}\right\|_{L^{2}(\Omega;\R^{3})}\le C\norm{\nabla\bold{u}}_{L^{2}(\Omega;\R^{3\times 3})},
\end{equation*}
holds for any $\bold{u}\in W^{1,2}(\Omega;\R^{3})$.
\end{lemma}

\begin{lemma}[Theorem 10.17 in~\cite{MR2499296}]\label{lem:Korn_Poincaré}
Let $\Omega\subset\R^{3}$ be a bounded Lipschitz domain, and let $1<p<+\infty$, $M_{0}>0$, $K>0$, $a>1$. Then there exists a postive constant $C=C(p,M_{0},K,a)$ such that the inequality
\begin{equation*}
\norm{\bold{u}}_{W^{1,p}(\Omega;\R^{3})}\le C\left(\norm{\nabla\bold{u}+\nabla^{T}\bold{u}-\f{2}{3}\Div\bold{u}\mathbb{I}}_{L^{p}(\Omega;\R^{3\times 3})}+\int_{\Omega}\rho|\bold{u}|\diff x\right),
\end{equation*}
holds for any $\bold{u}\in W^{1,p}(\Omega;\R^{3})$ and any non-negative function $\rho$ such that
\begin{equation*}
0<M_{0}\le \int_{\Omega}\rho \diff x, \quad \int_{\Omega}\rho^{a}\diff x\le K.     
\end{equation*}
\end{lemma}

\begin{proof}[Proof of Proposition~\ref{prop:mass_fraction_reg}]
We admit the existence of solutions and focus on a priori estimates. We multiply Equation~\eqref{eq:mass_fraction_reg} by $-\Delta c$. Using the boundary conditions and integrating in space yields

\begin{multline*}
\p_{t}\int_{\Omega}\f{|\nabla c|^{2}}{2}\diff x+\gamma\int_{\Omega}b(c)\left|\nabla\left(\f{\Delta c}{\rho}\right)\right|^{2}\diff x\\=\int_{\Omega}\f{1}{2}\Div(\bv_{n})|\nabla c|^{2}-\nabla\bv_{n}:\nabla c\otimes\nabla c \diff x +\int_{\Omega}b(c)\nabla\left(\f{\p\psi_{0}}{\p c}\right)\cdot\nabla\left(\f{\Delta c}{\rho}\right)\diff x-\int_{\Omega}\f{F_c}{\rho}\Delta c.    
\end{multline*}
Here, we have also used the formula~\eqref{eq:formulaDelta}. We use the $L^{\infty}$ bounds on $\bv_{n}$, $\Div(\bv_{n})$, $b(c),\rho$ the fact that $\f{F_{c}}{\rho}$ is also bounded in $L^{\infty}$, properties on $\p_{c}\psi_{0}$~\eqref{ass:Q_reg}, and obtain
\begin{equation*}
\p_{t}\int_{\Omega}\f{|\nabla c|^{2}}{2}\diff x+\gamma\int_{\Omega}b(c)\left|\nabla\left(\f{\Delta c}{\rho}\right)\right|^{2}\diff x\\ \le C\int_{\Omega}|\nabla c|^{2}\diff x+ C\int_{\Omega}\iffalse|\nabla c|\fi\left|\nabla\f{\Delta c}{\rho}\right|\diff x+ C\int_{\Omega}|\Delta c|.
\end{equation*}
We want to control the last term on the right-hand side. We use Lemma~\ref{lemma:PW} with $\bold{u}=\f{\Delta c}{\rho}(1,0,0)^{T}$ and obtain, together with Neumann boundary conditions on $c$, 
\begin{equation}\label{eq:regularity_laplace_c}
\left\|\f{\Delta c}{\rho}\right\|_{L^{2}(\Omega)}\le C\left\|\nabla\left(\f{\Delta c}{\rho}\right)\right\|_{L^{2}(\Omega;\R^{3})}.     
\end{equation}
Then, writing $\Delta c=\rho\f{\Delta c}{\rho}$ and using the $L^{\infty}$ bound on $\rho$,
\begin{equation*}
\int_{\Omega}|\Delta c|\le C\left\|\nabla\left(\f{\Delta c}{\rho}\right)\right\|_{L^{2}(\Omega;\R^{3})}. 
\end{equation*}
Finally, using Young's inequality and Gronwall's lemma, we obtain
\begin{equation}\label{eq:energy_c_reg}
\sup_{t\in(0,T)}\int_{\Omega}|\nabla c|^{2}\diff x+\gamma\int_{0}^{T}\int_{\Omega}b(c)\left|\nabla\left(\f{\Delta c}{\rho}\right)\right|^{2}\diff x\\ \le C. 
\end{equation}

With Lemma~\ref{lemma:PW} (and integrating the equation on $\rho c$ using also the boundary conditions) we obtain the bound
\begin{equation}\label{est:c_reg}
c\in L^{\infty}(0,T;H^{1}(\Omega))\cap L^{2}(0,T;H^{3}(\Omega)). 
\end{equation}

\iffalse Then, by definition of $\mu$ and Lemma~\ref{lemma:PW} \begin{equation}\label{est:mu_reg}
\mu\quad \in L^{2}(0,T;H^{1}(\Omega)).      
\end{equation}
Finally, using the two previous bounds and the equation on $c$ we also get 
\begin{equation*}
 \p_{t} c \in L^{2}(0,T;H^{-1}(\Omega)).    
\end{equation*}
\fi
\end{proof}

Having defined $\rho$ and $c$, we now solve Equation~\eqref{eq:velocity_reg} with a fixed point argument. We define the operator 
\begin{equation*}
\mathcal{M}[\rho]:X_{n}\to X_{n}^{*},\quad  \langle\mathcal{M}[\rho]\bv,\bold{w}\rangle:=\int_{\Omega}\rho \bv\cdot \bold{w} dx,\quad \bv,\bold{w}\in X_{n}.      
\end{equation*}
This operator (~\cite{feireisl-book}) $\mathcal{M}[\rho]$ is invertible, and 
\begin{equation}
\label{propertyM1}
\norm{\mathcal{M}^{-1}[\rho]}_{\mathcal{L}(X_{n}^{*};X_{n})} \le \f{1}{\inf_{\Omega}\rho},\qquad \norm{\mathcal{M}^{-1}[\rho_{1}]-\mathcal{M}^{-1}[\rho_{2}]}_{\mathcal{L}(X_{n}^{*};X_{n})}\le C(n,\underline{\rho})\norm{\rho_{1}-\rho_{2}}_{L^{1}(\Omega)},
\end{equation}
for any $\rho_{1},\rho_{2}\ge \underline{\rho}$. Finally, Equation~\eqref{eq:velocity_reg} can be reformulated as
\begin{equation}
\label{rephrase1}
\bv_{n}(t)=\mathcal{M}^{-1}[\rho(t)]\left(\bold{m}_{0}^{*}+\int_{0}^{t}\mathcal{N}[\bv_{n}(s),\rho(s),c(s)]\diff s\right), 
\end{equation}
with 
\begin{equation*}
    \langle \bold{m}^{*}_{0},\eta \rangle=\int_{\Omega}\bold{m}_{0}\cdot\eta \diff x,  
\end{equation*}
and
\begin{multline*}
\langle \mathcal{N}[\bv_{n},\rho,c],\eta\rangle=\int_{\Omega}\left(\rho \bv_{n}\otimes\bv_{n}-\mathbb{T}-\f{\gamma}{2} |\nabla c|^{2}\mathbb{I} + \gamma\nabla c\otimes \nabla c\right):\nabla\eta+p(\rho,c)\Div(\eta)\\
-(\eps\nabla\bv_{n}\nabla\rho+\kappa(\rho,c)\bv_{n})\cdot\eta\diff x.  
\end{multline*}

To prove that Equation~\eqref{rephrase1} has a solution, we apply Schauder's fixed-point theorem in a short time interval $[0,T(n)]$. Then, we need uniform estimates to iterate the procedure.

\begin{lemma}[Schauder Fixed Point Theorem]\label{Schauder}
Let $X$ be a Hausdorff topological vector space and $S$ be a closed, bounded, convex, and non-empty subset of $X$. Then, any compact operator $A : S \to S$ has at least one fixed point.
\end{lemma}
With notation of the lemma~\ref{Schauder}, we call $A$ the operator from Equation~\eqref{rephrase1} and $S=B(\bold{u}_{0,n})$ the unit ball with center $\bold{u}_{0,n}$ in $C([0,T];X_{n})$,  $\bold{u}_{0,n}$ is defined by 
\begin{equation*}
\int_{\Omega}\rho_{0}\bold{u}_{0,n}\cdot \eta \diff x=\int_{\Omega}\bold{m}_{0}\cdot\eta \diff x,\quad \forall \eta\in X_{n}.     
\end{equation*}

More precisely, we consider 
\begin{align*}
A:\,&S\to C([0,T];X_{n}),\\
&\bold{u}\mapsto \mathcal{M}^{-1}[\rho(t)]\left(\bold{m}_{0}^{*}+\int_{0}^{t}\mathcal{N}[\bold{u}(s),\rho(s),c(s)]\diff s\right). 
\end{align*}

\begin{lemma}\label{lem:fixed_point_reg}
There exists a time $T=T(n)$ small enough such that the operator $A$ maps $S$ into itself. Moreover, the mapping is continuous.  
\end{lemma}

\begin{proof}
By definition of $A$ and $\bold{m}_{0}^{*}$, we need to prove that 
$\norm{\mathcal{M}^{-1}[\rho(t)]\int_{0}^{t}\mathcal{N}(s)ds}_{C(0,T;X_{n})}\le 1$. With properties~\eqref{propertyM1}, it is sufficient to prove that there exists a final time $T$ small enough such that 
\begin{equation*}
\left\|\int_{0}^{t}N(s)\diff s\right\|_{C(0,T;X_{n}^{*})}\le \inf_{\Omega_{T}}\rho.    
\end{equation*}
Note that the infimum of $\rho$ needs to be taken over the set $\Omega_{T}=(0,T)\times\Omega$ as $\rho$ depends on time. But, since we only consider small times, using Lemma~\ref{lem:maximum_principle} we see that this infimum is bounded by below. More precisely, for every $T_0$, there exists $C(T_0)>0$ such that for every $T\le T_0$, $\inf_{\Omega_T}\rho\ge C(T_0)$. We recall that $X_{n}\subset C^{2}(\Omega;\R^3)$ is finite-dimensional. {With the definition of the tensor $\mathbb{T}$ and the pressure $p(\rho,c)$ given by~\eqref{eq:tensorT}-\eqref{eq:def-pressure} we estimate by Hölder's inequality:}
\begin{align*}
&\int_{0}^{t}\int_{\Omega}(\rho \bold{u}\otimes\bold{u}-\mathbb{T}-\f{\gamma}{2} |\nabla c|^{2}\mathbb{I} +\gamma \nabla c\otimes \nabla c
):\nabla\eta+p(\rho,c)\Div(\eta)-(\eps\nabla\bold{u}\nabla\rho+\kappa(\rho,c)\bold{u})\cdot\eta\diff x \diff s\\
&\le { C (\sqrt{T}+T)}(\norm{\eta}_{X_{n}}+\norm{\nabla\eta}_{X_{n}})(\norm{\rho}_{L^{\infty}}\norm{\bold{u}}_{L^\infty}^2+C\norm{\nu(c)}_{L^{2}}\norm{\nabla\bold{u}}_{L^\infty}+C\norm{\nabla c}_{L^{4}}^2+\norm{\rho}_{L^{\infty}}^{a}\\&+\norm{\rho}_{L^{\infty}}\norm{H(c)}_{L^{\infty}}+\eps\norm{\bold{u}}_{X_{n}}\norm{\nabla\rho}_{L^{\infty}}+\norm{\bold{u}}_{L^{\infty}}\norm{\kappa(\rho,c)}_{L^{2}}).
\end{align*}
Using assumptions of the subsection~\ref{subsec:assumptions} and Propositions~\ref{prop:density_reg}-\ref{prop:mass_fraction_reg}, we prove that all the quantities on the right-hand side are bounded, {with a bound that may depend on $n$}, except $\norm{\nabla c}_{L^4}$ which needs an argument. Note that from~\eqref{eq:set_W}, we deduce $\nabla c$ is bounded in $L^{2}(0,T;H^{2}(\Omega))\cap L^{\infty}(0,T;L^{2}(\Omega))$ (by a constant which depends on $\rho$, and also on $\norm{\bold{u}}_{L^\infty}, \norm{\nabla\bold{u}}_{L^\infty}$). {By Sobolev embedding with $d=3$, $\nabla c$ is bounded in $L^{2}(0,T;L^{\infty}(\Omega))\cap L^{\infty}(0,T;L^{2}(\Omega))$. By Hölder inequality (or interpolation), we obtain an $L^{4}(0,T;L^{4}(\Omega))$ bound: $\|\nabla c\|_{L^{4}L^{4}}^4 \le \|\nabla c\|_{L^{\infty}L^2}^2 \|\nabla c\|_{L^2 L^\infty}^2$}. With the previous estimates, and for $T$ small enough, we obtain the result. 
\end{proof}

\begin{lemma}
The image of $S$ under $A$ is in fact a compact subset of $S$. Therefore, $A$ admits a fixed point. 
\end{lemma}

\begin{proof}
We want to apply the Arzelà-Ascoli theorem to deduce the relative compactness of $A(S)$. From the previous computation, and using the fact that $X_n$ is finite-dimensional, we can prove that $A(S)$ is pointwise relatively compact. It remains to prove its equicontinuity.  We want to estimate for $t'\le t$ the $X_{n}$ norm of $\mathcal{M}^{-1}[\rho(t)]\left(\bold{m}_{0}^{*}+\int_{0}^{t}\mathcal{N}[\bold{u}(s),\rho(s),c(s)]\diff s\right)-\mathcal{M}^{-1}[\rho(t')]\left(\bold{m}_{0}^{*}+\int_{0}^{t'}\mathcal{N}[\bold{u}(s),\rho(s),c(s)]\diff s\right)$. 
For simplicity, we write $\mathcal{N}(s):=\mathcal{N}[\bold{u}(s),\rho(s),c(s)]$, and rewrite the previous difference as
\begin{equation*}
\mathcal{M}^{-1}[\rho(t)-\rho(t')]\left(m_{0}^{*}+\int_{0}^{t}\mathcal{N}(s)\diff s\right)+\mathcal{M}^{-1}[\rho(t')]\left(m_{0}^{*}+\int_{t'}^{t}\mathcal{N}(s)\diff s\right).
\end{equation*}
For the first term, we use~\eqref{propertyM1} and the Hölder continuity of $\rho$ given by Proposition~\ref{prop:density_reg}. For the second term, we repeat the computations in the proof of Lemma~\ref{lem:fixed_point_reg}.
This ends the result. 
\end{proof}

We have the existence of a small interval $[0,T(n)]$. To iterate the procedure in order  to prove that $T(n)=T$, it remains to find a bound on $\bv_{n}$ independent of $T(n)$. 

\begin{lemma}
$\bv_{n}$ is bounded in $X_{n}$ independently of $T(n)$. 
\end{lemma}
\begin{proof}
Note that we do not ask for a bound independent of $n$ but only of $T(n)$ since we use in the proof the fact that $X_{n}$ is finite-dimensional. The proof uses the energy structure of the equation. We differentiate Equation~\eqref{eq:velocity_reg} in time and take $\eta=\bv_{n}$ as a test function. 
This yields
\begin{multline}\label{eq:weakenergy1}
\f{d}{dt}\int_{\Omega}\rho\f{|\bv_{n}|^{2}}{2}\diff x+\f{1}{2}\int_{\Omega}\left(\p_{t}\rho+\Div(\rho\bv_{n})\right)|\bv_{n}|^{2}\diff x-\int_{\Omega}p(\rho,c)\Div(\bv_{n})\diff x -\f{\eps}{2}\int_{\Omega}\Delta\rho|\bv_{n}|^{2}\diff x\\+\int_{\Omega}\mathbb{T}:\nabla\bv_{n}\diff x+ \gamma\int_{\Omega}(\f{1}{2} |\nabla c|^{2}\mathbb{I} -( \nabla c\otimes \nabla c))
:\nabla\bv_{n}\diff x + \int_{\Omega}\kappa(\rho,c)|\bv_{n}|^{2}\diff x=0.
\end{multline}
Here we used
\begin{align*}
&\int_{\Omega}\p_{t}(\rho\bv_{n})\cdot\bv_{n}=\f{1}{2}\f{d}{dt}\int_{\Omega}\rho|\bv_{n}|^{2}\diff x + \f{1}{2}\int_{\Omega}\p_{t}\rho|\bv_n|^{2}\diff x,\\
&\int_{\Omega}\Div(\rho\bv_n \otimes \bv_n)\cdot \bv_n \diff x= \f{1}{2}\int_{\Omega}\Div(\rho \bv_n)|\bv_n|^2 \diff x,\\
&\eps\int_{\Omega}(\nabla\bv_{n}\nabla\rho)\cdot\bv_{n}\diff x=-\f{\eps}{2}\int_{\Omega}\Delta\rho|\bv_{n}|^{2}\diff x.   
\end{align*}
With~\eqref{eq:density_reg}, we see that~\eqref{eq:weakenergy1} reads
\begin{multline}\label{eq:weakenergy2}
\f{d}{dt}\int_{\Omega}\rho\f{|\bv_{n}|^{2}}{2}\diff x-\int_{\Omega}p(\rho,c)\Div(\bv_{n})\diff x+\int_{\Omega}\mathbb{T}:\nabla\bv_{n}\diff x\\+ \gamma\int_{\Omega}(\f{1}{2} |\nabla c|^{2}\mathbb{I} -( \nabla c\otimes \nabla c))
:\nabla\bv_{n}\diff x + \int_{\Omega}\kappa(\rho,c)|\bv_{n}|^{2}\diff x=0.
\end{multline}
Now as in~\eqref{eq:apriori2}, we obtain with the artificial viscosity
\begin{multline*}
\p_{t}(\rho\psi_{0})+\Div(\rho\psi_{0}\bv_{n})+p\Div(\bv_{n})-\psi_{0}\eps\Delta\rho-\eps\rho\f{\p\psi_{0}}{\p\rho}\Delta\rho=\Div(b(c)\nabla\mu)\mu\\+\Div(\p_{t}c\nabla c)-\p_{t}\left(\f{|\nabla c|^{2}}{2}\right)+\gamma\Delta c\bv_{n}\cdot\nabla c +\mu F_c.    
\end{multline*}
Integrating this equation in space, and summing with~\eqref{eq:weakenergy2}, we obtain 
\begin{multline}\label{eq:weakenergy3}
\f{d}{dt}\int_{\Omega}\rho\left(\f{|\bv_{n}|^{2}}{2}+\psi_{0}\right)+\gamma\f{|\nabla c|^{2}}{2}\diff x+\eps\int_{\Omega}\nabla\left(\psi_{0}+\rho\f{\p\psi_{0}}{\p\rho}\right)\cdot\nabla\rho\diff x\\+\int_{\Omega}\mathbb{T}:\nabla\bv_{n}\diff x+ \int_{\Omega}b(c)|\nabla\mu|^{2}\diff x + \int_{\Omega}\kappa(\rho,c)|\bv_{n}|^{2}\diff x=\int_{\Omega}\mu F_c\diff x.
\end{multline}
By definition of $\psi_{0}$, we obtain
\begin{multline*}
\eps\int_{\Omega}\nabla\left(\psi_{0}+\rho\f{\p\psi_{0}}{\p\rho}\right)\cdot\nabla\rho\diff x = \eps\int_{\Omega}\left(((a-1)+(a-1)^{2})\rho^{a-2}+\f{H(c)}{\rho}\right)|\nabla\rho|^{2}\diff x \\+ \eps\int_{\Omega}\left(H'(c) (\log(\rho)+1)+Q'(c)\right)\nabla c\cdot\nabla\rho\diff x.
\end{multline*}
Therefore, the energy reads
\begin{multline}\label{eq:weakenergy4}
\f{d}{dt}\int_{\Omega}\rho\left(\f{|\bv_{n}|^{2}}{2}+\psi_{0}\right)+\gamma\f{|\nabla c|^{2}}{2}\diff x+\eps\int_{\Omega}\left(((a-1)+(a-1)^{2})\rho^{a-2}+\f{H(c)}{\rho}\right)|\nabla\rho|^{2}\diff x\\+\int_{\Omega}\mathbb{T}:\nabla\bv_{n}\diff x+ \int_{\Omega}b(c)|\nabla\mu|^{2}\diff x + \int_{\Omega}\kappa(\rho,c)|\bv_{n}|^{2}\diff x=\int_{\Omega}\mu F_c\diff x\\- \eps\int_{\Omega}\left(H'(c) (\log(\rho)+1)+Q'(c)\right)\nabla c\cdot\nabla\rho\diff x.
\end{multline}
We need to prove that the right-hand side can be controlled in term of the left-hand side to obtain estimates. For the first term on the right-hand side, we treat it as in the proof of Proposition~\ref{prop:energy_smooth}. For the second term, we know 
by assumption on $H$ and $Q$, and the fact that  $(\log(\rho)+1)^{2}$ is bounded  by a constant times $\f{1}{\rho}+(a-1)\rho^{a-2}$ that it can be bounded in terms of the left-hand side. Note that we used the hypothesis $|Q'(c)|\le C.$ This is based on the fact that $Q$ is in fact $Q_{\eps_{Q}}$ so that we have $|Q'(c)|\le C(\f{1}{\eps_{Q}})$ with a constant that blows up when $\eps_{Q}$ is sent to 0. As we intend to send $\eps_{Q} \to 0$ in the next step, it is important to notice that we can still manage to have this energy inequality since in fact the term $\eps\int_{\Omega}Q'(c)\nabla c\cdot \nabla\rho\diff x$ can be estimated by $\f{\eps}{4}\int_{\Omega}\left(((a-1)+(a-1)^{2})\rho^{a-2}+\f{H(c)}{\rho}\right)|\nabla\rho|^{2}\diff x$ and $\int_{\Omega}\eps C(\frac{1}{\eps_{Q}})|\nabla c|^{2}\diff x$. Since $\eps$ will be sent to $0$ before $\eps_{Q}$, the energy inequality will still hold independently of $\eps_{Q}$ in the limit $\eps\to 0$. With Gronwall's lemma, and properties of the tensor $\mathbb{T}$, we deduce that $\bv_{n}$ is bounded in $L^{2}(0,T(n);H^{1}(\Omega;\R^{3}))$ independently of $T(n)$. {Also, the previous bounds do not depend on $n$}. Since all the norms are equivalent, it is also bounded in $L^{1}(0,T(n);W^{1,\infty}(\Omega,\R^{3}))$. Therefore, we can apply the maximum principle stated in Lemma~\ref{lem:maximum_principle}, and obtain that the density $\rho$ is bounded from below by a constant independent of $T(n)$. Then, using once again the energy inequality, we obtain that $\bv_{n}$ is bounded uniformly in time in $L^{2}(\Omega;\R^{3})$. This procedure can be repeated for every final time $T$.
\end{proof}

Finally, we are left with the following proposition
\begin{proposition}
For any fixed $n$ and $T$, there exists a solution ($\rho, c, \bv_n$) defined on $(0,T)$ (with appropriate regularity) to~\eqref{eq:density_reg}-\eqref{eq:mass_fraction_reg}-\eqref{eq:velocity_reg} subject to boundary conditions~\eqref{eq:neumann_reg} and initial conditions~\eqref{eq:density_reg}. Moreover, this solution satisfies the energy dissipation inequality
\begin{multline}\label{eq:weakenergy5}
E(t)+\eps\int_{\Omega_{t}}\left((a+a^2)\rho^{a-1}+\f{H(c)}{\rho}\right)|\nabla\rho|^{2}\diff x\diff t\\+\int_{\Omega_t}\mathbb{T}:\nabla\bv_{n}\diff x\diff t+ \int_{\Omega_t}b(c)|\nabla\mu|^{2}\diff x\diff t + \int_{\Omega_t}\kappa(\rho,c)|\bv_{n}|^{2}\diff x\diff t\le C + C E(0),
\end{multline}
where 
$$
E(t)=\int_{\Omega}\rho\left(\f{|\bv_{n}|^{2}}{2}+\psi_{0}\right)+\gamma\f{|\nabla c|^{2}}{2}\diff x,
$$
and with a constant $C=C\left(1,\f{\eps}{\eps_{Q}}\right)$ {that does not depend on $n$}. 
\end{proposition}

Now, we need to find estimates, independent of $n$, to pass to the limit $n\to\infty$. Since $\rho$ and $c$ depend on $n$, we write $\rho_{n}$ and $c_{n}$ from now on. 

\begin{proposition}\label{prop:estim_uniform_n}
We have the following estimates uniformly in $n$ and $\eps$:
\begin{enumerate}[label=(A\arabic*)]
\item \label{item_estim12} $\{\rho_{n}\psi_{0}\}$ in $L^{\infty}(0,T; L^1(\Omega))$,
\item \label{item_estim22} $\{\rho_{n} \}$ in $L^{\infty}(0,T; L^a(\Omega))$,
\item \label{item_estim32} $\{\mathbb{T}:\nabla \bv_{n} \}$ in $L^{1}(0,T; L^1(\Omega))$,
\item \label{item_estim42} $\{\sqrt{\rho_{n}}\bv_{n} \}$ in $L^{\infty}(0,T; L^2(\Omega;\R^{3}))$,
\item \label{item_estim52} $\{\sqrt{b(c_n)}\nabla\mu_n\}$ in $L^{2}(0,T; L^2(\Omega;\R^{3}))$,
\item \label{item_estim62} $\{\bv_{n} \}$ in $L^{2}(0,T; H^{1}_{0}(\Omega;\R^{3}))$,
\item \label{item_estim72} $\{\sqrt{\eps}\nabla\rho_{n} \}$ in $L^{2}(0,T; L^2(\Omega))$,
\item \label{item_estim82} $\{c_{n}\}$ in $L^{\infty}(0,T; H^1(\Omega))$,
\item \label{item_estim92} $\{\rho_{n}\p_{c}\psi_{0}\}$ in $L^{\infty}(0,T; L^r(\Omega))$ for $r<\f{6a}{6+a}$, 
\item \label{item_estim102} $\{\mu_{n}\}$ in $L^{2}(0,T; H^1(\Omega))$,
\item \label{item_estim112} $\{\rho_{n}\mu_{n}\}$ in $L^{2}(0,T; L^{6a/(6+a)})$,
\item \label{item_estim122} $\{c_{n}\}$ in $L^{2}(0,T; W^{2,r}(\Omega))\cap L^{2+\nu}(0,T;W^{1,2+\nu})$ for some $\nu>0$,
\item \label{item_estim132} $\{\rho_{n}c_{n}\}$ in $L^{\infty}(0,T; L^{\f{6a}{6+a}}(\Omega))$,
\item \label{item_estim142} $\{\rho_{n}c_{n}\bv_{n}\}$ in $L^{2}(0,T; L^{\f{6a}{3+4a}}(\Omega))$,
\item \label{item_estim152} $\{p(\rho_{n},c_{n})\}$ in $L^{1+\tilde{\nu}}((0,T)\times\Omega))$ for some $\tilde{\nu}>0.$

\end{enumerate}
\end{proposition}

\begin{proof}
Estimates~\ref{item_estim12}-\ref{item_estim22}-\ref{item_estim32}-\ref{item_estim42}-\ref{item_estim52} follow immediately from the energy equality~\eqref{eq:weakenergy5}. Estimate~\ref{item_estim62} is the result of Lemma~\ref{lem:Korn_Poincaré} and estimates~\ref{item_estim22}-\ref{item_estim32}-\ref{item_estim42}. To obtain estimate~\ref{item_estim72}, we multiply Equation~\eqref{eq:density_reg} by $\rho_{n}$, and using integration by parts, we obtain 
\begin{equation*}
2\eps\int_{0}^{T}\int_{\Omega}|\nabla \rho_{n}|^{2}\diff x\diff t\le  \norm{\rho_{0}}_{L^{2}(\Omega)}^{2}+\norm{\rho_{n}}_{L^{\infty}(0,T;L^{2}(\Omega))}^{2}+\norm{\rho_{n}}_{L^{2}(0,T;L^{4}(\Omega))}^{2}\norm{\nabla\bv_{n}}_{L^{2}(0,T;L^{2}(\Omega)^d)}.    
\end{equation*}
Using \ref{item_estim22} and~\ref{item_estim62}, we deduce~\ref{item_estim72}. To prove Estimate~\ref{item_estim82}, we first notice that equality~\eqref{eq:weakenergy5} provides the uniform bound on $\{\nabla c_{n}\}$ in $L^{2}(0,T;L^{2}(\Omega))$. To conclude with Lemma~\ref{lemma:PW}, we need to bound $\int_{\Omega}\rho_{n}c_{n}$. Combining Equations~\eqref{eq:density_reg}-\eqref{eq:mass_fraction_reg}, we obtain 
\begin{equation*}
\p_{t}(\rho_{n}c_{n})+\Div(\rho_{n}c_{n}\bv_{n})=-\eps c\Delta\rho+\Div(b(c)\nabla\mu)+F_c.     
\end{equation*}
Integrating in space, using the boundary conditions, and Estimate~\ref{item_estim72}, the $L^{2}$ bound on $\{\nabla c_{n}\}$, assumption~\ref{eq:assumpt-source} yields $\{\int_{\Omega}\rho_{n} c_{n}\}$ is in $L^{\infty}(0,T)$. We deduce Estimate~\ref{item_estim82}. Estimate~\ref{item_estim92}  follows from the definition of $\psi_{0}$ and Estimate~\ref{item_estim12}. Estimate~\ref{item_estim102} follows from Estimates~\ref{item_estim52}-\ref{item_estim92} and Lemma~\ref{lemma:PW}. Estimate~\ref{item_estim112} follows from Estimates~\ref{item_estim22}-\ref{item_estim102}. Estimate~\ref{item_estim122} is a consequence of Equation~\eqref{eq:main3}, the previous estimates and interpolation. The two next estimates are a consequence of the other estimates and Sobolev embeddings. Finally, the last estimate on the pressure can be adapted from~\cite[Subsection 2.5]{Miranville-2019-compressible}. This estimate is useful when we obtain the convergence a.e. of $\rho_n$ and $c_n$ so we can obtain strong convergence of $p(\rho_n , c_n)$ in $L^{1}$ by Vitali's convergence theorem.  
\end{proof}

From~\cite{feireisl-book}, we also obtain the following Proposition
\begin{proposition}
There exists $r>1$ and $p>2$ such that
\begin{align*}
& \p_{t}\rho_{n},\Delta\rho_{n} \quad\text{are bounded in $L^{r}((0,T)\times\Omega)$},\\
&\nabla\rho_{n} \quad \text {is bounded in $L^{p}(0,T;L^{2}(\Omega,\R^{3}))$},
\end{align*}
independently of $n$ (but not independently of $\eps$). 
\end{proposition}

With all the previous bound, we can pass to the limit when $n\to\infty$ and obtain the different equation and energy estimates in a weak formulation. Since the passage to the limit $n\to\infty$ is simpler than the next passage $\eps\to 0$, we only detail the latter. Indeed, as $n\to\infty$ we can obtain easily strong convergence of $\rho$ which helps a lot in the different limits. So we assume that we can pass to the limit and that the bounds obtained in Proposition~\ref{prop:estim_uniform_n} still hold independently of $\eps$.  It remains now to send $\eps$ to 0. 

We recall the equations that we want to pass to the limit into:
\begin{align}
\p_{t}\rho_{\eps} +\Div(\rho_{\eps}\bv_{\eps})=\eps\Delta\rho_{\eps},\label{eq:epsilon_eq1}\\
\p_{t}(\rho_{\eps}c_{\eps})+\Div(\rho_{\eps}c_{\eps}\bv_{\eps})=-\eps c_\eps \Delta\rho_{\eps}+\Div(b(c_{\eps})\nabla\mu_{\eps})+F_{c_\eps},\label{eq:epsilon_eq2}
\end{align}
and for every $\eta$ (sufficiently regular)
\begin{multline}\label{eq:epsilon_eq3}
\int_{\Omega}\rho_{\eps}\bv_{\eps}(t)\cdot\eta\diff x- \int_{\Omega}\bold{m}_{0}\cdot\eta\diff x- \int_{0}^{t}\int_{\Omega}\rho_{\eps} \bv_{\eps}\otimes\bv_{\eps}:\nabla_{x}\eta\diff x\diff s- \int_{0}^{t}\int_{\Omega}p(\rho_\eps,c_\eps)\Div(\eta)\diff x\diff s \\
+ \eps\int_{0}^{t}\int_{\Omega}(\nabla\bv_{\eps}\nabla\rho_{\eps})\cdot\eta\diff x\diff s+\int_{0}^{t}\int_{\Omega}\mathbb{T_\eps}:\nabla\eta\diff x\diff s+ \gamma\int_{0}^{t}\int_{\Omega}(\f{1}{2} |\nabla c_\eps|^{2}\mathbb{I} -( \nabla c_\eps \otimes \nabla c_\eps))
:\nabla\eta\diff x\diff s \\
+ \int_{\Omega}\int_{0}^{t}\kappa(c_\eps)\bv_{\eps}\cdot\eta\diff x\diff s=0.
\end{multline}

Using Proposition~\ref{prop:estim_uniform_n}, which yields uniform estimates in $\eps$, we pass to the limit in the previous equations. The difficult terms are the one involving nonlinear combinations. Indeed, it is not clear that we can obtain strong convergence of $\rho_{\eps}$ as we have no estimates on higher order derivatives. We use the following lemma, see~\cite{MR1637634}.

\begin{lemma}
Let $g_n$, $h_n$ converge weakly to $g$, $h$ respectively in $L^{p_1}(0,T;L^{p_2}(\Omega))$, $L^{q_1}(0,T;L^{q_2}(\Omega))$ where $1\le p_1 , p_2\le +\infty$ and 
$$
\f{1}{p_1}+\f{1}{q_1}=\f{1}{p_2}+\f{1}{q_2}=1. 
$$
We assume in addition that
\begin{equation}\label{ass:lemma_lions}
\f{\p g_n}{\p t}\quad \text{is bounded in $L^{1}(0,T;W^{-m,1}(\Omega))$ for some $m\ge 0$ independent of $n$}, 
\end{equation}
and 
\begin{equation}\label{ass2:lemma_lions}
\norm{h_{n}-h_{n}(t,\cdot+\xi)}_{L^{q_1}(0,T;L^{q_2}(\Omega))}\to 0\quad \text{as $|\xi|\to 0$, uniformly in $n$}.    
\end{equation}
Then, $g_n h_n$ converges to $gh$ in the sense of distributions. 
\end{lemma}

\begin{remark}
This lemma admits many variants, and it is possible to identify the weak limit of the products with lower regularity, we refer for instance to~\cite{MR3466213}. 
\end{remark}

We want to apply the previous lemma to the terms $\rho_{\eps}\bv_{\eps}$, $\rho_{\eps} c_{\eps}$, $\rho_{\eps}\mu_{\eps}$, $\rho_{\eps}c^{2}_\eps$, $\rho_{\eps}\bv_{\eps}$, $\rho_\eps \bv_\eps \otimes \bv_\eps$, $\rho_\eps \bv_\eps c_\eps$. We admit that $\f{\p \rho_\eps}{\p t}$, $\f{\p \rho_{\eps}\bv_{\eps}}{\p t}$ and $\f{\p \rho_{\eps}c_{\eps}}{\p t}$ satisfy~\eqref{ass:lemma_lions} by using Proposition~\ref{prop:estim_uniform_n} and Equations~\eqref{eq:epsilon_eq1}-\eqref{eq:epsilon_eq2}-\eqref{eq:epsilon_eq3}. The compactness in space required in~\eqref{ass2:lemma_lions} also uses Proposition~\ref{prop:estim_uniform_n}.  We refer also to~\cite[Subsection 3.1]{Miranville-2019-compressible} for similar results. The terms $\eps c_\eps \Delta\rho_\eps$ and $\eps\int_{0}^{t}\int_{\Omega}(\nabla\bv_{\eps}\nabla\rho)\cdot\eta\diff x\diff s$ converge to 0 (the first  one in the distributional sense) thanks to estimates~\ref{item_estim72}-\ref{item_estim82}.

It remains to pass to the limit in (i.e identifying the weak limits)
\begin{align*}
& p(\rho_\eps,c_\eps), \quad \f{1}{2}|\nabla c_{\eps}|^{2}, \quad \nabla c_{\eps}\otimes \nabla c_{\eps},\\
&b(c_{\eps})\nabla\mu_{\eps},\quad  F_{c_\eps}(\rho_{\eps},c_{\eps}), \quad \rho_{\eps}\p_{c}\psi_{0}.
\end{align*}
 The convergence of the last term is used to identify $\rho\mu$. To prove the previous convergences, we need to prove strong compactness in $L^{2}$ of $c_{\eps},\nabla c_{\eps}$ and convergence a.e. of $\rho_{\eps}$ to use Vitali's convergence theorem. But they follow from the arguments in~\cite{abels_diffuse_2008} and~\cite[Section 3.3 and 3.4]{Miranville-2019-compressible}.{We obtain
 \begin{lemma}
 Up to a subsequence (not relabeled),
 \begin{align}
 &\rho_\eps\to \rho \text{ a.e.}\label{conv1_eps}\\
 & c_\eps \to c \text{ a.e. and strongly in $L^{2}(0,T;L^{2}(\Omega))$}\label{conv2_eps}\\
  & \nabla c_\eps \to \nabla c \text{ a.e. and strongly in $L^{2}(0,T;L^{2}(\Omega))$} \label{conv3_eps}
 \end{align}
 \end{lemma}
Altogether, we can pass to the limit in every term of the equations:
\begin{itemize}
    \item $p(\rho_\eps,c_\eps)$: we use~\eqref{conv1_eps}, \eqref{conv2_eps} and \ref{item_estim152}
    \item $\f{1}{2}|\nabla c_{\eps}|^{2}$ and $\nabla c_\eps \otimes \nabla c_\eps$: \eqref{conv3_eps}
    \item $b(c_{\eps})\nabla \mu_\eps$: \eqref{conv2_eps} and \ref{item_estim102}
    \item $F_{c_\eps}(\rho_\eps, c_\eps)$: \eqref{conv1_eps}, \eqref{conv2_eps} and \eqref{eq:assumpt-source}
    \item $\rho_\eps \p_{c}\psi_{0}$: \eqref{conv1_eps}, \eqref{conv2_eps} and \ref{item_estim92}.
\end{itemize}
}

This concludes the argument.

\subsection{Sending $\eps_{Q}\to 0$}

The last step in our proof is to let $\eps_{Q}$ vanishes and recover the existence of weak solutions for the double well potential $Q(c)=\f{1}{4}c^{2} (1-c)^{2}.$ Since we have the energy estimates from before, {that still hold by properties of the weak convergence}, the work is essentially the same but we have to be careful about two points. The first one is to indeed have an energy estimate independent of $\eps_{Q}$. We discussed this point after Equation~\eqref{eq:weakenergy4} and, hence, we do not repeat it here. The second point are the estimates obtained in Proposition~\ref{prop:estim_uniform_n}. However,  the estimates are essentially the same, except for estimate~\ref{item_estim92} {(that is the only one containing $Q$)} which becomes 
\begin{equation}\label{new_item_estim92}
\{\rho\p_{c}\psi_{0}\} \text{ in $L^{\infty}(0,T;L^{\f{2a}{a+2}}(\Omega))$}.   \end{equation}
This can be proved knowing that, when $\eps_{Q}\approx 0$, we have that for $c$ large $\rho Q_{\eps_{Q}}'(c)\approx\rho c^{3}$, and we use estimates~\ref{item_estim22}-\ref{item_estim82}. Altogether, the reasoning to pass to the limit is the same and we conclude. 

\section{Numerical scheme for the G-NSCH model} \label{sec:scheme}

We propose a numerical scheme for the G-NSCH model~\eqref{eq:main1}--\eqref{eq:main4} subjected to periodic boundary conditions.

We combine ideas from the numerical scheme for the variant of the compressible NSCH system in~\cite{Quaolin-2020-compressible} and fast structure-preserving scheme for degenerate parabolic equations~\cite{Fukeng-2021-bounds, huang-SAV-fourth}. Namely, we adapt the relaxation~\cite{shi-1995-relaxation} of the Navier-stokes part as used in~\cite{Quaolin-2020-compressible}.  
The part of the scheme for the Cahn-Hilliard part of the system is designed using the GSAV method. More precisely, a variant used for degenerate parabolic models that preserves the physical bounds of the solution~\cite{Fukeng-2021-bounds,huang-SAV-fourth}. 

Indeed, we expect that the volume fraction $c$ remains within the physically (or biologically) relevant bounds $c \in (0,1)$. Thus, following~\cite{Fukeng-2021-bounds, huang-SAV-fourth}, we construct the invertible mapping $T: \mathbb{R} \to (0,1) $, with $c = T(v)$, transforming Equations~\eqref{eq:main2}--\eqref{eq:main3} into 
\begin{equation}
\begin{aligned}
\rho\left(\p_t v + (\vecu \cdot \nabla) v \right) &=  \frac{1}{T'(v)}\left( \Div(b(c) \nabla \mu ) + F_c\right),\\
\rho \mu &= -\gamma T'(v) \Delta v  -\gamma T''(v)\abs{\nabla v}^2 + \rho \frac{\p \psi_0}{\p c}.
\end{aligned}
\label{eq:transform1-bound}
\end{equation}

Following~\cite{Fukeng-2021-bounds} and~\cite{huang-SAV-fourth}, we can choose 
\[
T(v) = \frac{1}{2} \tanh(v) + \frac{1}{2},\text{ or } T(v) = \frac{1}{1+\exp(-v)},
\]
thus, preserving the bounds $c \in (0,1)$.

The SAV method allows to solve efficiently (and also linearly) the nonlinear Cahn-Hilliard part while preserving the dissipation of a modified energy.
In the following, we assume that it exists a positive constant $\underline C$ such that the energy associated with the Cahn-Hilliard part, \ie 
\[
E[t](\rho,c) = \int_\Omega \f\gamma2 \abs{\nabla c}^2 + \rho \psi_0( \rho, c) = E_0[t] + E_1[t] ,
\]
with $E_1$ the nonlinear part of the energy, and $E_0$ the linear part,
is bounded from below, \ie $E_1 + \underline C \ge 1$ 

We define 
\[
r(t) = {E(t) +C_0 } , \quad \text{with}\quad C_0 = 2\underline C + \norm{E(\rho^0,c^0)}_{L^\infty(\Omega)},
\]
and apply the SAV method. System~\eqref{eq:transform1-bound} becomes
\begin{equation}
\begin{aligned}
\rho\left(\p_t v + (\vecu \cdot \nabla) v \right) &=  \frac{1}{T'(v)}\left( \Div(b(c) \nabla \mu ) + F_c\right),\\
\rho \mu &= -\gamma T'(v) \Delta v  -\gamma T''(v)\abs{\nabla v}^2 + \rho \frac{\p \psi_0}{\p c},\\
\frac{\dd r}{\dd t} &= - \frac{r(t)}{E[t] +C_0} \int_\Omega b(c) \abs{\nabla \mu}^2-\mu F_{c} \,\dd x,
\end{aligned}
\end{equation}
One can easily see that the previous modifications do not change our system at the continuous level.

\subsection{One-dimensional scheme}
We consider our problem in a one-dimensional domain $\Omega = (0,L)$. Even though $\vecu$ is now a scalar, we still denote it in bold font to not make the confusion with $v$ from the transformation $c = T(v)$. 
As mentioned previously, we relax the Navier-Stokes part of our system. Namely, we introduce a relaxation parameter $\iota \ge 0$ and write $U = (\rho, \rho \vecu)$. We rewrite Equation~\eqref{eq:main4} as
\begin{equation}
    \begin{cases}
        \partial_t U + \partial_x V = G(U),\\
        \partial_t V + A \partial_x U = -\frac{1}{\iota}(V - F(U)),
    \end{cases}
\end{equation}
in which $G(U) = (0,-\kappa \vecu)$,$F(U) = (\rho \vecu, \rho \vecu^2 + p - \left(\frac{4}{3}\nu(c) + \eta(c)\right) \p_x \vecu   +\frac{\gamma}{2} \abs{\p_x c}^2)  $ and ${A = \text{diag}(a_1,a_2)}$ satisfying Liu's subcharacteristic condition
\[
A \ge F'(U), \quad \forall U.
\]
In what follows, and following~\cite{Quaolin-2020-compressible}, we use 
\[
a_1 = a_2 = \max\left\{\sup\left(\vecu + \sqrt{\partial_\rho p}\right)^2, \sup\left(\vecu - \sqrt{\partial_\rho p}\right)^2 \right\}.
\]

We discretize the domain using a set of $N_x$ nodes located at the center of control volumes of size $\Delta x$ such that $\Omega = \bigcup_{j=0,\dots,N_x-1} [x_{j-\frac{1}{2}}, x_{j+\frac{1}{2}}]$.

Our scheme follows the discrete set of equations 
\begin{align}
    U^*_j &= U^n_j, \label{eq:discrete1D1}\\
    V^*_j &= V^n_j - \frac{\Delta t}{\iota}\left(V^*_j - F(U^*_j) \right),\\
    U^{n+1}_j &= U^*_j - \frac{\Delta t}{\Delta x}\left(V^*_{j+\frac{1}{2}} - V^*_{j-\frac{1}{2}}  \right) + \Delta t G(U^{n+1}_j),\label{eq:discrete1D3}\\
    V^{n+1}_j &= V^*_j - \frac{\Delta t}{\Delta x}A \left(U^*_{j+\frac{1}{2}} - U^*_{j-\frac{1}{2}}  \right), \label{eq:discrete1D4}\\
    \rho^{n+1}_j  T'(v^n_j)\big( \frac{\overline{v}^{n+1}_j - v^n_j}{\Delta t} &+ \vecu^{n+1}_j\cdot (\nabla \overline{v}^{n+1})_j \big) = g(c^n,\mu^{n+1}, \rho^{n+1})_j,\label{eq:discrete1D5}\\
    g(c^n,\mu^{n+1}, \rho^{n+1})_j &= \left( \f{1}{\Delta x} \left((b(c^n)\nabla \mu^{n+1})_{j+\frac{1}{2}} -(b(c^n)\nabla \mu^{n+1})_{j-\frac{1}{2}}  \right) \right) + F_c(\rho^{n}_j,c^n_j),\\
    \rho^{n+1}_j \mu^{n+1}_j &= \left(-\gamma T'(v^n_j)(\Delta \overline{v}^{n+1})_j -\gamma T''(v^n_j) (\nabla v^n)_j \cdot (\nabla \overline{v}^{n+1})_j  \right)+ \rho^{n+1}_j\left(\frac{\p \psi_0}{\p c}\right)_{j}^n ,\label{eq:discrete1D7}\\
    \sum_j \dx T( \lambda \overline{v}^{n+1}_j) &= \sum_j\dx c^0 + \sum_{r=1}^{n} \dt\sum_j \dx F_c(\rho_j^r,c^r_j),\label{eq:Tdiscreteinteg} \\
    \overline{c}^{n+1} &=  T( \lambda \overline {v}^{n+1}),\\
    \frac{1}{\Delta t}\left(r^{n+1} - r^n\right) &= - \frac{r^{n+1}}{E(\overline{c}^{n+1}) + C_0}   \Delta x \sum_{j}  b(\overline{c}_j^{n+1}) \abs{(\nabla \mu^{n+1})_j}^2 \nonumber\\ & +\frac{{r^{n+1}}}{E(\overline{c}^{n+1}) + C_0}   \Delta x \sum_{j}\mu^{n+1}_j F_c(\rho_j^{n+1},\overline{c}_j^{n+1}),\label{eq:discrete1D8}\\
    \xi^{n+1} &= \frac{r^{n+1}}{E(\overline{c}^{n+1}) + C_0},\label{eq:discrete1D9}\\  
    c^{n+1}_j &= \nu^{n+1} \overline{c}^{n+1}_j, \quad \text{with}\quad \nu^{n+1} = 1-(1-\xi^{n+1})^2,\label{eq:discrete1D10}\\
    v^{n+1}_j &= \lambda \nu^{n+1}\overline{v}^{n+1}_j. \label{eq:discrete1D11}
\end{align}

\begin{remark}[Computation of interface values]
    To obtain the interface values $U^*_{j+\frac{1}{2}}, U^*_{j-\frac{1}{2}}$ and $V^*_{j+\frac{1}{2}}, V^*_{j-\frac{1}{2}}$, we use the upwind method, \ie 
    \[
        U_{j+\frac{1}{2} } = \frac{1}{2} (U_j + U_{j+1}) - \frac{\sqrt{a_1}}{2} \left(V_{j+1}-V_j \right), \;
        V_{j+\frac{1}{2} } = \frac{1}{2} (V_j + V_{j+1}) - \frac{1}{2\sqrt{a_2}} \left(U_{j+1}-U_j \right).
    \]
    
    We also mention that similarly to~\cite{Quaolin-2020-compressible}, one can implement a MUSCL scheme (see \eg~\cite{leveque2002finite}) to obtain a higher order reconstruction. The upwind method permits to rewrite Equations~\eqref{eq:discrete1D3}--\eqref{eq:discrete1D4} as
    \begin{align}
            U^{n+1}_j &= U^*_j -\frac{\dt}{2\Delta x} (V^*_{j+1}-V^*_{j-1}) + \frac{\dt}{2\Delta x} \sqrt{a} (\delta^2_x U^*_j) + \Delta tG(U_j^{n+1}), \label{eq:firstvelo-discrete-upw}\\
            V^{n+1}_j &= V^*_j -\frac{a \dt}{2\Delta x} (U^*_{j+1}-U^*_{j-1}) + \frac{\dt}{2\Delta x} \sqrt{a} (\delta^2_x V^*_j),\label{eq:secondvelo-discrete-upw}
    \end{align}
    where we used the notation $\delta^2_x U = U_{j+1} - 2 U_j + U_{j-1}$. In Equations~\eqref{eq:firstvelo-discrete-upw}--\eqref{eq:secondvelo-discrete-upw}, we emphasize that $U^\ast = U^n$ and $V^\ast = V^n - \frac{\dt}{\iota} \left(V^\ast - F(U^n)\right)$.
\end{remark}

\begin{remark}[Algorithm to compute the solution of the discrete equations' system]
Equations~\eqref{eq:discrete1D1} to~\eqref{eq:discrete1D4} are solved from Equations~\eqref{eq:firstvelo-discrete-upw}--\eqref{eq:secondvelo-discrete-upw}, hence, a solution $(U^{n+1}, V^{n+1})$ is computed just from vector computations. The coupling between Equation~\eqref{eq:discrete1D5} and Equation~\eqref{eq:discrete1D7} is also linear (nonlinear terms are taken at the previous time step to linearize the equations). We solve this coupled system using the GMRES algorithm but we emphasize that other iterative solver could work as long as they allow the matrix of the linear system to be non-symmetric.  
The coefficient $\lambda$ is computed using an iterative method. 
Then, the discrete solution $(\overline v^{n+1},\mu^{n+1})$, together with the coefficient $\lambda$, is used in Equation~\eqref{eq:discrete1D8} to find $r^{n+1}$ and, in Equation~\eqref{eq:discrete1D9}, $\xi^{n+1}$. At this point, we solve Equation~\eqref{eq:discrete1D10} and\eqref{eq:discrete1D11} from the previous steps.
\end{remark}

In the following, we use the notations,
\[
\duality{U}{V} = \Delta x \sum_{j} U_j V_j,\quad \text{and} \quad \norm{U}^2 = \duality{U}{U}.
\]
We also use $\Delta_{0,x} U := \frac{1}{2}(U_{j+1} - U_{j-1})$.

Our numerical scheme possesses the following important properties:
\begin{proposition}[Energy stability, bounds and mass preserving] \label{prop:1dschemeprops}
Assuming the CFL-like condition $\frac{\Delta t}{\Delta x}\sqrt{a_1}\le 1$ and the condition
{\begin{equation}
    \Delta t\le  C\f{C_0}{E[\overline{c}^{n}]}, 
    \label{eq:condrn}
\end{equation}}
our numerical scheme satisfies the energy dissipation-like inequality
\begin{equation}
\norm{\sqrt{a} U^{n+1}}^{\color{blue}2} + \norm{V^{n+1}}^{\color{blue}2} + r^{n+1} \le \norm{\sqrt{a} U^{n}}^{\color{blue}2}+ \norm{V^{\star}}^{\color{blue}2} + C^{n+1}r^{n},
\label{eq:dissnum}
\end{equation}
where {$r^{n+1}\ge 0$ and} 
$$
C^{n+1}=\f{1}{1+\f{\Delta t}{E(\overline{c}^{n+1})+C_{0}} \sum_{j=1}^{N_x}b(\overline{c}^{n+1}_j)|\nabla\mu^{n+1}_j|^{2} - \mu^{n+1}_jF_{c}(\rho^{n+1}_j,\overline{c}^{n+1}_j)}. 
$$
{The previous constant can be estimated only in terms of $E[\overline{c}^{n}]$ and therefore do not depend on the step $n+1$.} Furthermore, the numerical scheme preserves the physically relevant bounds of the mass fraction, \ie 
\[
0 < c^{n+1} <1.
\]
\end{proposition}

\begin{remark}
Note that the constant $C^{n+1}$ is smaller than $1$ whenever the nonnegative part of the dissipation of the energy is greater than the increase of energy induced by the source term $F_{c}$. This of course satisfied when we have $F_{c}=0$ for instance.      
\end{remark}

\begin{proof}
    We start with Equation~\eqref{eq:firstvelo-discrete-upw}, and using the definition of the function $G(U^{n+1}_j)$ as well as assuming $\kappa(c) \ge 0 $ (for $c\in \R$), after taking the square on both sides, multiplying by $\dx$ and summing over the nodes $j=0,...,N_x$, we have 
    \[
    \begin{aligned}
    \norm{U^{n+1}}^2 \le \norm{U^{n}}^2 + \left(\frac{\dt}{2\dx}\right)^2\norm{\Delta_{0,x}V^\star}^2 &+ \left(\frac{\dt \sqrt{a}}{2\dx} \right)^2\norm{\delta_x^2 U^n}^2 - \frac{\dt}{\dx}\duality{\Delta_{0,x}V^* }{U^n} \\ 
    &+ \frac{\dt\sqrt{a}}{\dx} \duality{U^n}{\delta^2_x U^n} - \frac{\sqrt{a}\dt^2}{2\dx^2}\duality{\Delta_{0,x} V^\star}{\delta^2_x U^n}.
    \end{aligned}
    \]
    Repeating the same computations for Equation~\eqref{eq:secondvelo-discrete-upw}, we have  
    \[
    \begin{aligned}
    \norm{V^{n+1}}^2 \le \norm{V^{n}}^2 + \left(\frac{a\dt}{2\dx}\right)^2\norm{\Delta_{0,x}U^n}^2 &+ \left(\frac{\dt \sqrt{a}}{2\dx} \right)^2\norm{\delta_x^2 V^\star}^2 - \frac{a\dt}{\dx}\duality{\Delta_{0,x}U^n }{V^\star} \\ 
    &+ \frac{\dt\sqrt{a}}{\dx} \duality{U^\star}{\delta^2_x V^\star} - \frac{a^\f32\dt^2}{2\dx^2}\duality{\Delta_{0,x} U^n}{\delta^2_x V^\star}.
    \end{aligned}
    \]
    At this point, the proof is similar to the proof of~\cite[Theorem 4.1]{Quaolin-2020-compressible} (these steps use the periodic boundary conditions and the summation by parts formula to cancel some terms when summing both of the previous equations together), to obtain for a constant $C>0$,
    \[
    \norm{\sqrt{a}U^{n+1}}^2 + \norm{V^{n+1}}^2 \le C\left( \norm{\sqrt{a}U^{n}}^2 + \norm{V^{*}}^2\right).
    \]
  Then, for the Cahn-Hilliard part, we easily obtain from Equation~\eqref{eq:discrete1D8}
    \[
        r^{n+1} \left(1 + \Delta t \f{\dx \sum_j   b(\overline{c}_j^{n+1}) \abs{(\nabla \mu^{n+1})_j}^2 - \mu^{n+1}_j F_c(\rho^{n+1}_j,\overline{c}^{n+1}_j)}{E[\overline{c}^{n+1}]+C_0}\right) = r^n.
    \]
    Therefore, as long as 
    \[
    E(\overline{c}^{n+1}) + C_{0} + \Delta t\left(\dx \sum_j b(\overline{c_j}^{n+1}) \abs{(\nabla \mu^{n+1})_j}^2- \mu^{n+1}_j F_c(\rho^{n+1}_j,\overline{c}^{n+1}_j)\right) \ge 0,
    \]
    so does $r^{n+1}$. 
    Assuming $\norm{F_c}_{L^\infty}< C$, it remains to control the discrete $L^1$ norm of $\mu^{n+1}$. Performing the same computations as in the proof of Proposition~\ref{prop:energy_smooth} in continuous case, it follows that
    $$
    \left|\dx \sum_j  \mu^{n+1}_j F_c(\rho^{n+1}_j,\overline{c}^{n+1}_j)\right| \le C + C E[\overline{c}^{n}]+ \f{1}{2}\dx \sum_j   b(\overline{c}^{n+1}_j) \abs{(\nabla \mu^{n+1})_j}^2.
    $$
    Of course, one first needs to prove that a discrete version of Lemma~\ref{lemma:PW} holds. At the continuous level, this theorem is proved by contradiction using Rellich's theorem. Hence, a similar proof can be obtained at the discrete level, in the spirit of the Poincaré-Wirtinger inequality, see~\cite[Lemma 3.8, Remark 3.16]{MR1804748}. Based on these evidences we use a discrete version of Lemma~\ref{lemma:PW} to conclude conclude that there exists $C$ a universal constant such that provided $r^n\ge 0$ and 
    $$
    \Delta t\le  C\f{C_0}{E[\overline{c}^{n}]}, 
    $$
     so does $r^{n+1}\ge 0$, and~\eqref{eq:dissnum} follows.

    Finally, from the definition of $\xi^{n+1}$ and $C_0$, we have 
    \[
    0 < \xi^{n+1} < \frac{r^0}{E(\overline{c}^{n+1})+C_0} \le 2.
    \]

    The bounds for the mass fraction $c$ are ensured by the transformation $T(v)$. This finishes the proof. 
\end{proof}
\begin{remark}
   We observe during numerical simulations that the condition~\eqref{eq:condrn} is obtained for reasonably small $\dt$. 
    We also note that if we do not consider any source term, \ie $F_c = 0$, the scheme satisfies the dissipation relation 
    \[
    \norm{\sqrt{a} U^{n+1}}\corr{}{^2} + \norm{V^{n+1}}\corr{}{^2}  + r^{n+1} \le C\left(\norm{\sqrt{a} U^{n}}\corr{}{^2} + \norm{V^{\star}}\corr{}{^2}\right)  + r^n,
    \]
    with the stability condition 
    \[
    \frac{\Delta t}{\Delta x}\sqrt{a_1}\le 1.
    \]
\end{remark}

\section{Numerical experiments} \label{sec:results}
In this section, we use the assumptions on the functionals stated in the "Framework for numerical simulations" paragraph in subsection~\ref{subsec:assumptions}.  Throughout this section we use the double-well logarithmic potential 
\[
    \psi_\text{mix} = \f12 \left(\alpha_1(1-c)\log(\rho(1-c)) + \alpha_2c\log(\rho c)\right) - \f\theta 2 (c-\frac 1 2 )^2 + k,
\]
with $k=100$, and $\theta = 4$ ($\alpha_1$ and $\alpha_2$ are specified later).
We also use a degenerate mobility, \ie 
\[
b(c) = c(1-c). 
\]

We start by using the one-dimensional scheme~\eqref{eq:discrete1D1}--\eqref{eq:discrete1D11} with no exchange term and friction, \ie $\kappa(\rho,c) = 0$ and $F_c(\rho,c) = 0$, and we verify that the scheme preserves all the properties stated in Proposition~\ref{prop:1dschemeprops}. We then use a non-zero exchange term and we compare the solution with same friction forces for the two phases or contrast of friction forces.  
Then, we perform two-dimensional simulations with friction forces contrast. 

Finally, we verify the spatial and temporal convergence orders of the scheme. 

%Then, we perform a more biologically relevant test case. We assume that the source term $F_c$ is pressure dependent and models the growing of a tumor in a healthy tissue (that does not grow). 

\begin{remark}[Implementation details]
    All numerical schemes are implemented using Python 3 and the Numpy and Scipy modules. The linear system for the Cahn-Hilliard part of the model is solved using the Generalized Minimal RESidual iteration (GMRES) iterative solver (function available in the \textbf{scipy.sparse.linalg} module). The tolerance on the convergence of the residual is indicated in each of the following subsections. 
To find the $\lambda$ that allows to compute the correct mass, we use the function \textbf{fsolve} of the \textbf{scipy.optimize} module which uses a modification of the Powell's conjugate direction method. 
%The code is freely available at the PLMlab (Gitlab based software hosted by hosted by Mathrice-CNRS) at the url: \url{https://plmlab.math.cnrs.fr/apoulain/nsch-compressible}.
\end{remark}

\subsection{One dimensional numerical test cases }
\paragraph{Comparison between matching and non-matching densities.}
\label{subsec:testcase1}
We start with a one-dimensional test cases to show the spatiotemporal evolution of the density, mass fraction, and velocity. We also verify numerically the properties stated in Proposition~\ref{prop:1dschemeprops}. We compare numerical results for matching and non-matching densities for the phases of the fluid. For this comparison, we set $\kappa(\rho,c) = 0$ and $F_c(\rho,c) = 0$. 

We use the computational domain $\Omega = (0,1)$ discretized in $N_x = 128$ cells. We take $T = 0.5$  (this has been chosen because the system reaches a meta-stable state by that time) and use the initial time step $\dt = 1 \times 10^{-5}$ (this time step size is adapted from the CFL-like condition stated in Proposition~\ref{prop:1dschemeprops}.

We choose the width of the diffuse interface to be $\gamma = 1/600$, the viscosity to be constant $\nu(c) = 1 \times 10^{-2}$, and $\eta =2 \times 10^{-2} $ , the relaxation parameter to be $\iota = 1 \times 10^{-5}$, and the exponent for the barotropic pressure equals to $a = 3$.

To model matching densities for the two phases of the fluid, we choose $\alpha_1 = \alpha_2 = 1$. To represent non-matching densities for the two phases, we can choose $\alpha_1 \neq \alpha_2$.  
This allows us to model a fluid for which the phase denoted by the index $1$ is denser compared to the phase indicated by the index $2$. Indeed, this can been seen on the effect of the values $\alpha_1$ and $\alpha_2$ on the potential. Taking $\alpha_1 < \alpha_2$ shifts the well corresponding to phase $1$ very close to $0$ compared to the other phase. This models the fact that the fluid $1$ is in fact more compressible and thus aggregates of pure phase $1$ appear denser.  

We choose constant initial conditions for the density and the pressure, \ie 
\[
\rho^0_j = 0.8,\quad \vecu^0_j = 0.5, \quad j=0,\dots,N_x-1. 
\]
The initial mass fraction is assumed to be a constant with a small random noise, \ie 
\[
c^0_j = \underline{c} + 0.05 r_j, \quad j=0,\dots,N_x-1,
\]
with $\underline c = 0.5$ and $r$ is a vector of random values between 0 and 1 given by the uniform distribution.  

We choose a tolerance for the convergence of the residual for the GMRES algorithm of $\text{rtol} = 10^{-10}$. 

Figure~\ref{fig:1D_comparison_solu} compares the results obtained for matching and non-matching densities for the two phases of the fluid. For the two cases, we report the evolution of the density $\rho$, the mass fraction $c$, the velocity $\vecu$ and the pressure $p$ at different times. We observe that, for both cases, after an initial regularization of the initial condition, the separation of the two phases of the fluid occurs and small aggregates appear (see first and second columns of Figure~\ref{fig:1D_comparison_solu}). Then, the coarsening of the small aggregates into larger ones occurs. We arrive at the end of the simulation to the solution depicted in the two figures on the last column of Figure~\ref{fig:1D_comparison_solu}. Hence, we can conclude that our numerical scheme catches well the spinodal decomposition of the binary fluid while it is transported to the right (since the velocity $\vecu$ is positive during these simulations). 

A difference between the two simulations is observed on the densities and pressures. Indeed, for matching densities, we observe that $\rho$ organizes such that it is equal in aggregates of each phases and drops at the interfaces between the aggregates. We also observe a drop of pressure $p$ at the interface, probably explained by capillary effects.   
For non-matching densities, as expected, there is a density difference between aggregates of phase 1 and 2. Indeed, selecting $\alpha_1 < \alpha_2$ makes the aggregates of phase 1 denser compared to aggregates of phase 2. Our explanation is that, due to the fact that attractive effects are stronger in phase 1, more mass is allowed to move inside aggregates of phase 1. However, as the pressure function accounts for the difference $\alpha_1 \neq \alpha_2$, the pressure equilibrates to  a field similar to the matching density case (\ie $p$ varies from an equilibrium value and depicts a drop at the interface between the aggregates of the different phases). 

\begin{figure}
\centering 
     \begin{subfigure}[b]{0.24\textwidth}
\centering 
\includegraphics[width=0.99\linewidth]{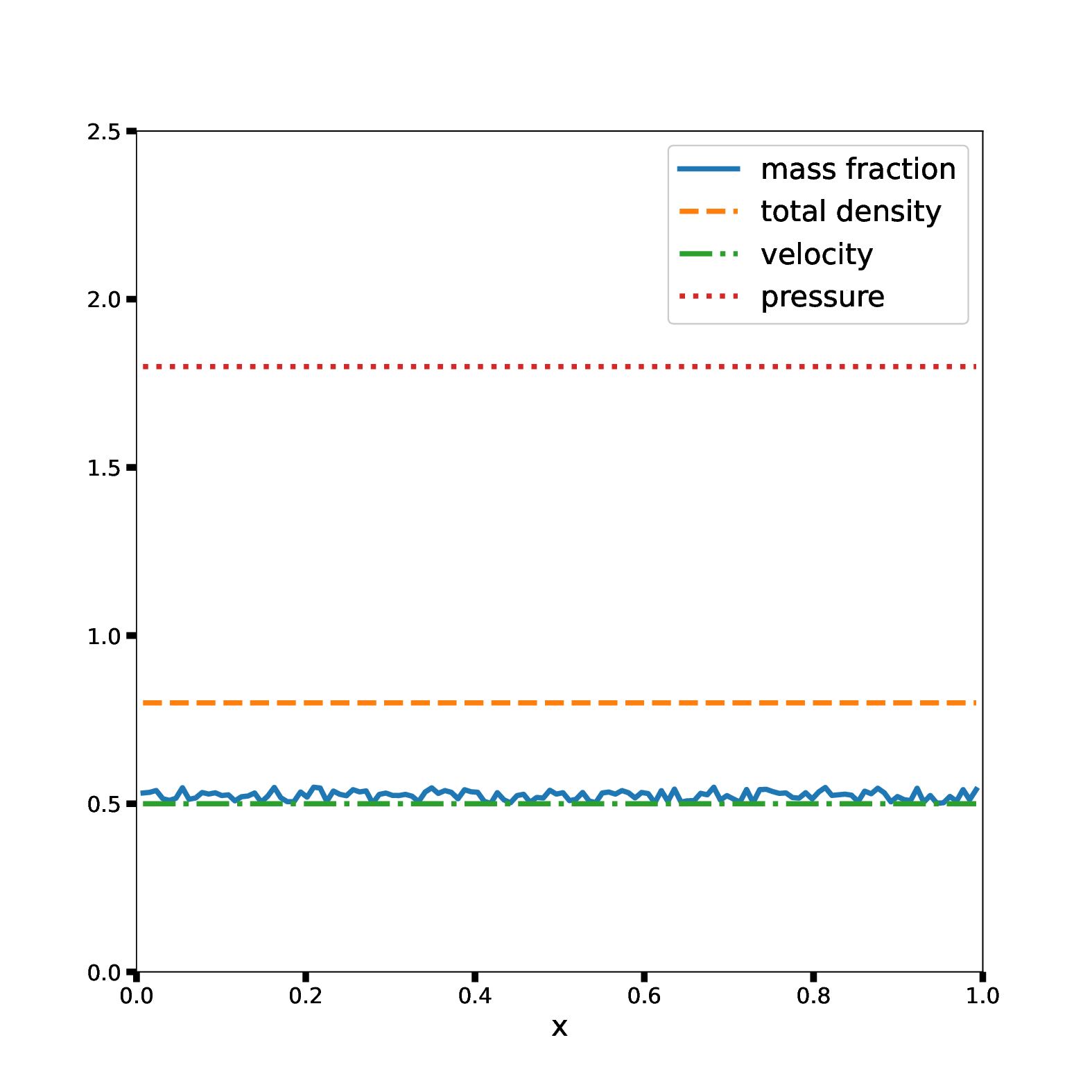}
\caption{$T=0$}
\end{subfigure}
\begin{subfigure}[b]{0.24\textwidth}
\centering 
\includegraphics[width=0.99\linewidth]{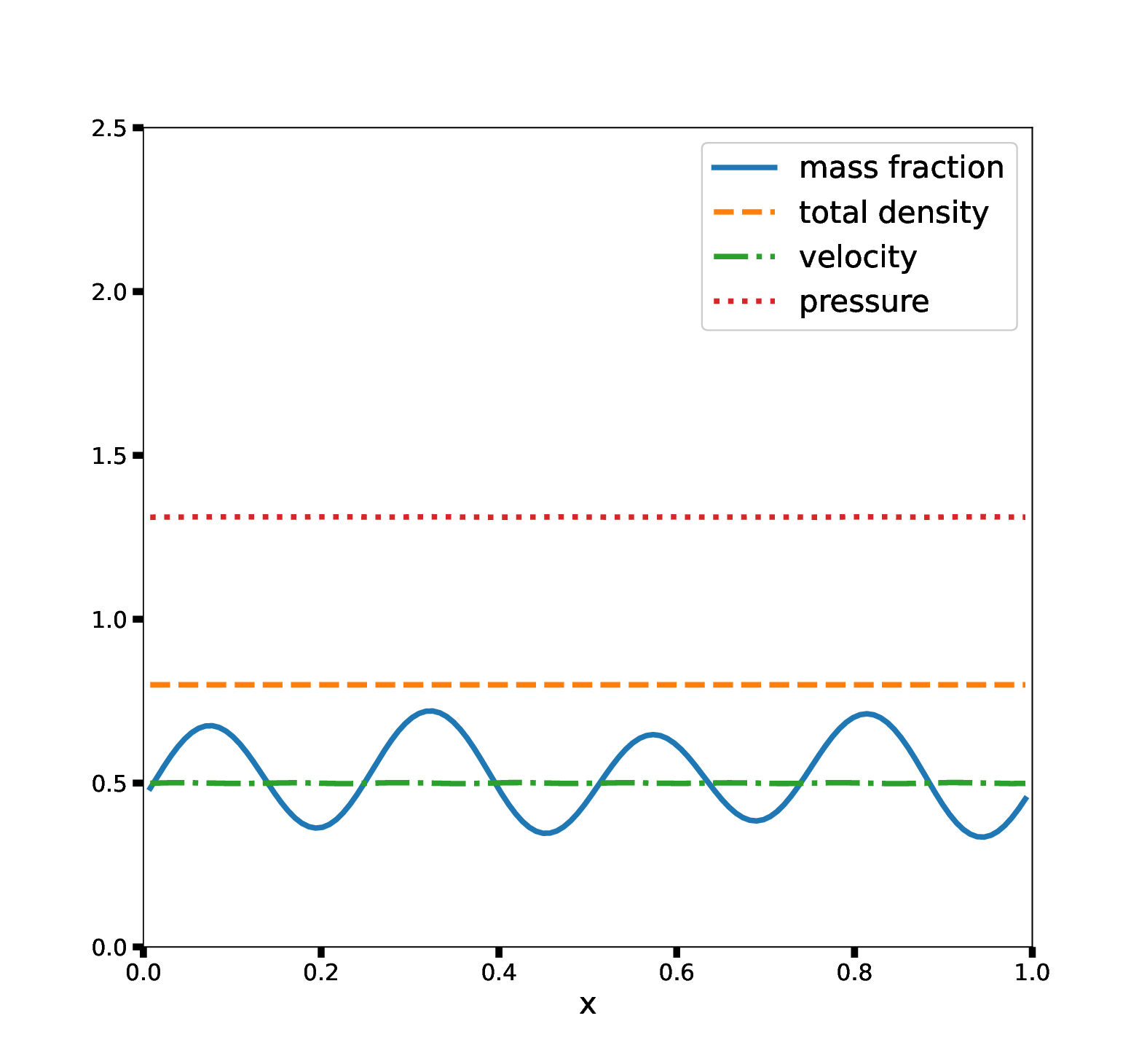}
\caption{$T=0.03$}
\end{subfigure}
\begin{subfigure}[b]{0.24\textwidth}
\centering 
\includegraphics[width=0.99\linewidth]{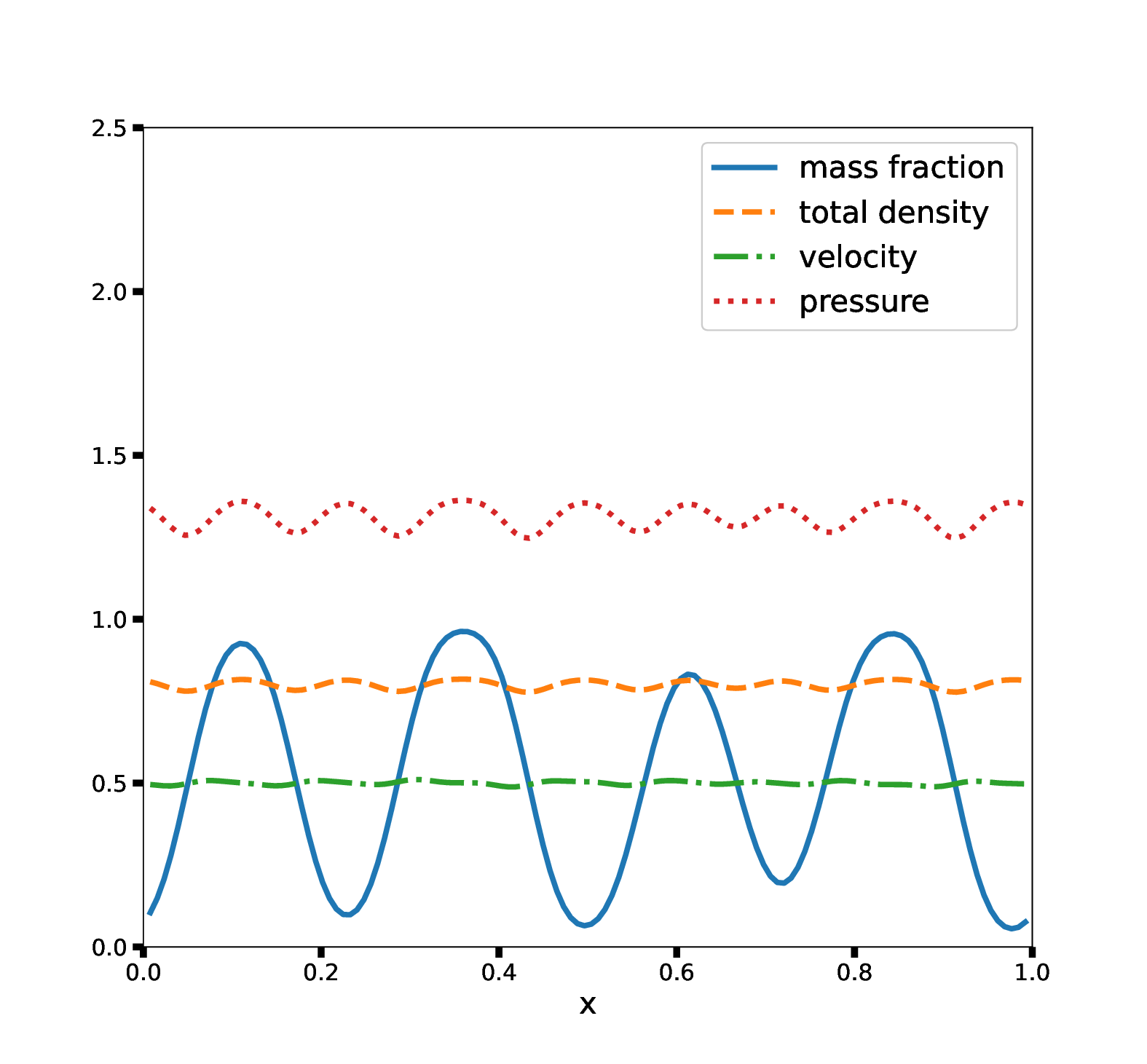}
\caption{$T=0.1$}
\end{subfigure}
\begin{subfigure}[b]{0.24\textwidth}
\centering 
\includegraphics[width=0.99\linewidth]{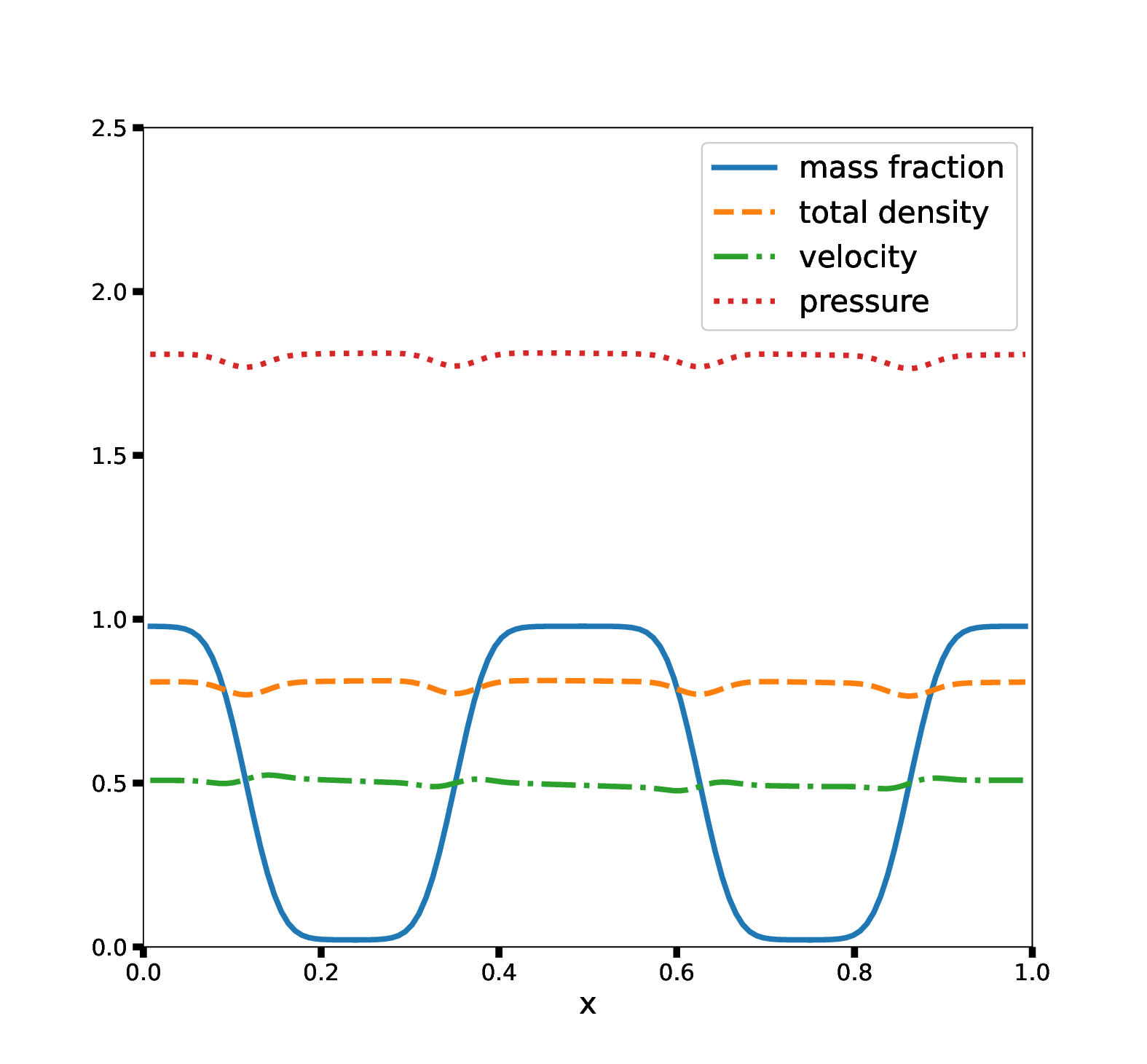}
\caption{$T=0.5$}
\end{subfigure}\\
\begin{subfigure}[b]{0.24\textwidth}
\centering 
\includegraphics[width=0.99\linewidth]{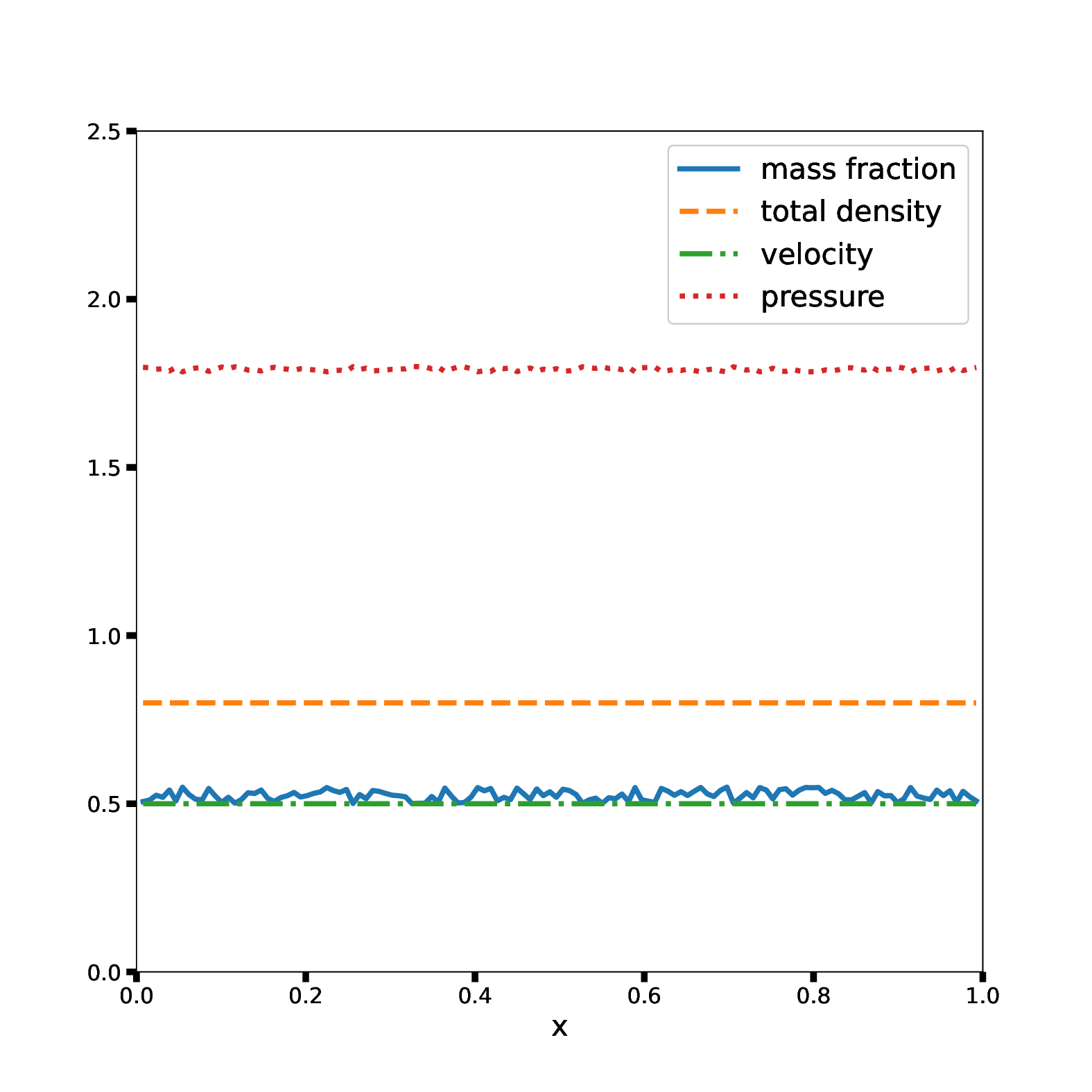}
\caption{$T=0$}
\end{subfigure}
\begin{subfigure}[b]{0.24\textwidth}
\centering 
\includegraphics[width=0.99\linewidth]{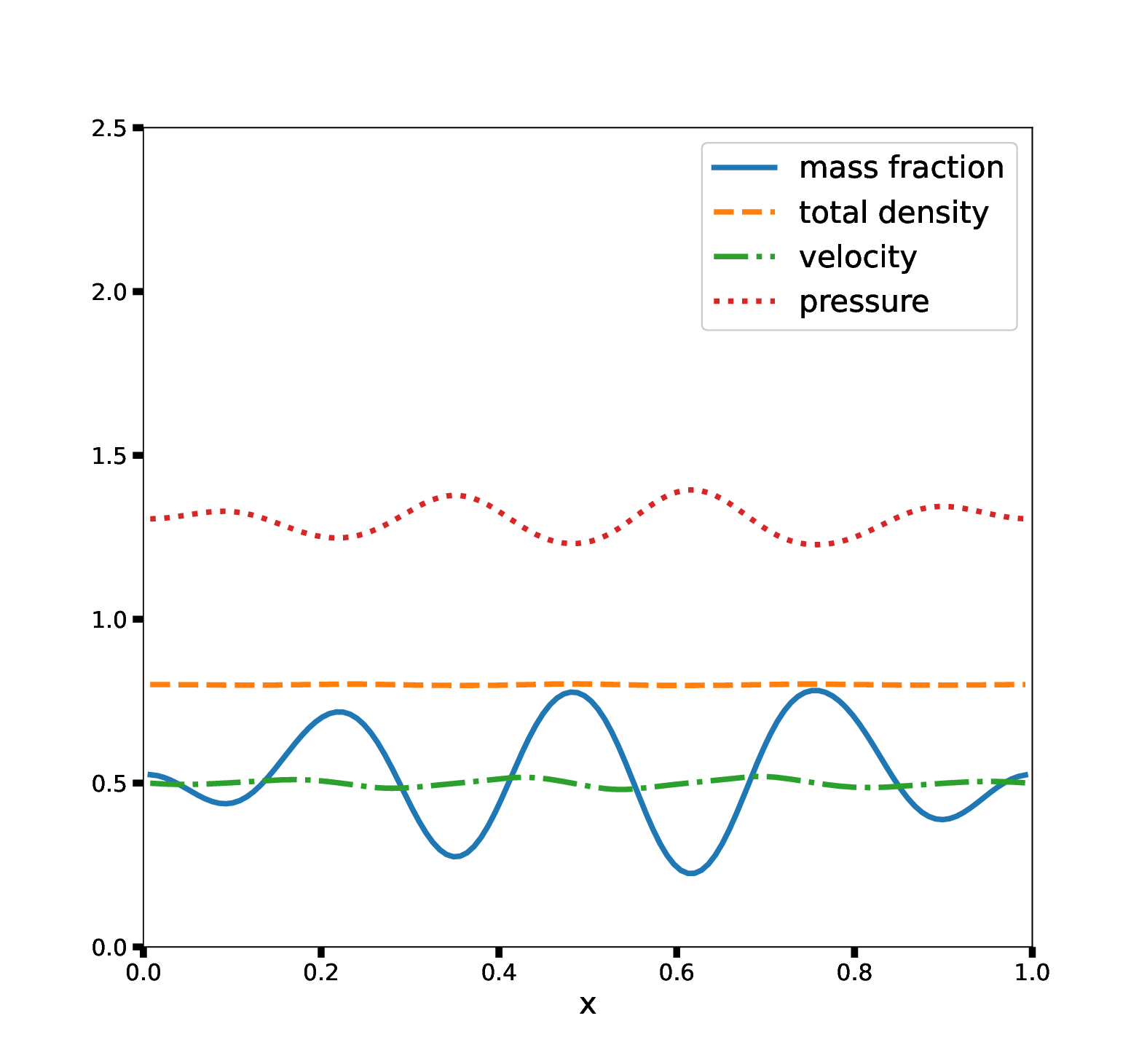}
\caption{$T=0.03$}
\end{subfigure}
\begin{subfigure}[b]{0.24\textwidth}
\centering 
\includegraphics[width=0.99\linewidth]{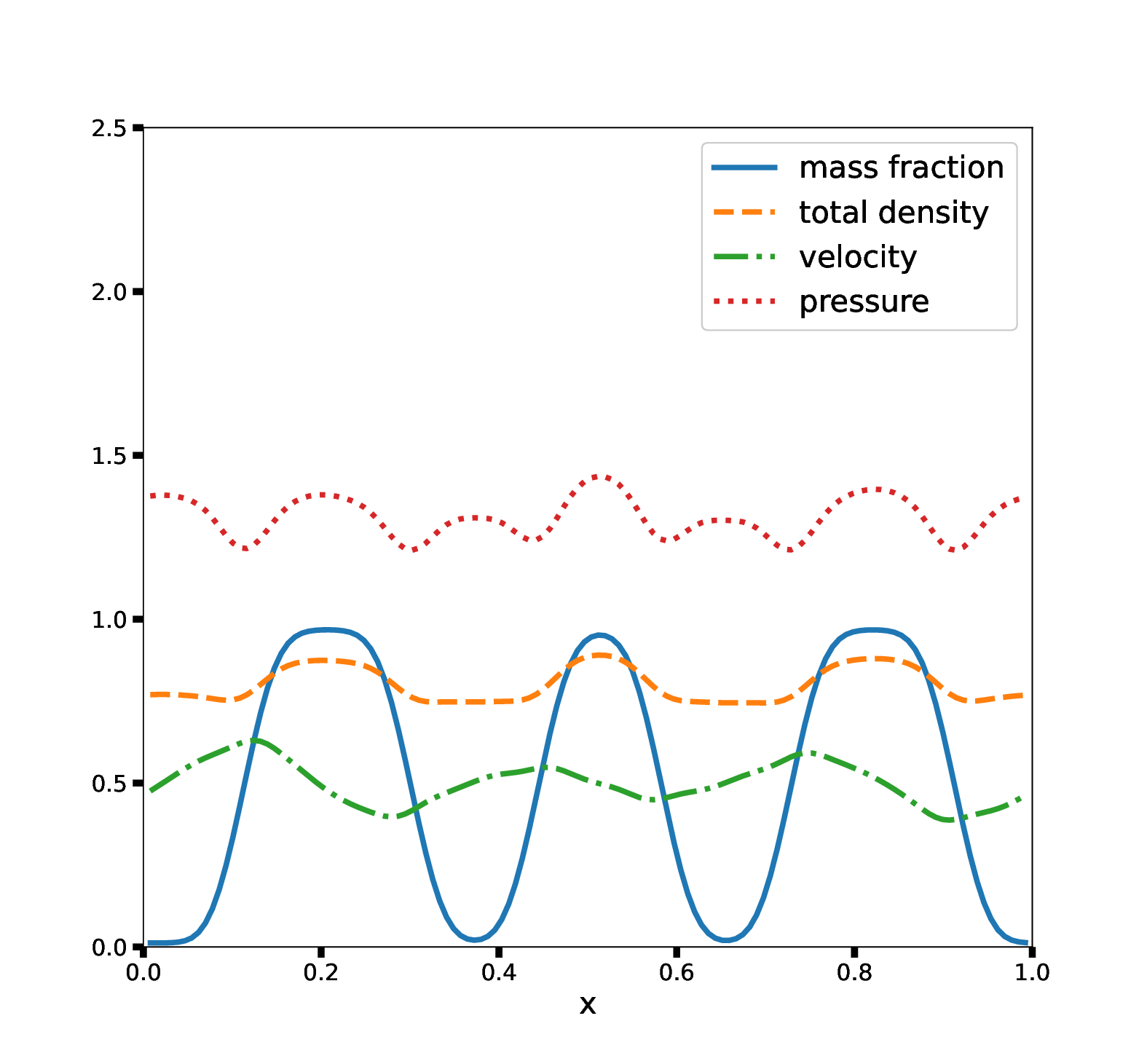}
\caption{$T=0.1$}
\end{subfigure}
\begin{subfigure}[b]{0.24\textwidth}
\centering 
\includegraphics[width=0.99\linewidth]{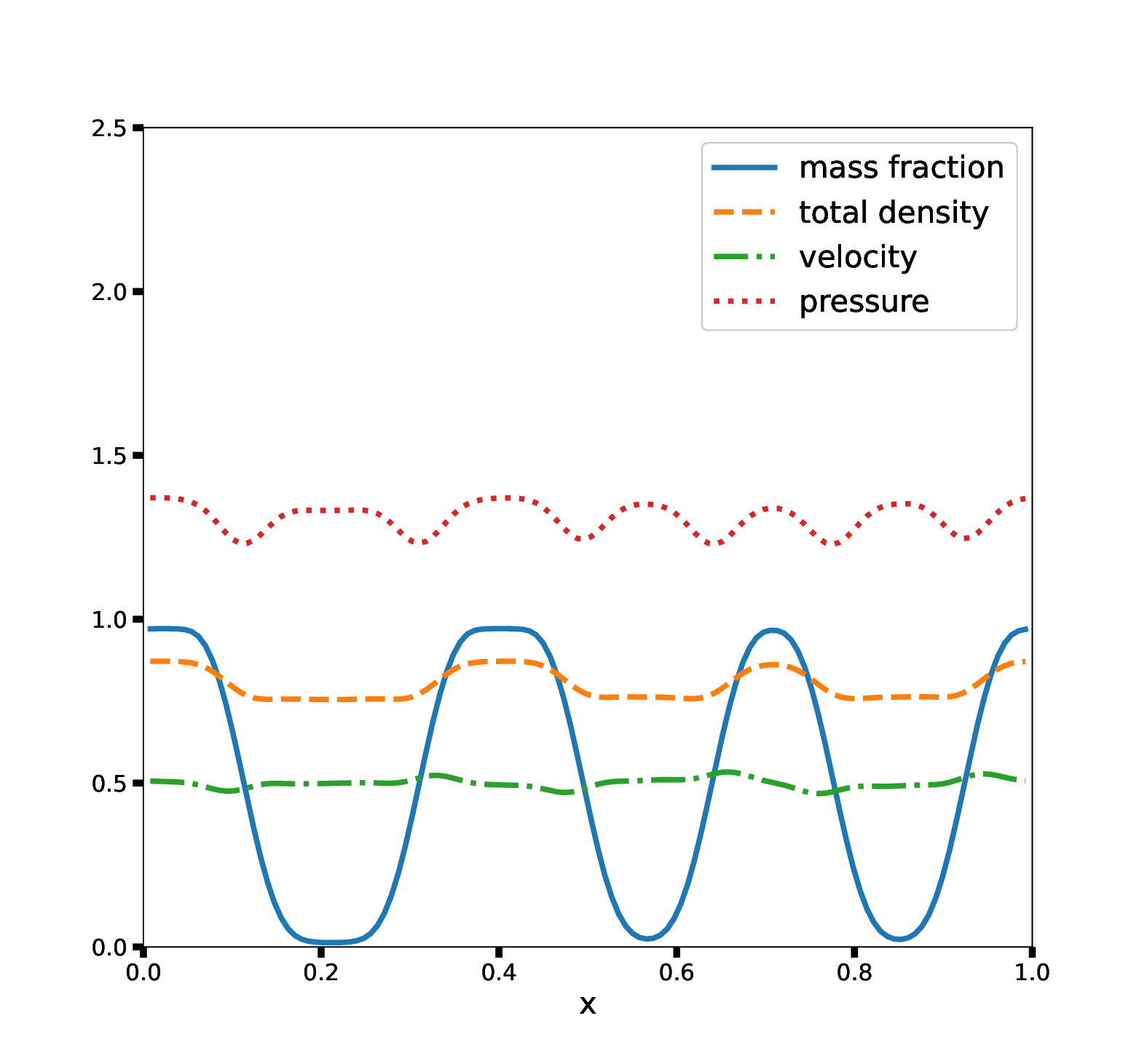}
\caption{$T=0.5$}
\end{subfigure}
\caption{Simulation of compressible Navier-Stokes-Cahn-Hilliard model (with $\kappa = F_c = 0$). Matching densities (Top row) and non-matching densities (bottom row) for the two phases of the fluid.}
\label{fig:1D_comparison_solu}
\end{figure}

Figure~\ref{fig:comparison_properties} shows that, for both cases (\ie matching and non-matching densities), our numerical scheme preserved the properties presented in Proposition~\ref{prop:1dschemeprops}. 
We defined the discrete dissipation of energy 
\begin{equation}\label{eq:discrete_ineq}
\frac{\dd E}{\dd t} = \norm{\sqrt{a} U^{n+1}}^2 + \norm{V^{n+1}}^2 + r^{n+1} -\left[ \norm{\sqrt{a} U^{n}}^2+ \norm{V^{*}}^2 + C^{n+1}r^{n}\right].
\end{equation}
We emphasize that as no exchange term was present in the previous simulation $C^{n+1}$ is bounded from above by $1$, hence we used $C^{n+1} = 1$ for the simulations in this paragraph.
Figure~\ref{fig:comparison_properties} presents the temporal evolution of the dissipation $\frac{\dd E}{\dd t}$, the mass $\int_\Omega \rho c\dd x $, the minimum and maximum values of $c$, and the value of $\xi$. 
We observe for both cases that the dissipation~\eqref{eq:discrete_ineq} is strictly negative, as expected by proposition~\ref{prop:1dschemeprops}. The mass of fluid $1$ is preserved up to a small numerical error (we emphasize that the error on the initial mass at the end of the simulation is less than $10^{-9}$ for both simulations). This latter result is expected as we set $F_c(\rho,c) = 0$ for these simulations. We observe that the physical bounds of the mass fraction are ensured, \ie maximum and minimum values for $c$ lie the interval $(0,1)$. The scalar variable is very close to $1$ (up to an error of order $10^{-5}$) as observed in Figure~\ref{fig:1D_xi}. This verifies that the modified energy $r^{n+1}$ and the real energy of the Cahn-Hilliard part of the model
\[
E^{n+1} = \Delta x \sum_j \frac{\gamma}{2} \abs{(\nabla c)_j}^2 + \rho^{n+1}\left(\frac{\partial \psi_0}{\partial c^{n+1}}\right)_j,
\]
are close. 

\begin{figure}[h!]
     \centering
     \begin{subfigure}[b]{0.49\textwidth}
         \centering
         \includegraphics[width=\textwidth]{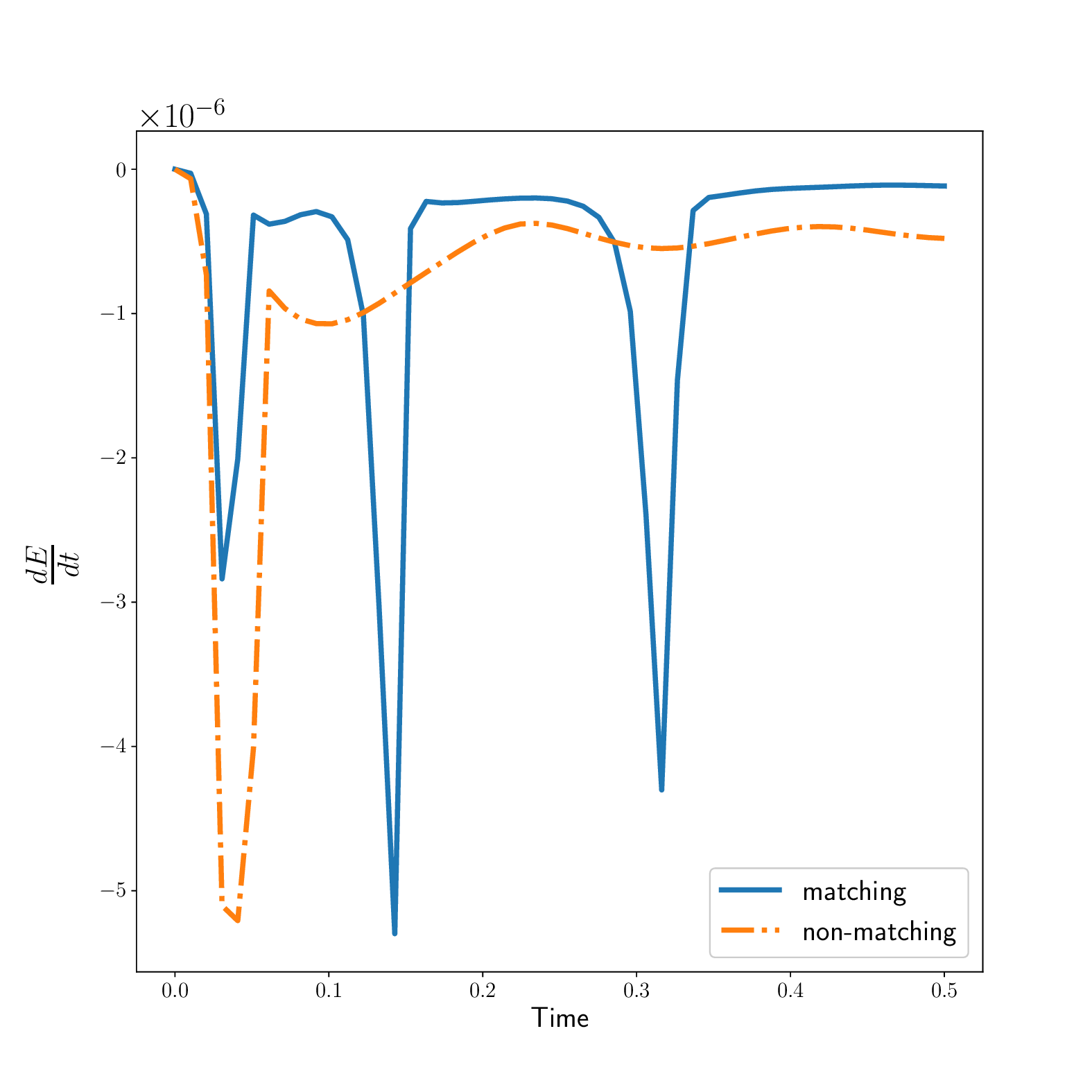}
         \caption{Dissipation $\frac{\dd E}{\dd t}$. The value of the dissipation remains negative ensuring that Inequality~\eqref{eq:dissnum} is satisfied at the discrete level. }
         \label{fig:1D_diss}
     \end{subfigure}
     \hfill
     \begin{subfigure}[b]{0.49\textwidth}
         \centering
         \includegraphics[width=\textwidth]{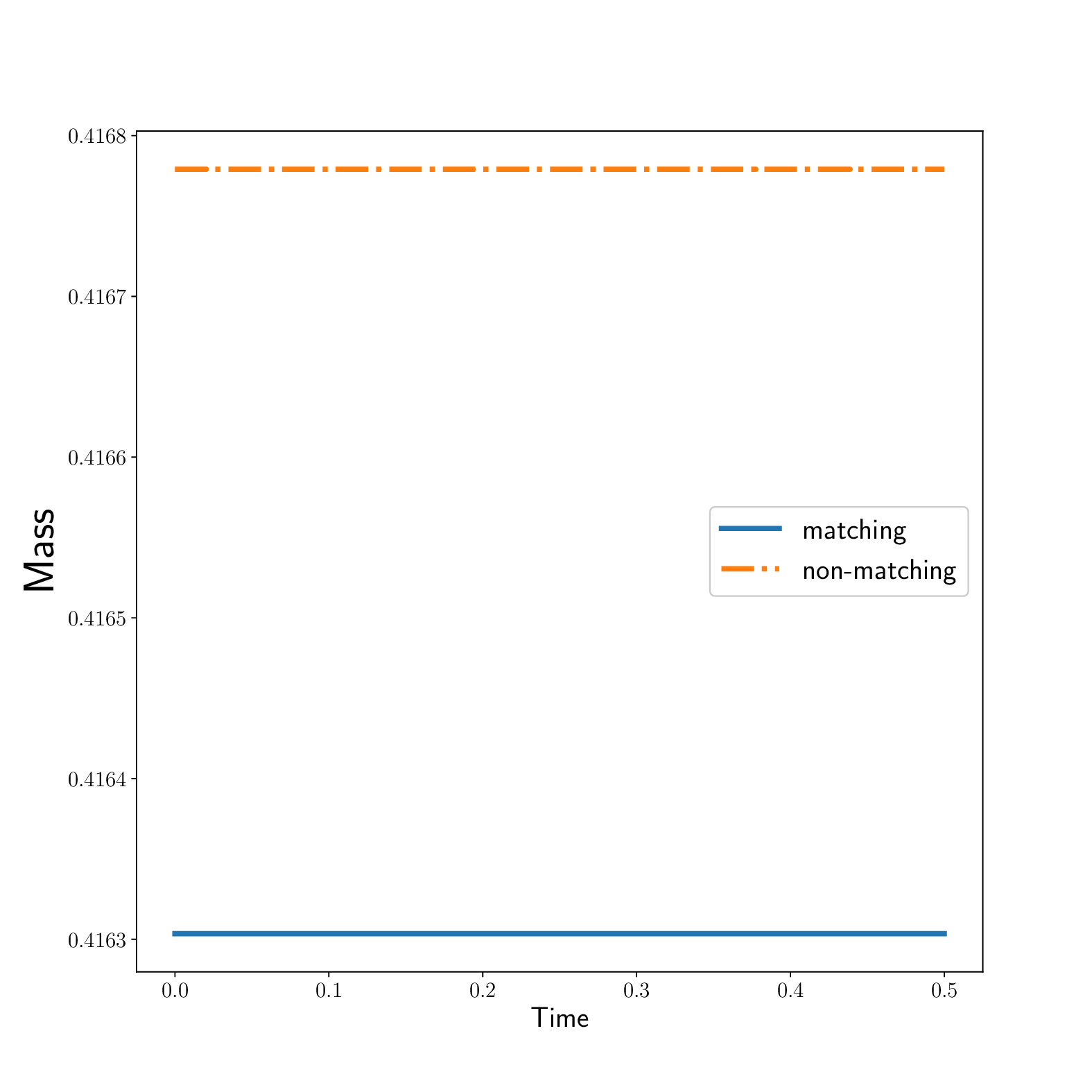}
         \caption{Mass $\int_\Omega \rho c\dd x $ of fluid 1. Difference of initial mass is due to the random initial conditions for the two simulations. }
         \label{fig:1D_mass}
     \end{subfigure}\\
     
     \begin{subfigure}[b]{0.49\textwidth}
         \centering
         \includegraphics[width=\textwidth]{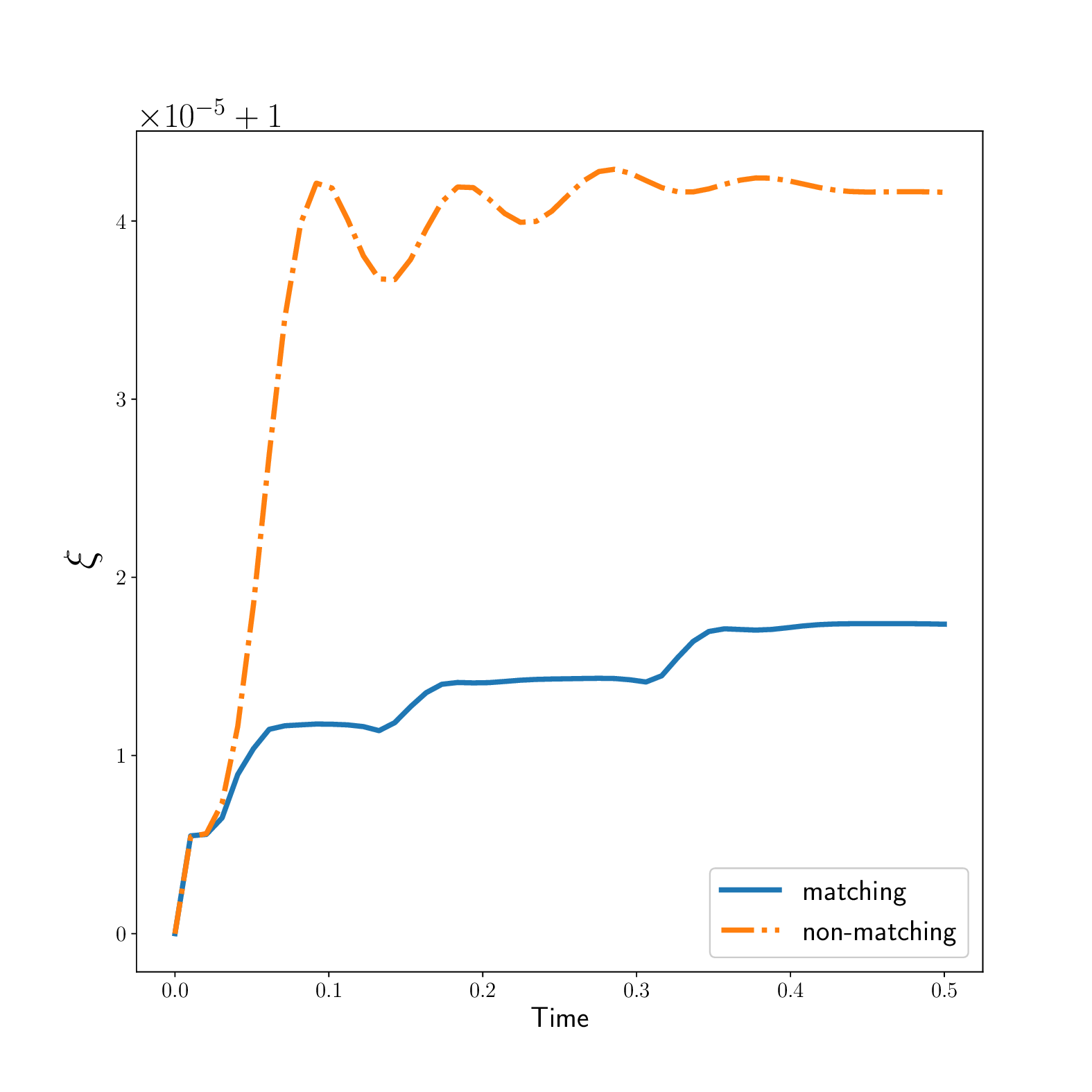}
         \caption{Scalar variable $\xi$. The value remains close to $1$ (error is of order $10^{-5}$), ensuring that the modified energy $r^{n+1}$ is close to the real energy $E^{n+1}$.}
         \label{fig:1D_xi}
     \end{subfigure}
     \hfill 
     \begin{subfigure}[b]{0.49\textwidth}
         \centering
         \includegraphics[width = \textwidth]{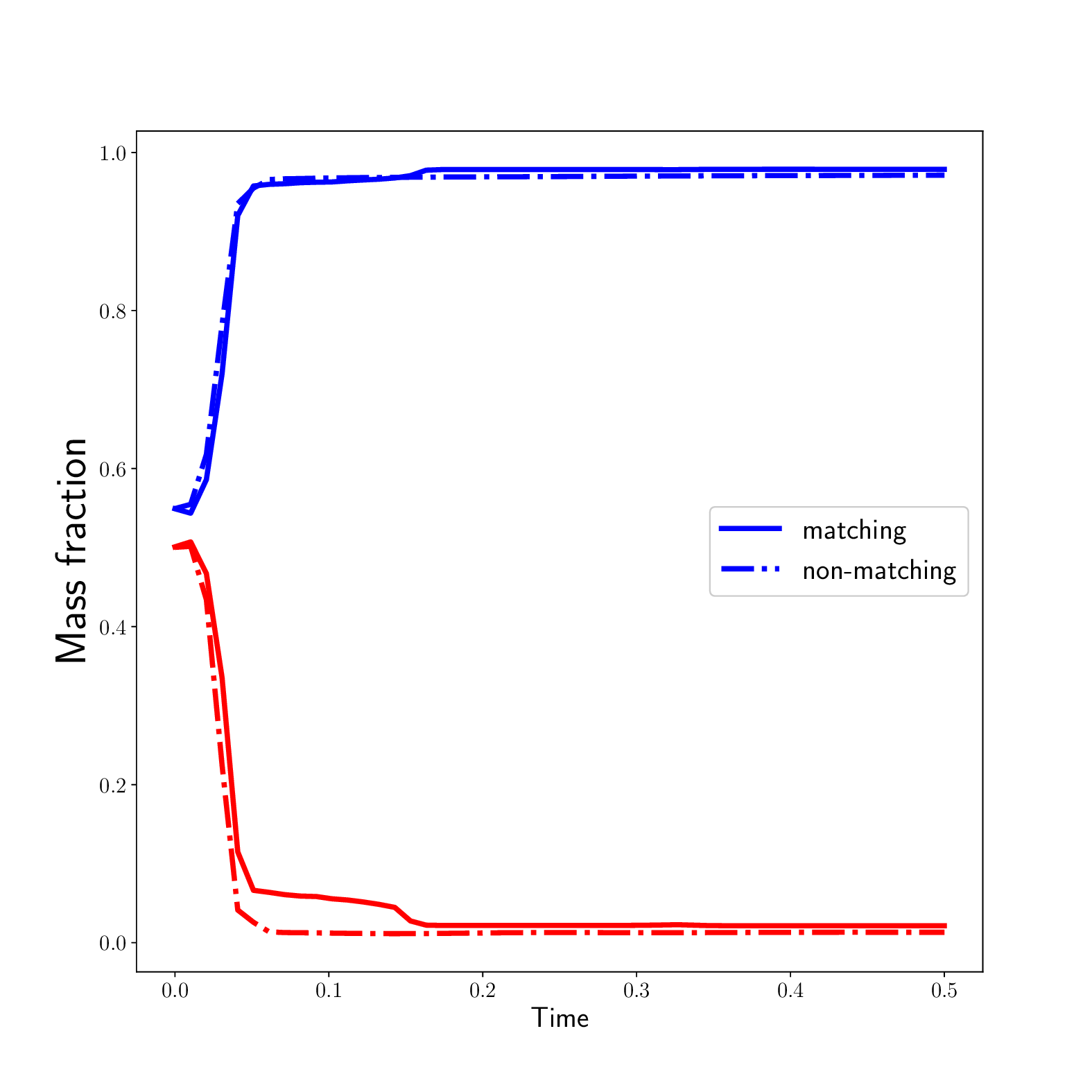}
         \caption{minimum and maximum of $c$. The mass fraction lies in the physically relevant interval $(0,1)$ throughout the simulation.}
         \label{fig:1D_maxmin}
     \end{subfigure}\\
        \caption{Temporal evolution of the dissipation of the energy $\frac{\dd E}{\dd t}$, mass of the fluid $1$ given by $\int_\Omega \rho c \,\dd x$, scalar variable $\xi$, and of the minimum and maximal values of the mass fraction $c$ for matching densities (solid lines) and non-matching densities (dash-dotted lines).}
        \label{fig:comparison_properties}
\end{figure}

\paragraph{Mass exchange and contrast of friction forces. }

In this test case, we consider mass exchange between the two phases and friction effects. We choose
\[
F_c(\rho,c) = r_\text{trans} \rho c (1-c/c_\text{max}), \quad \kappa(\rho, c) = \rho c\kappa_1+ \rho (1-c) \kappa_2,
\]
where $0<c_\text{max}<1$ denotes the mass fraction at which we have an equilibrium for exchange of mass and $r_\text{trans}$ is the rate of mass exchange. 
In this test case, we use $c_\text{max} = 0.9$, and $r_\text{trans} = 1$. 
We compare the solution obtained with no contrast of friction effect, $\ie$ $\kappa_1 = \kappa_2 = 10$, and the solution obtained with $\kappa_1 = 0,\; \kappa_2 = 10$. 
To study the long time behavior of the numerical simulations, we set $T = 5$. 
The rest of the parameters and the initial conditions are chosen as for the previous non-matching densities test case. 

Figure~\ref{fig:1D_comparison_solu_frict_source} compares the solutions for the two cases. In both cases, we observe that the separation of the two phases occurs and that the velocity decreases in time due to friction effects. Furthermore, as time passes, phase $1$ of the fluid increases due to the exchange term $F_c(\rho,c)\neq 0$ and, hence, zones of mass fraction close to the value $1$ enlarge. At the end of both simulations, there is one large aggregate of fluid $1$. The difference between the two simulations appears clearly at time $t = 0.1$. When the friction forces are stronger in phase two compared to phase $1$, \ie $\kappa_2 > \kappa_1$, zones of larger density appear destabilized, \ie the shape of the aggregates is not symmetric (compare the solution for Figure~\ref{fig:stable_aggregates} and Figure~\ref{fig:unstable_aggregates}). For each aggregate of fluid $1$, the density at the right of the aggregate is larger compared to the left. We conclude that the contrast in friction forces is captured well by the model as simulations depict a contrast of velocity for the two phases of the fluid and leading to less regular parterns for the densities. Furthermore, we emphasize that even with non-zero mass exchange and friction forces, the numerical scheme ensures the properties stated in Proposition~\ref{prop:1dschemeprops} as observed in Figure~\ref{fig:comparison_properties_friction}. We emphasize that compared to the simulation without source in which $\xi$ seems to converge to a constant value, the variable $\xi$ increases slightly with time (compare Figures~\ref{fig:1D_xi} and~\ref{fig:1D_xi_frict}). A possible remedy to this issue is discussed in the conclusion of this article. 

\begin{figure}
\centering 
     \begin{subfigure}[b]{0.24\textwidth}
\centering 
\includegraphics[width=0.99\linewidth]{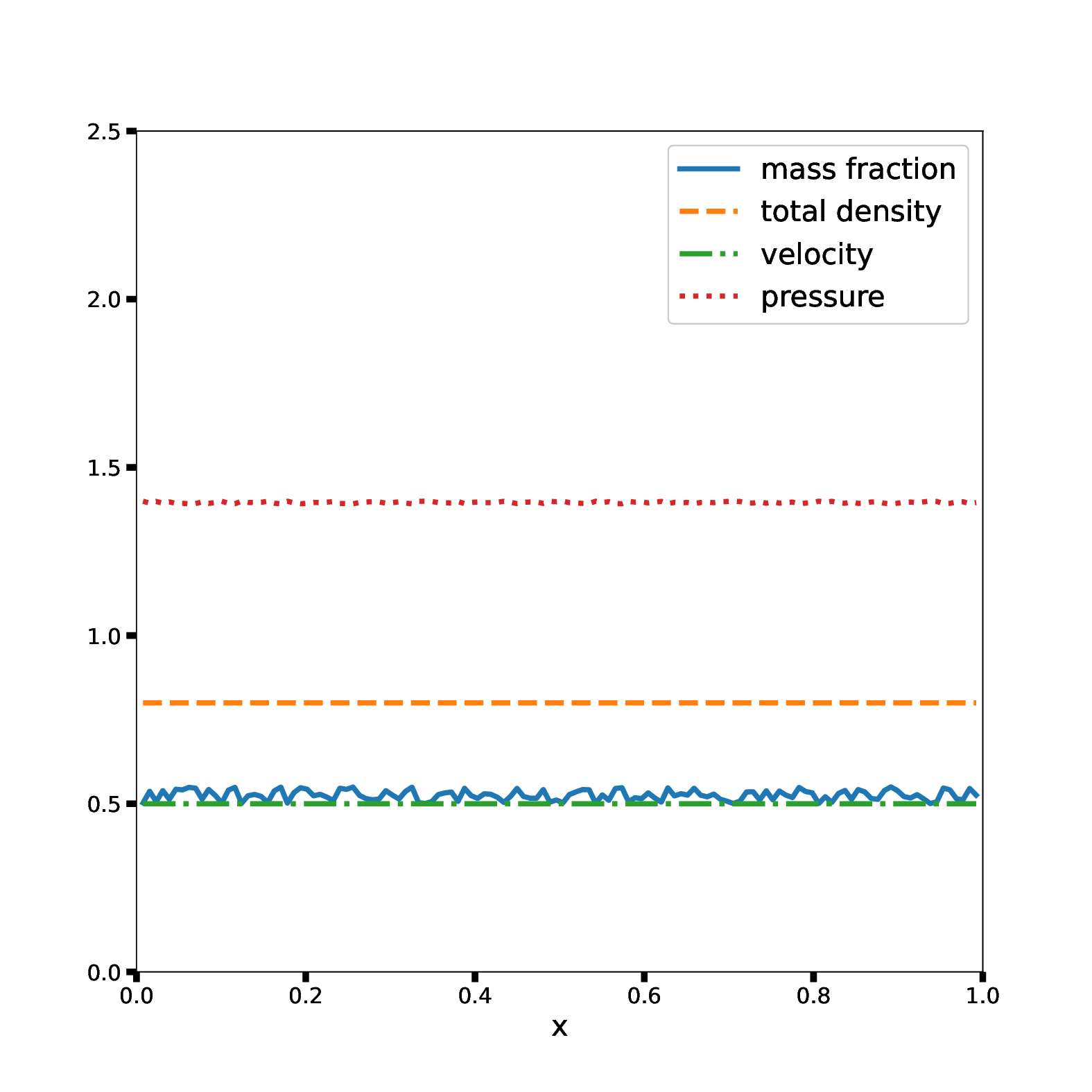}
\caption{$t=0$}
\end{subfigure}
\begin{subfigure}[b]{0.24\textwidth}
\centering 
\includegraphics[width=0.99\linewidth]{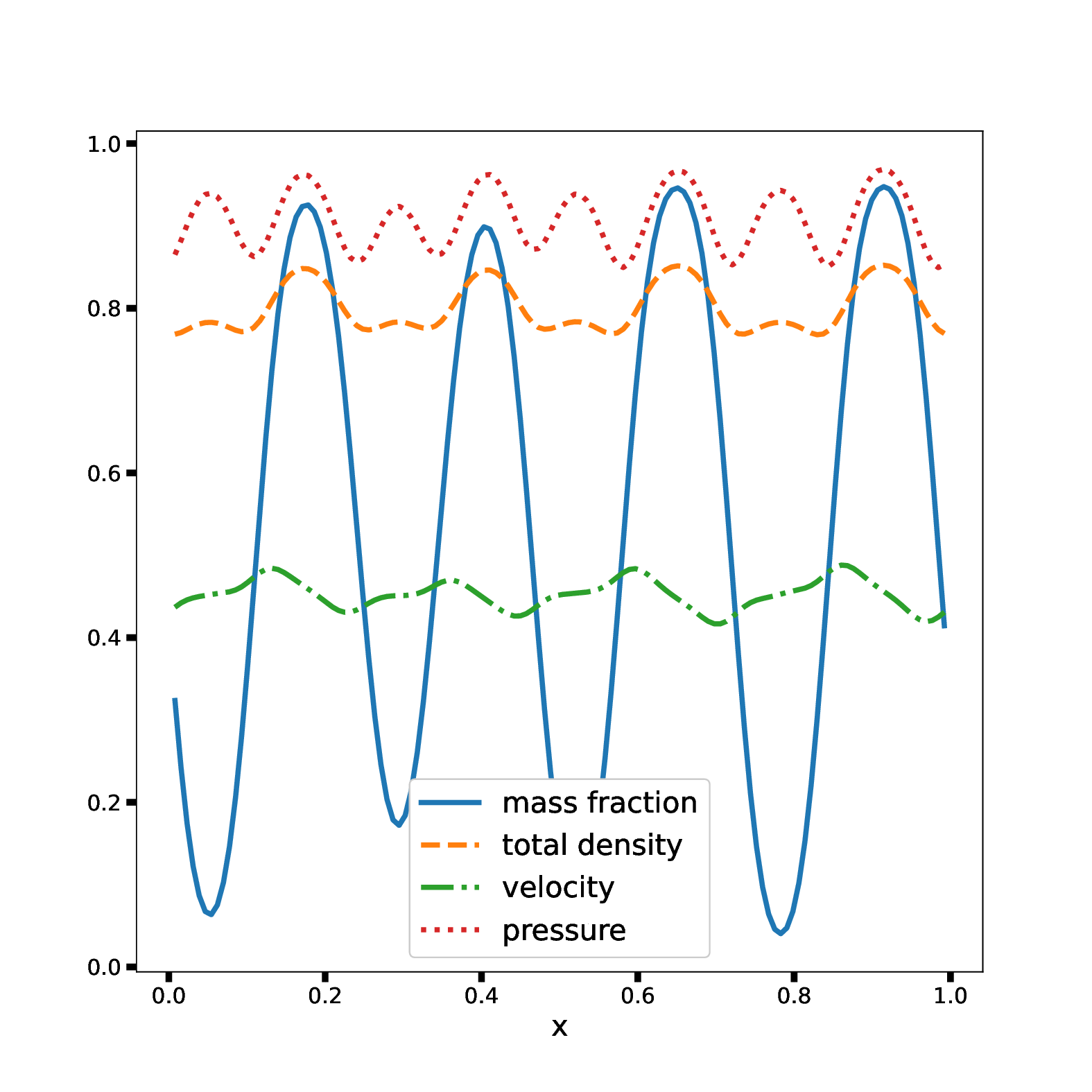}
\caption{$t=0.1$}
\label{fig:stable_aggregates}
\end{subfigure}
\begin{subfigure}[b]{0.24\textwidth}
\centering 
\includegraphics[width=0.99\linewidth]{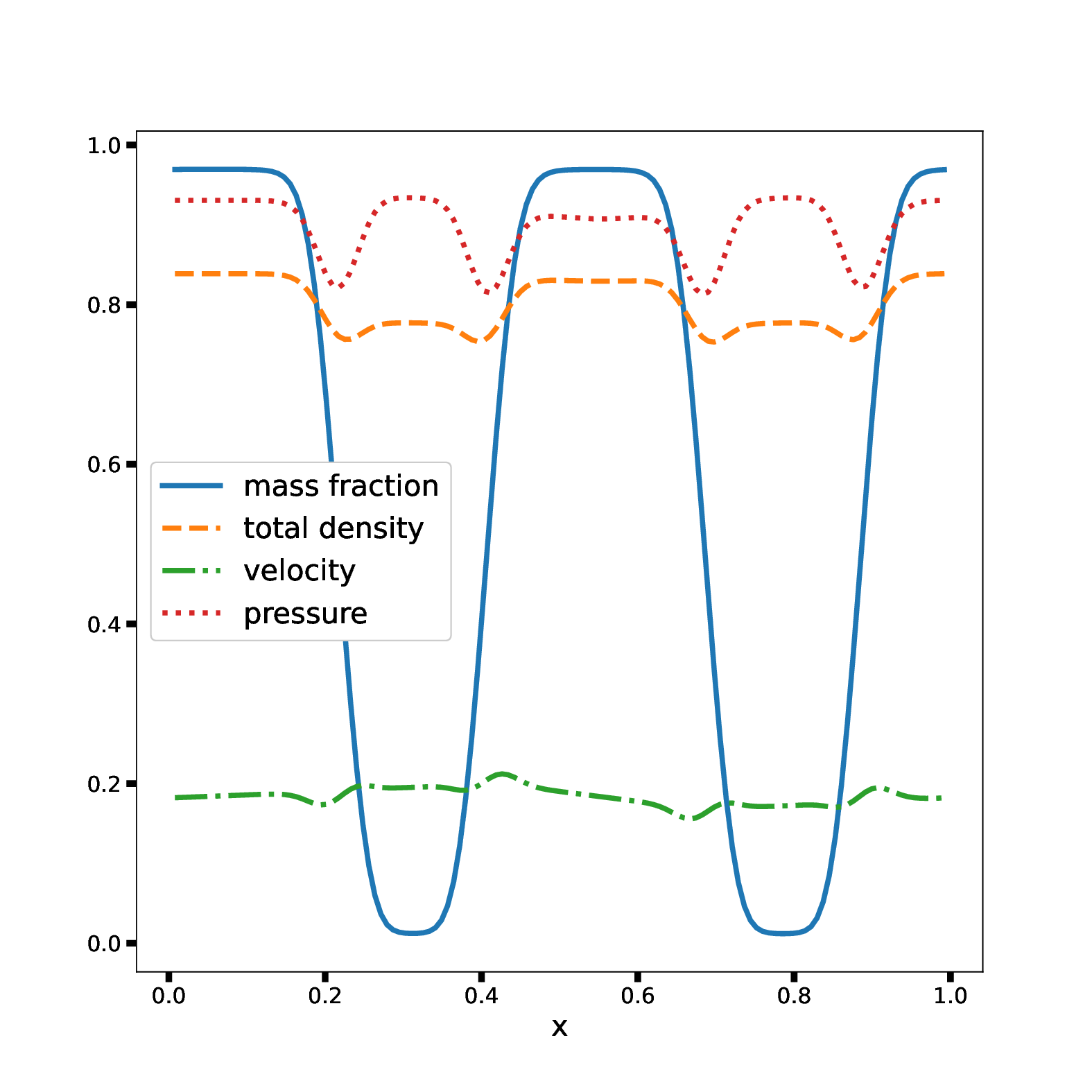}
\caption{$t=1$}
\end{subfigure}
\begin{subfigure}[b]{0.24\textwidth}
\centering 
\includegraphics[width=0.99\linewidth]{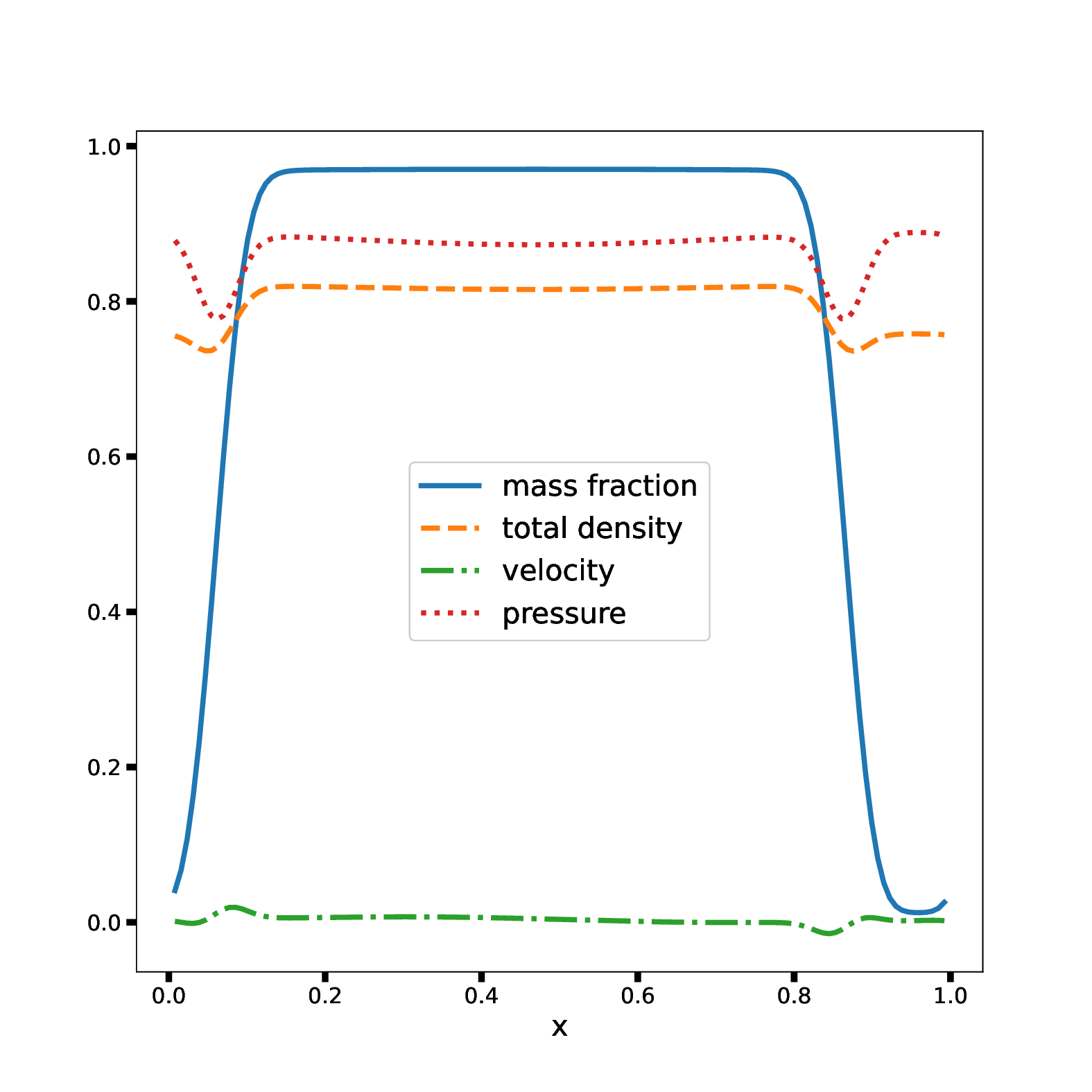}
\caption{$t=5$}
\end{subfigure}\\
\begin{subfigure}[b]{0.24\textwidth}
\centering 
\includegraphics[width=0.99\linewidth]{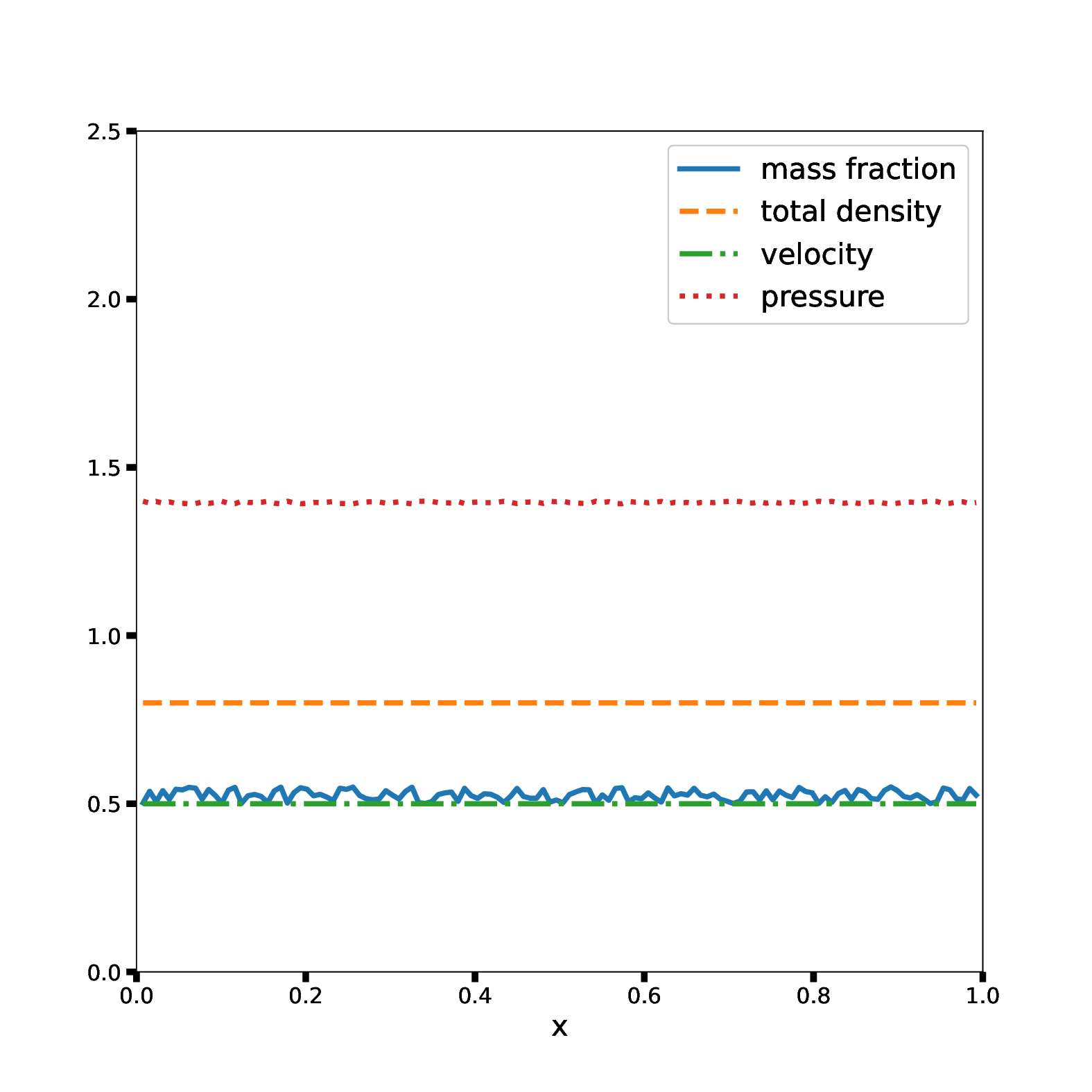}
\caption{$t=0$}
\end{subfigure}
\begin{subfigure}[b]{0.24\textwidth}
\centering 
\includegraphics[width=0.99\linewidth]{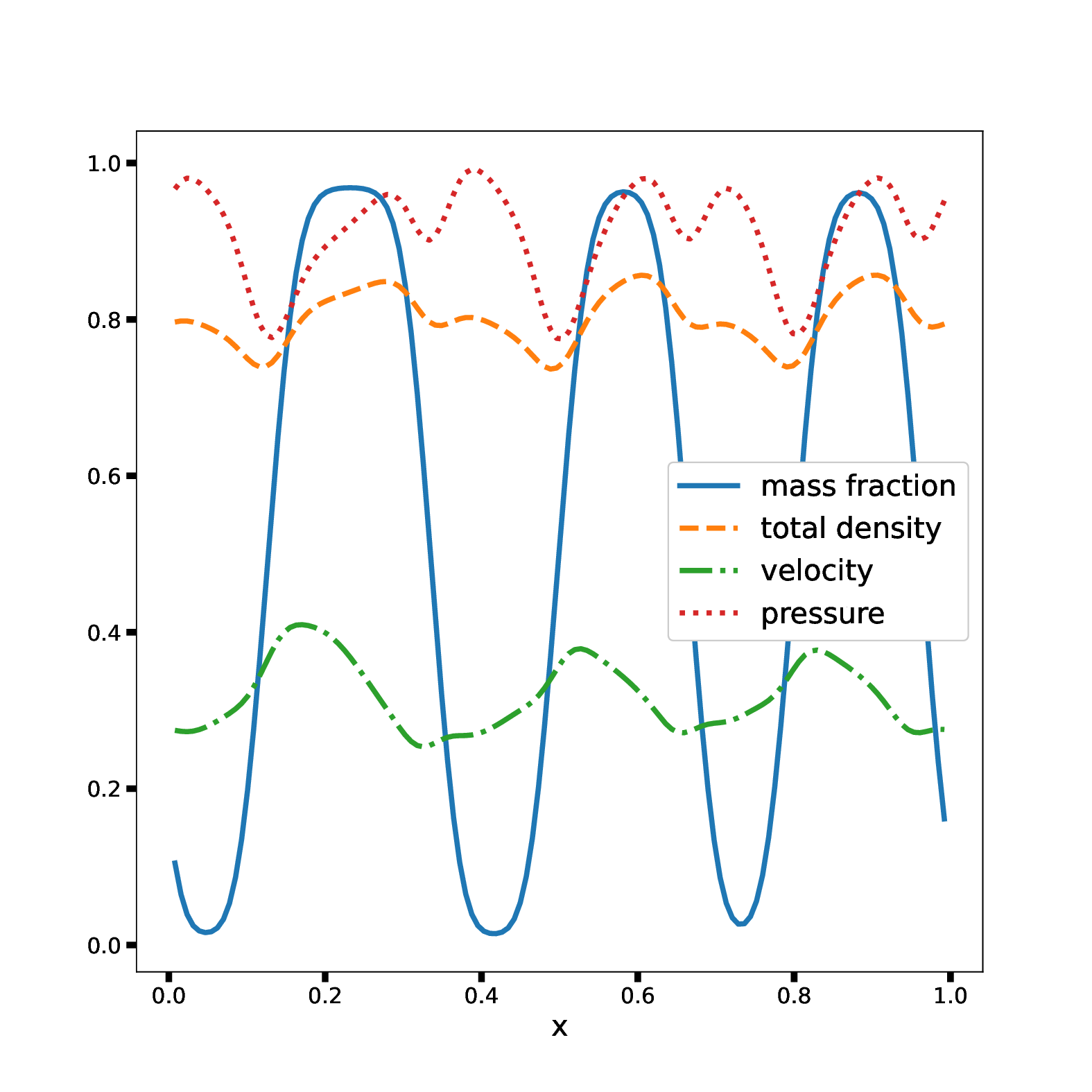}
\caption{$t=0.1$}
\label{fig:unstable_aggregates}
\end{subfigure}
\begin{subfigure}[b]{0.24\textwidth}
\centering 
\includegraphics[width=0.99\linewidth]{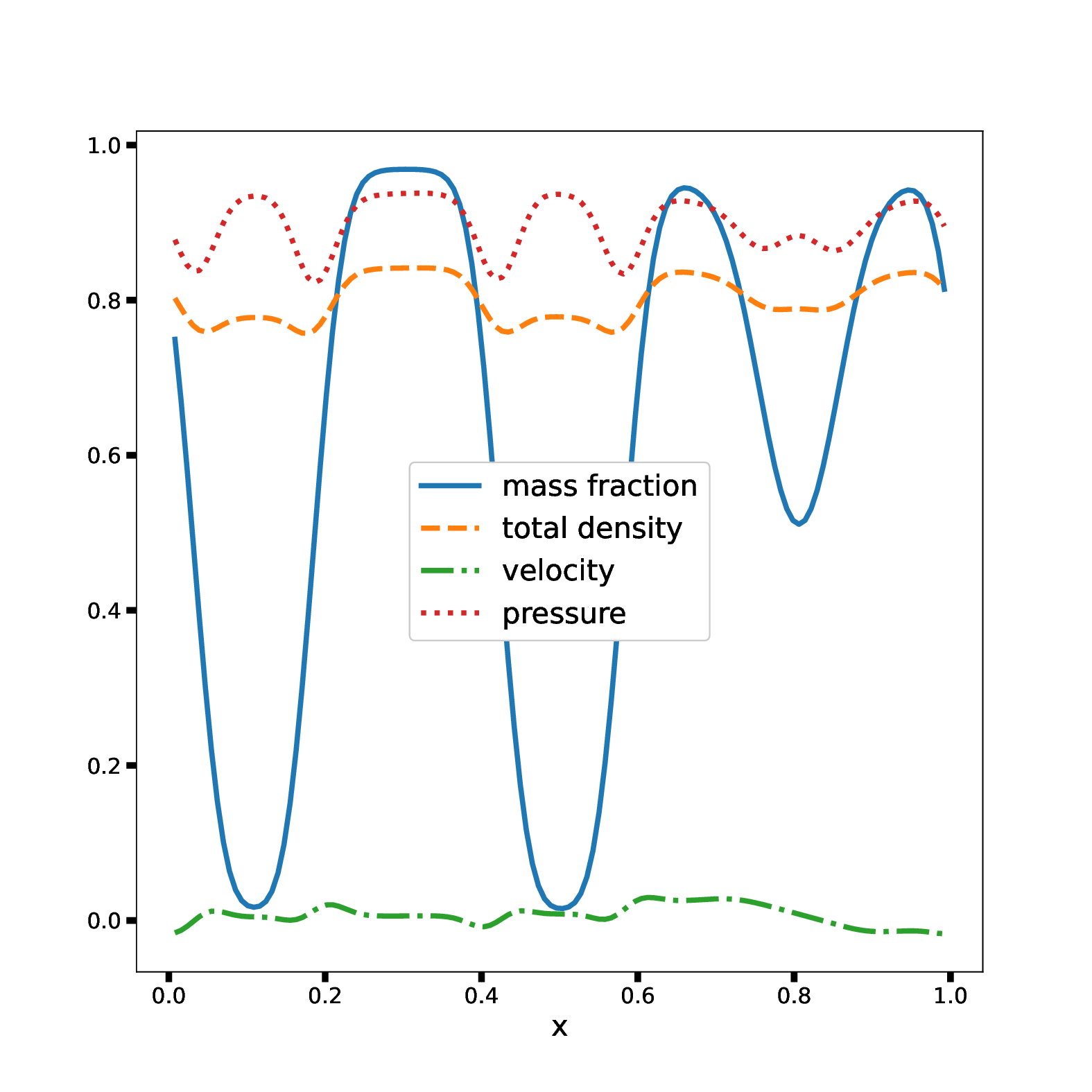}
\caption{$t=1$}
\end{subfigure}
\begin{subfigure}[b]{0.24\textwidth}
\centering 
\includegraphics[width=0.99\linewidth]{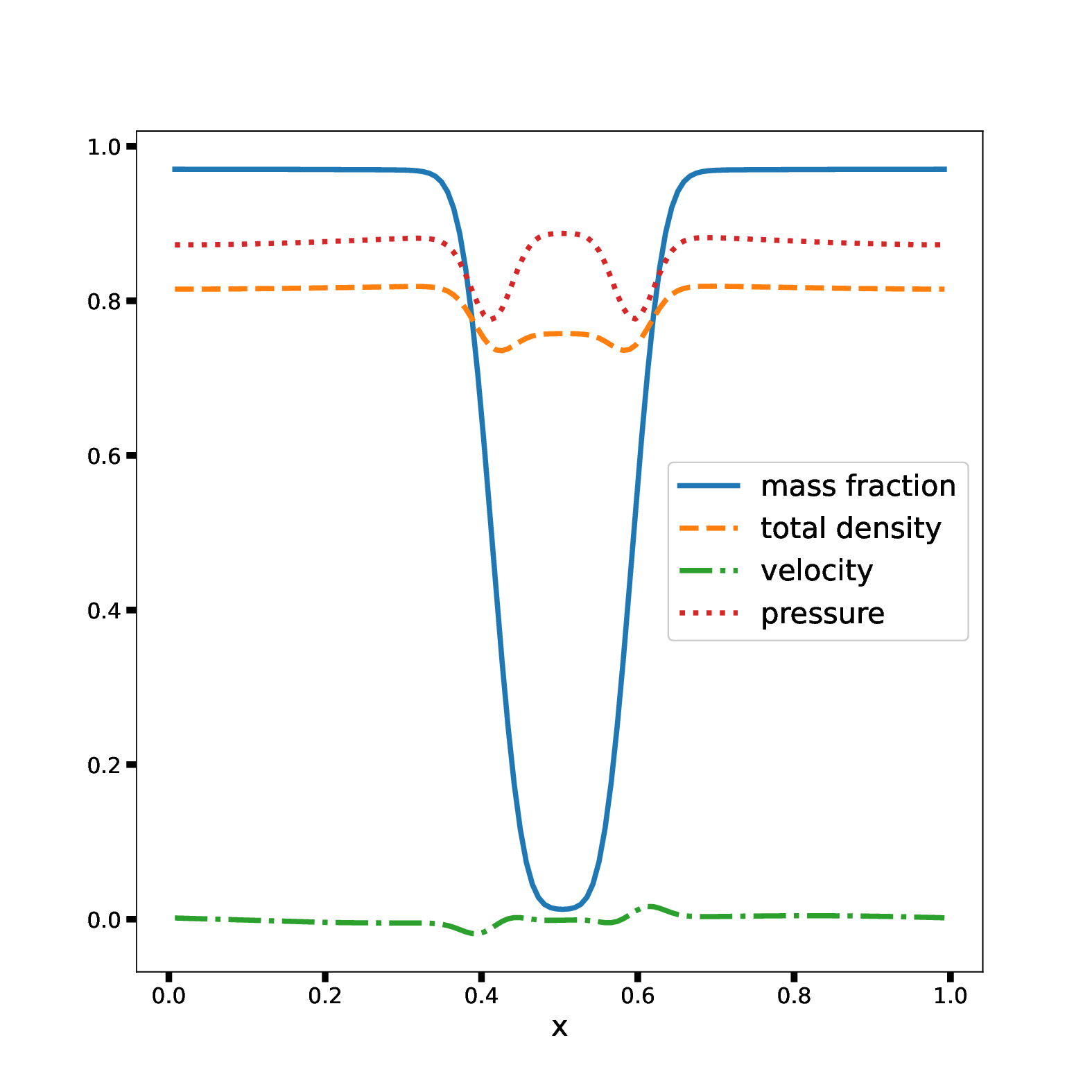}
\caption{$t=5$}
\end{subfigure}
\caption{Simulation of compressible Navier-Stokes-Cahn-Hilliard model nonmatching densities, exchange term ($F_c(\rho,c)\neq 0$), same friction effects for both fluids (top row) and contrast of friction forces (bottom row).}
\label{fig:1D_comparison_solu_frict_source}
\end{figure}

\begin{figure}[h!]
     \centering
     \begin{subfigure}[b]{0.24\textwidth}
         \centering
         \includegraphics[width=\textwidth]{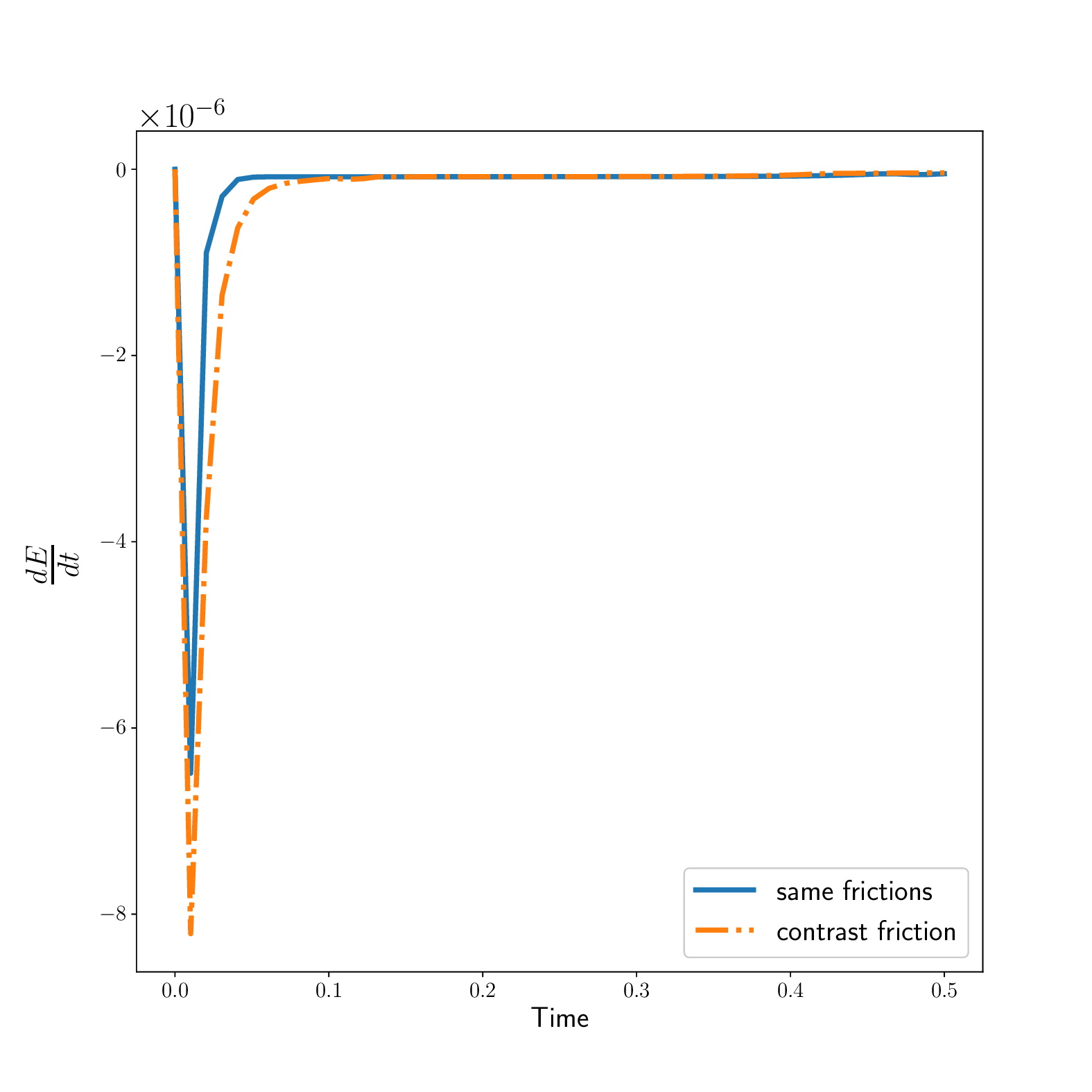}
         \caption{Dissipation $\frac{\dd E}{\dd t}$. }
         \label{fig:1D_diss_frict}
     \end{subfigure}
     \hfill
     \begin{subfigure}[b]{0.24\textwidth}
         \centering
         \includegraphics[width=\textwidth]{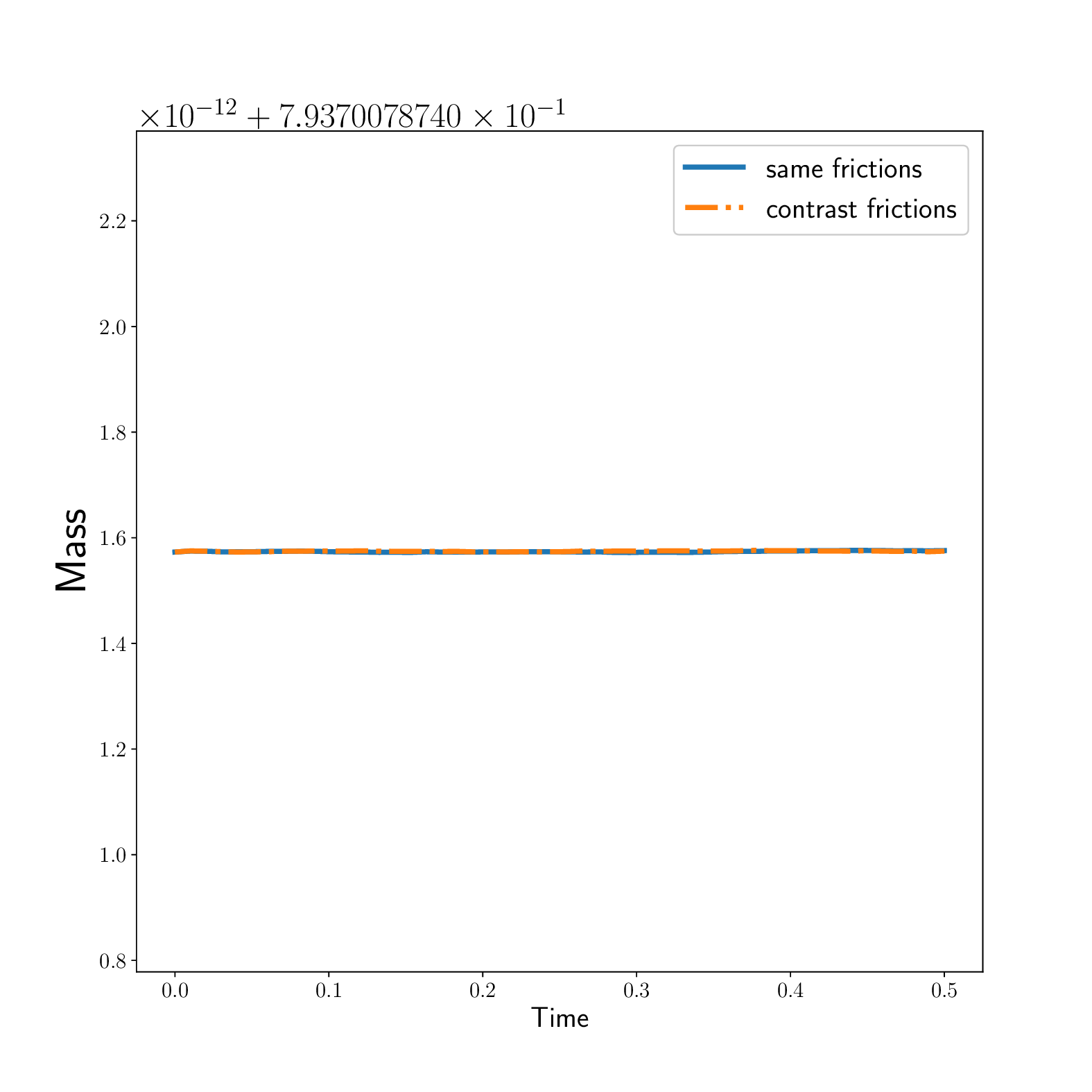}
         \caption{Total mass $\int_\Omega \rho \dd x $.}
         \label{fig:1D_mass_frict}
     \end{subfigure}
     \begin{subfigure}[b]{0.24\textwidth}
         \centering
         \includegraphics[width=\textwidth]{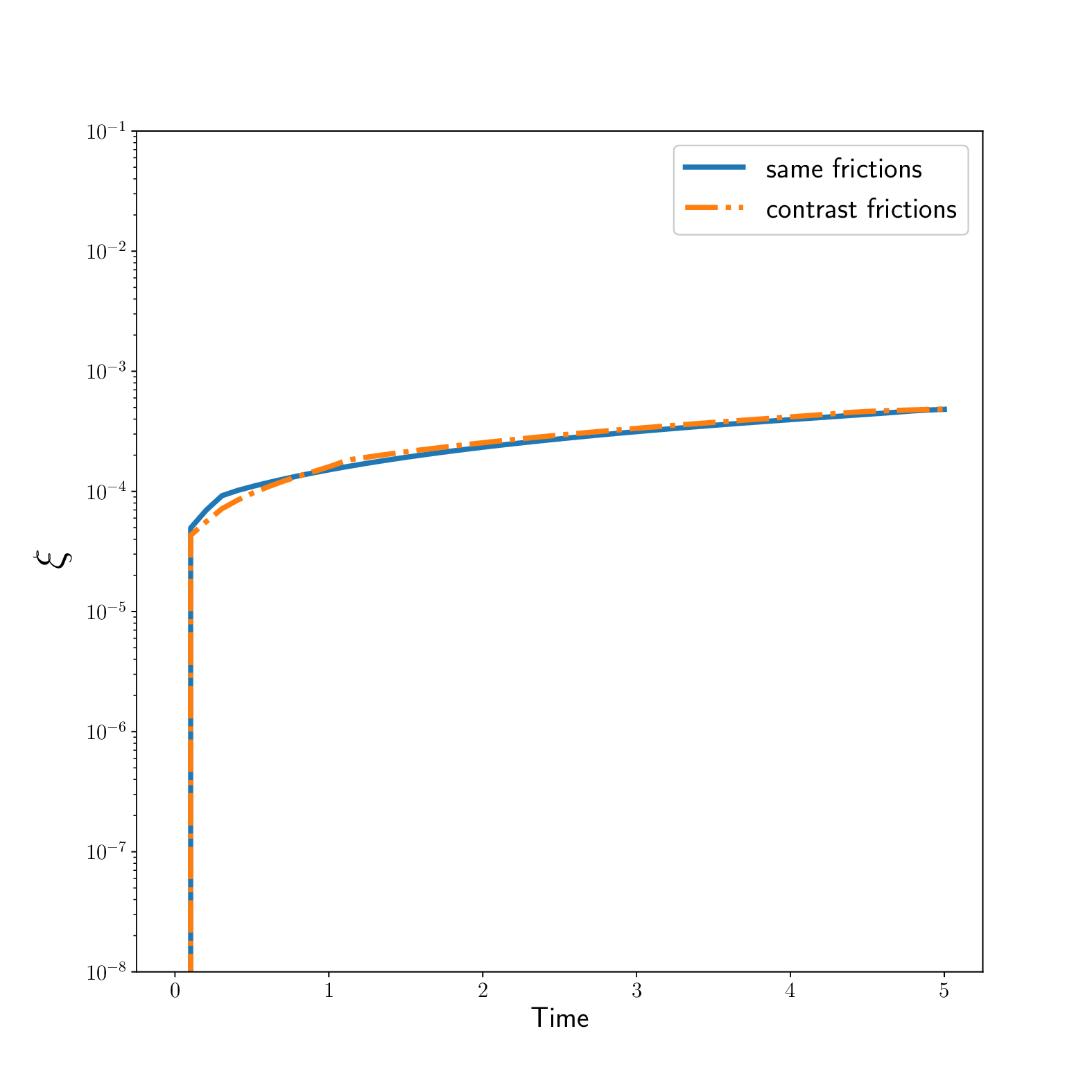}
         \caption{Scalar variable $\xi$. }
         \label{fig:1D_xi_frict}
     \end{subfigure}
     \hfill 
     \begin{subfigure}[b]{0.24\textwidth}
         \centering
         \includegraphics[width = \textwidth]{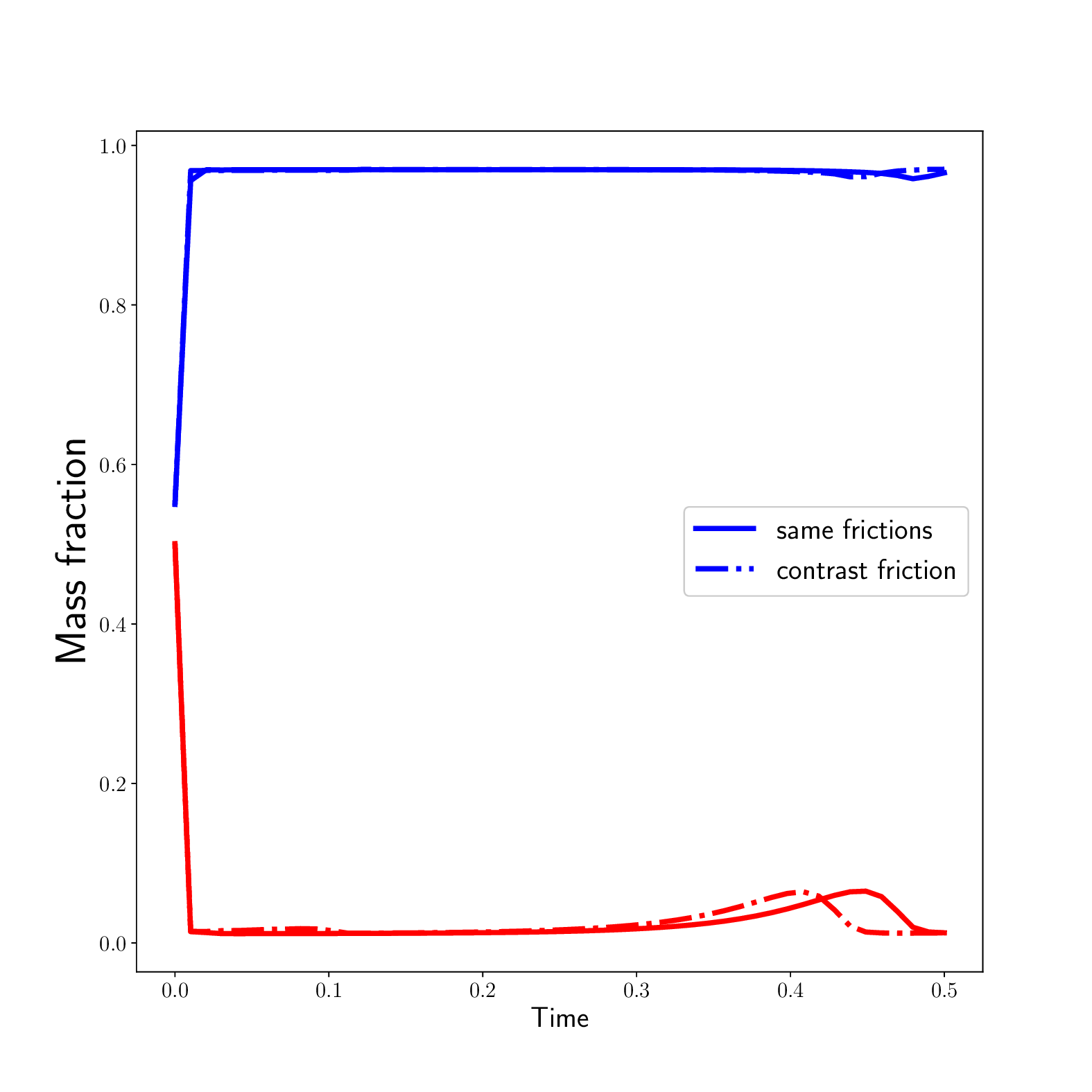}
         \caption{$\min{c}$ and $\max{c}$. }
         \label{fig:1D_maxmin_frict}
     \end{subfigure}\\
        \caption{Temporal evolution of the dissipation of the energy $\frac{\dd E}{\dd t}$, mass of the fluid $1$ given by $\int_\Omega \rho c \,\dd x$, scalar variable $\xi$, and of the minimum and maximal values of the mass fraction $c$ for matching densities (solid lines) and non-matching densities (dash-dotted lines).}
        \label{fig:comparison_properties_friction}
\end{figure}

\subsection{Two-dimensional numerical test cases}

We now simulate the G-NSCH system~\eqref{eq:main1}--\eqref{eq:main4} in two dimensions. Details about the two-dimensional numerical scheme can be found in Appendix~\ref{app:2D_scheme}.

\paragraph{Phase separation with non-matching densities and contrast of friction strengths. }
We use $N_x = N_y = 64$ cells in each direction. We fix the final time at $T=1$.
We set up two simulations, both with no exchange terms $F_c(\rho,c) = 0$. For the first, we consider no contrast of friction effects, \ie $\kappa_1 = \kappa_2 = 10$, $\nu_1 = \nu_2 = 0.01$, and $\eta_1 = \eta_2 = 0.02$.
For the second simulation, we take $\kappa_1 =0$, $\kappa_2 = 10$, $\nu_1 = \nu_2 = 0.01$, and $\eta_1 = \eta_2 = 0.02$. Hence, fluid $2$ has stronger friction effects.

The other parameters are $\gamma = \frac{1}{800}$, $\theta = 4$, $\alpha_1 = 0.8$, $\alpha_2 = 1.2$, $\iota = 10^{-4}$, $C_0=100$, $a=1.5$. The tolerance of GMRES solver is set to $tol = 10^{-10}$. 

The initial velocities in both directions are constants in space $\vecu_x^0 = 0.5$ and $\vecu_y^0 = 0.5$. The initial density is also constant in space $\rho^0= 0.8$. The initial mass fraction is set to a perturbed constant $c^0 = 0.3 - 0.05 r$, with $r$ a random uniform number for each cell center.

Figure~\ref{fig:2D_comparison_solu_frict_source} compares the temporal evolution of the density of fluid $1$ given by $\rho c$ for both cases.  
We observe that both solutions depict phase separation and progressive coarsening of small aggregates into larger ones. This phenomenon occurs as the fluid is transported to the top right corner (we recall that we implemented periodic boundary conditions). Careful inspection of the relative density $\rho c$ distribution inside each aggregates reveals the effect of the contrast of friction between the two solutions. Indeed, as the fluid $1$ encounters a resistance when transported by the flow (because it pushes a fluid that experiences more friction), the mass of fluid $1$ seems to concentrate in the top right corner of each aggregate. This can be observed inspecting the level lines depicted on Figure~\ref{fig:circled}. Indeed on this figure, we see that the top of each aggregate is not localized in the center of the aggregates but is shifted to the top-right.   
This indicates that the 2D scheme captures correctly the effect seen with the 1D numerical scheme when contrast of friction between the two phases is considered. 
\begin{figure}
\centering 
     \begin{subfigure}[b]{0.24\textwidth}
\centering 
\includegraphics[width=0.99\linewidth]{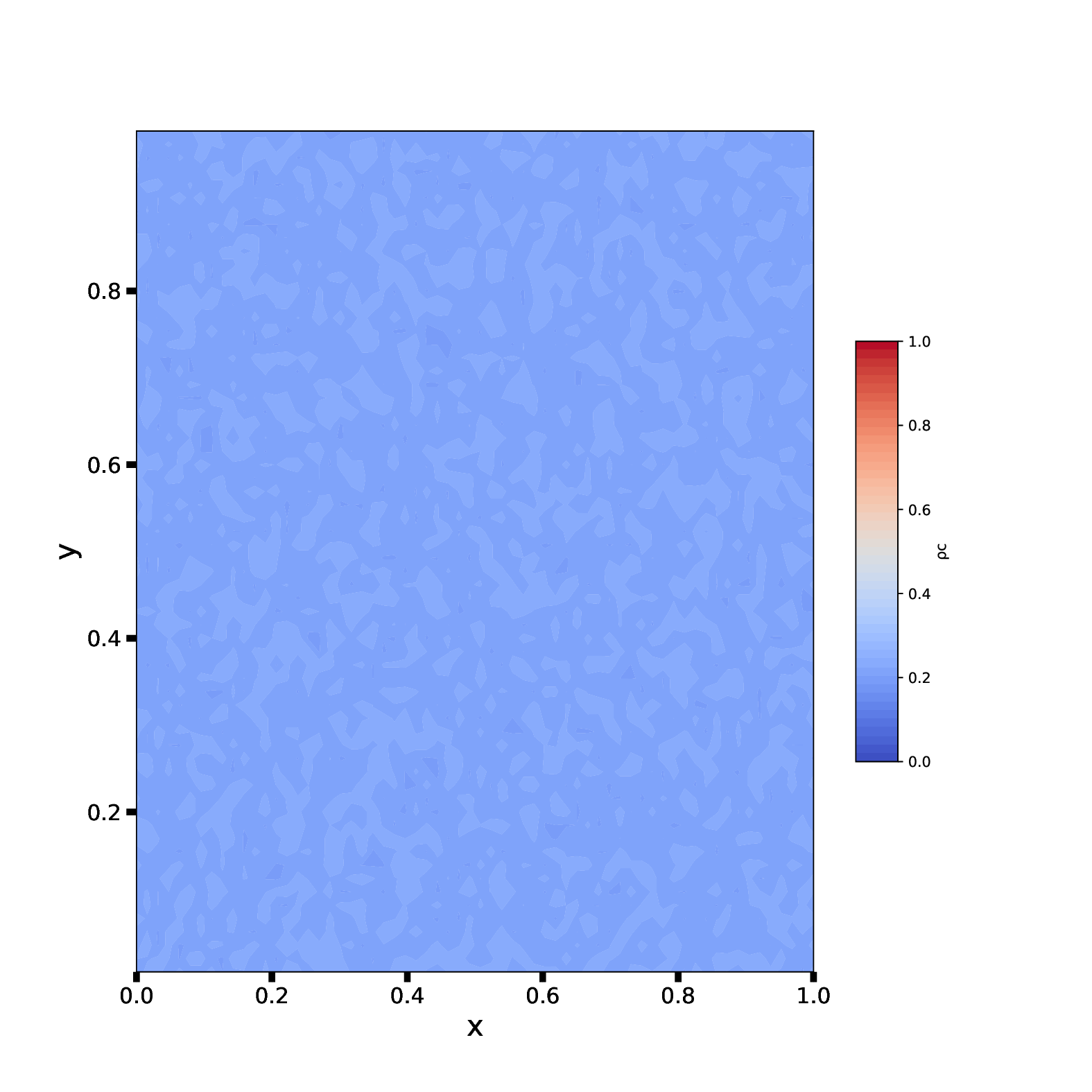}
\caption{$t=0$}
\end{subfigure}
\begin{subfigure}[b]{0.24\textwidth}
\centering 
\includegraphics[width=0.99\linewidth]{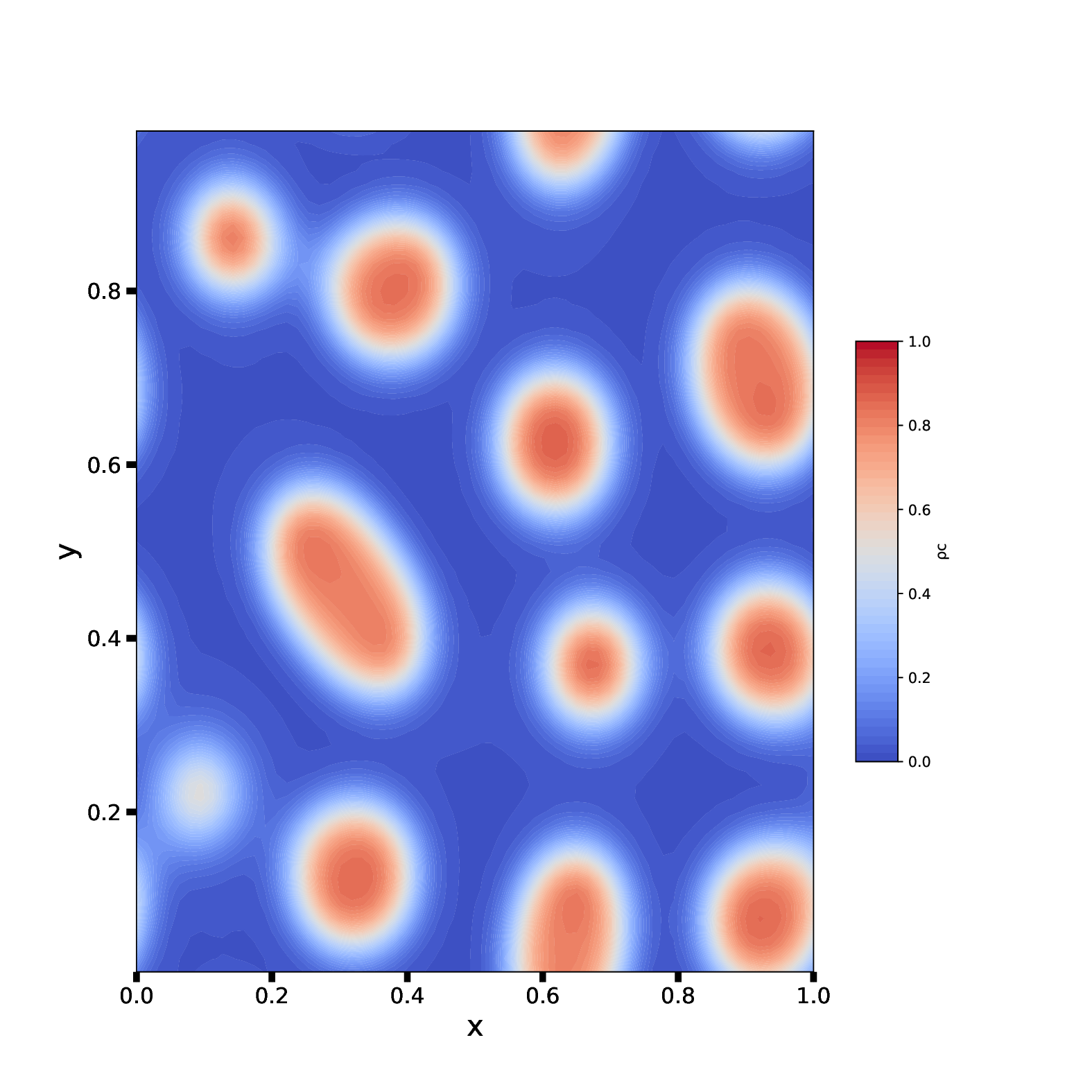}
\caption{$t=0.1$}
\end{subfigure}
\begin{subfigure}[b]{0.24\textwidth}
\centering 
\includegraphics[width=0.99\linewidth]{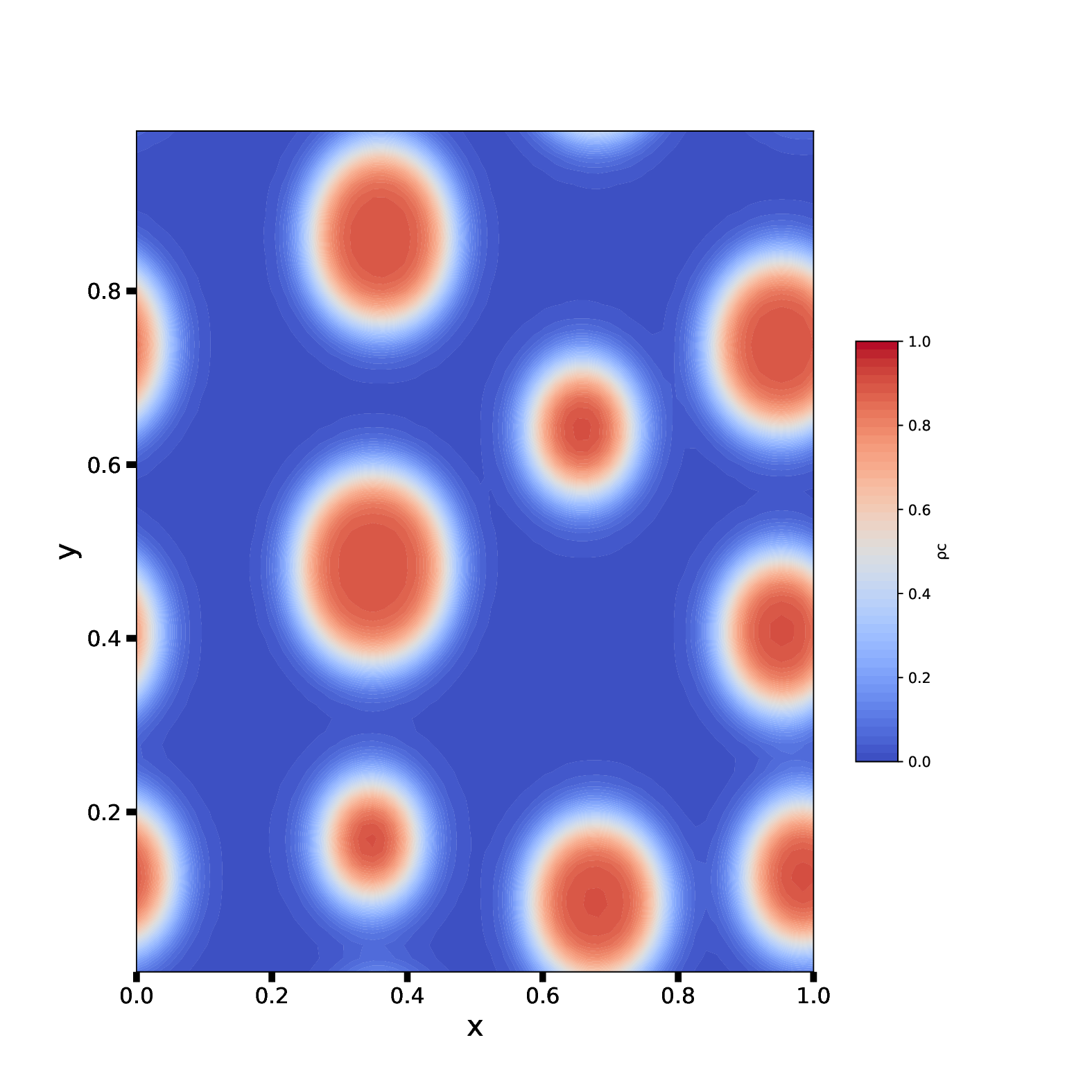}
\caption{$t=0.5$}
\end{subfigure}
\begin{subfigure}[b]{0.24\textwidth}
\centering 
\includegraphics[width=0.99\linewidth]{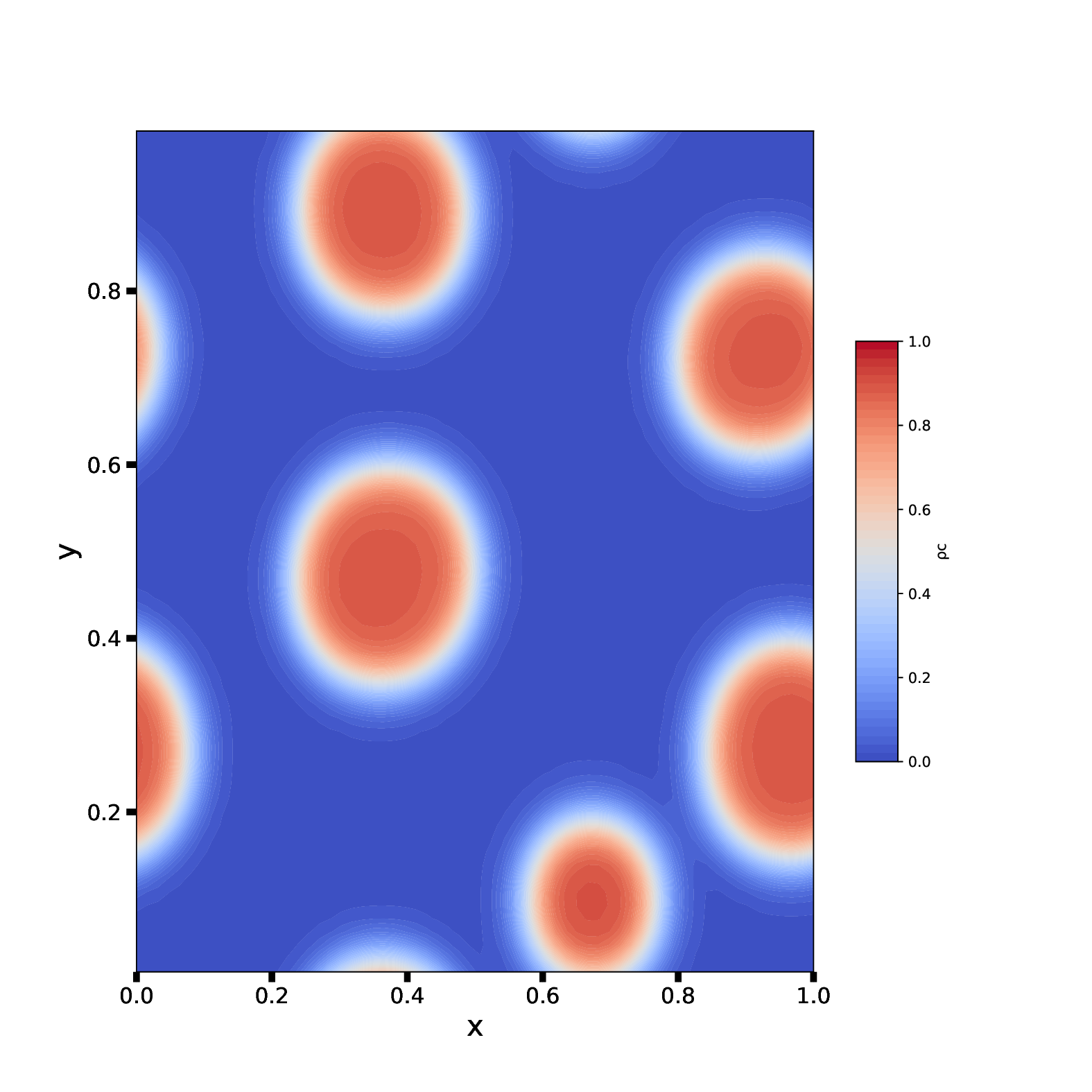}
\caption{$t=1$}
\end{subfigure}\\
\begin{subfigure}[b]{0.24\textwidth}
\centering 
\includegraphics[width=0.99\linewidth]{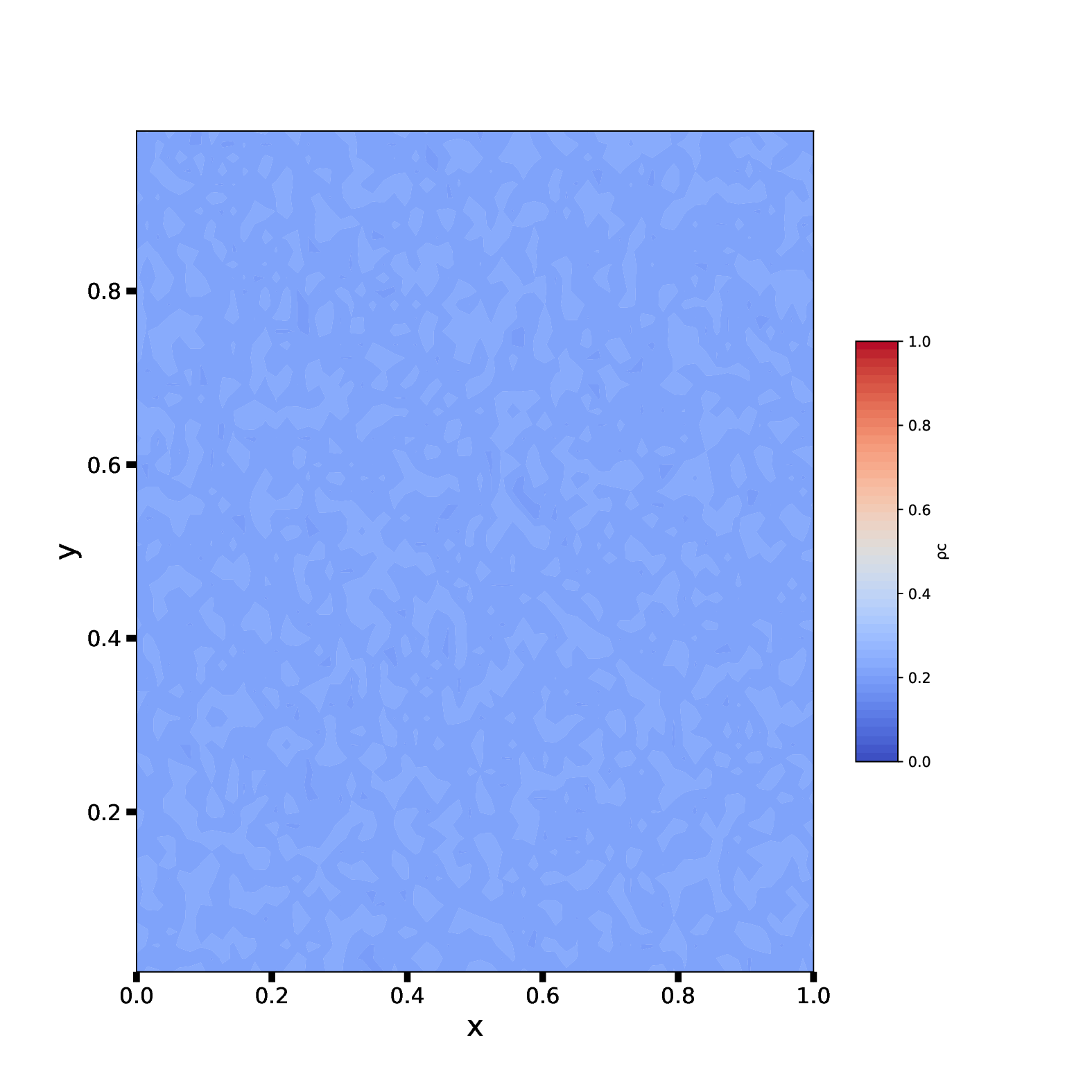}
\caption{$t=0$}
\end{subfigure}
\begin{subfigure}[b]{0.24\textwidth}
\centering 
\includegraphics[width=0.99\linewidth]{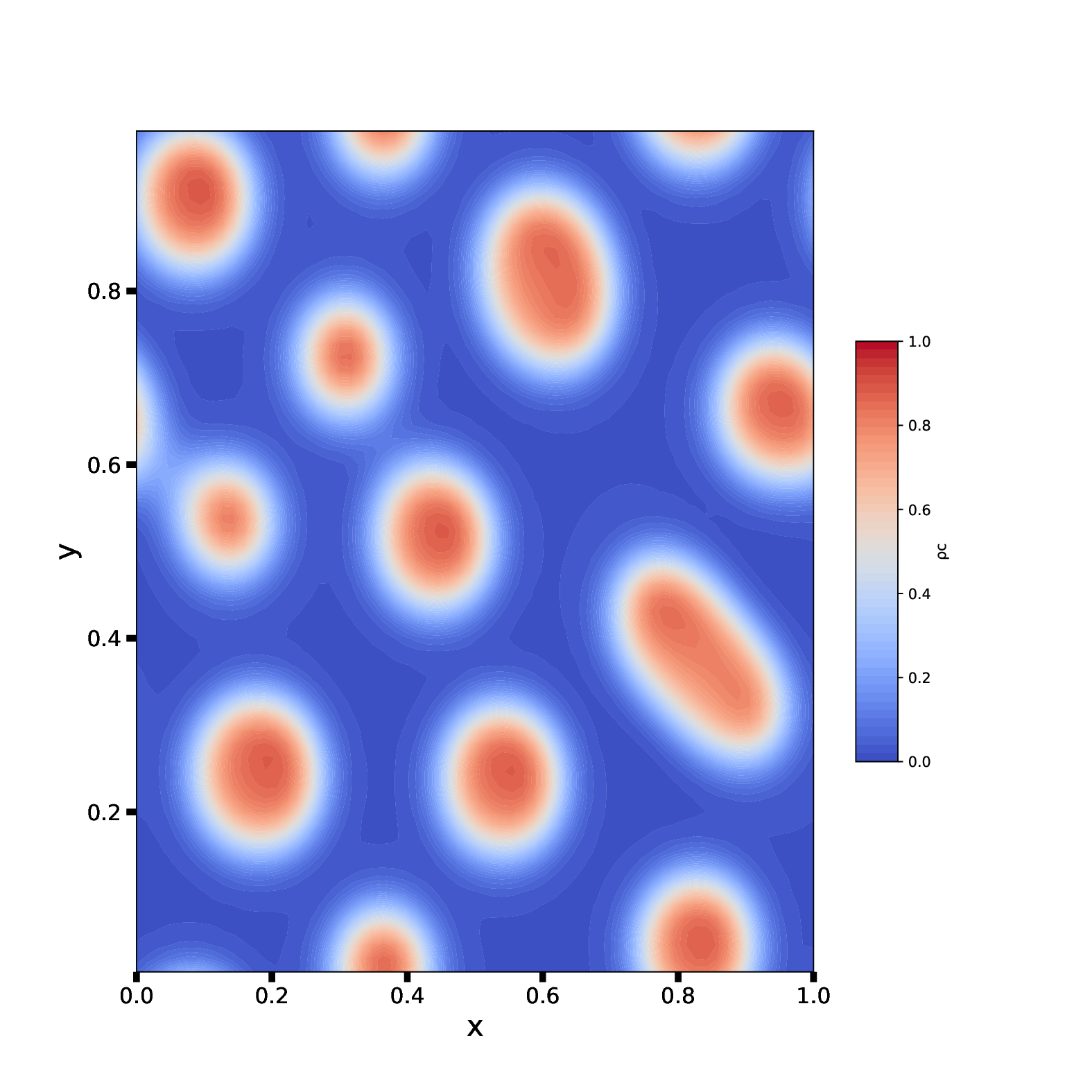}
\caption{$t=0.1$}
\end{subfigure}
\begin{subfigure}[b]{0.24\textwidth}
\centering 
\includegraphics[width=0.99\linewidth]{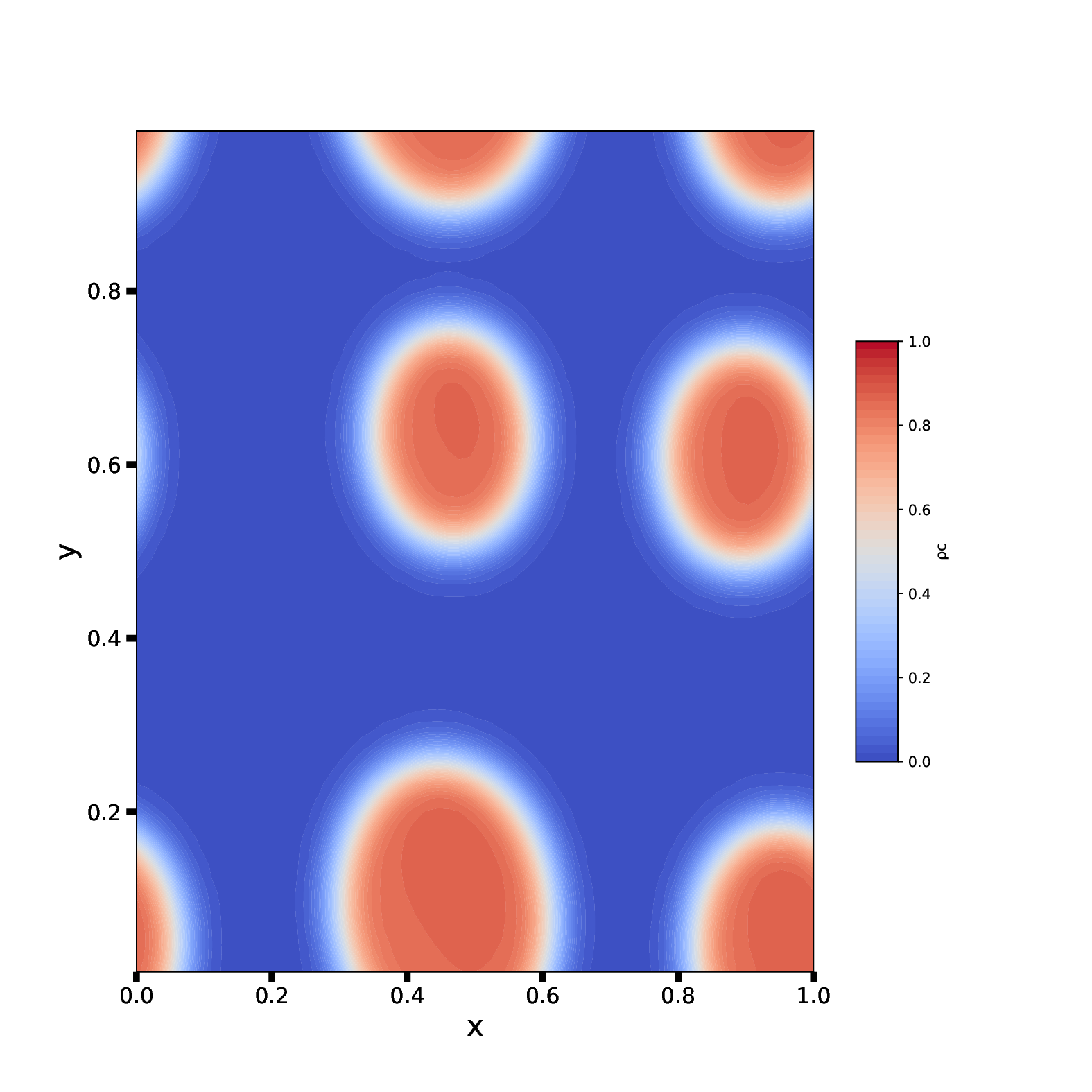}
\caption{$t=1$}
\label{fig:aggregates}
\end{subfigure}
\begin{subfigure}[b]{0.24\textwidth}
\centering 
\includegraphics[width=0.99\linewidth]{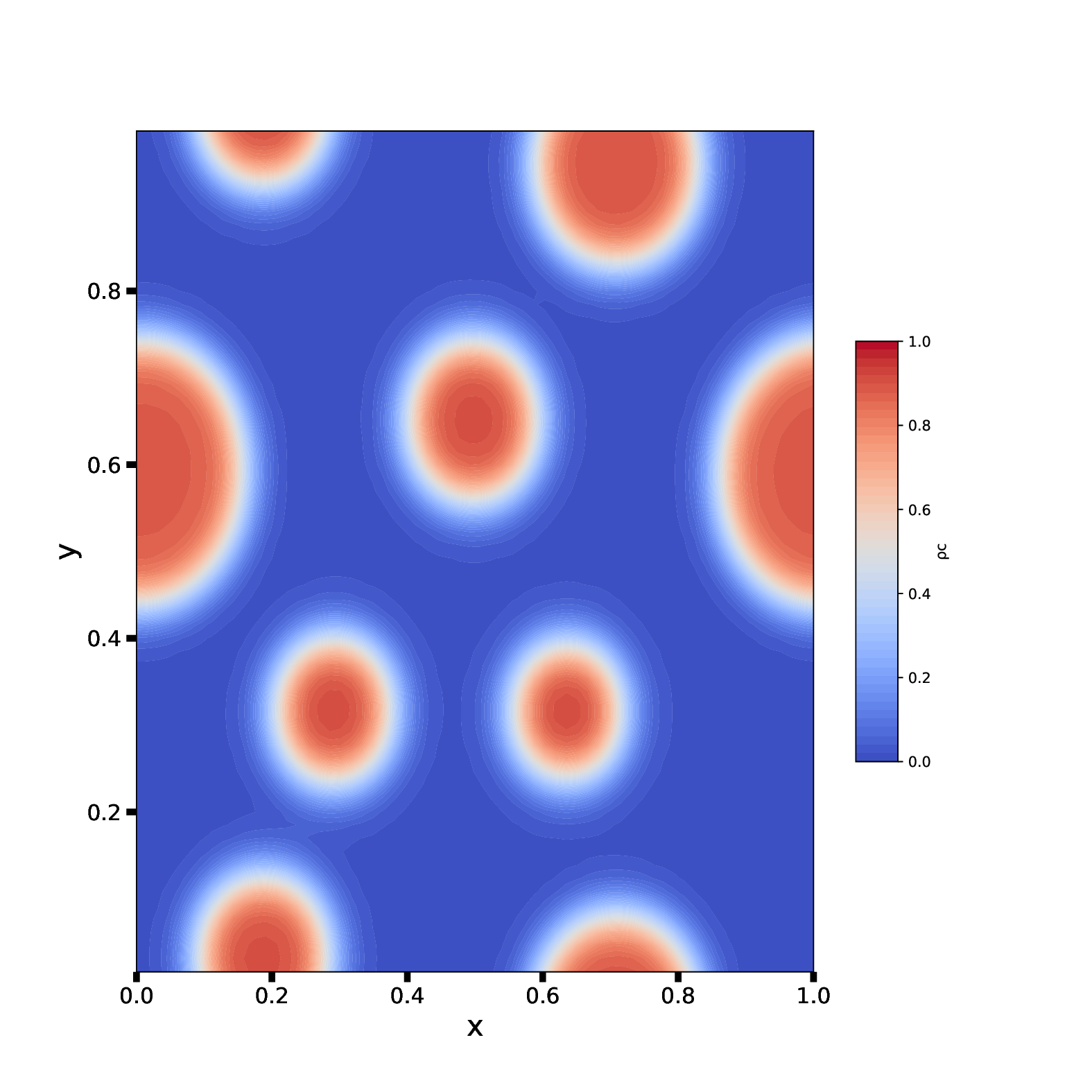}
\caption{$t=5$}
\end{subfigure}
\caption{Two dimensional simulations of compressible Navier-Stokes-Cahn-Hilliard model with nonmatching densities, same friction effects for both fluids (top row) and contrast of friction (bottom row).}
\label{fig:2D_comparison_solu_frict_source}
\end{figure}

\begin{figure}
    \centering
    \begin{subfigure}[b]{0.40\textwidth}
    \includegraphics[width=0.99\linewidth]{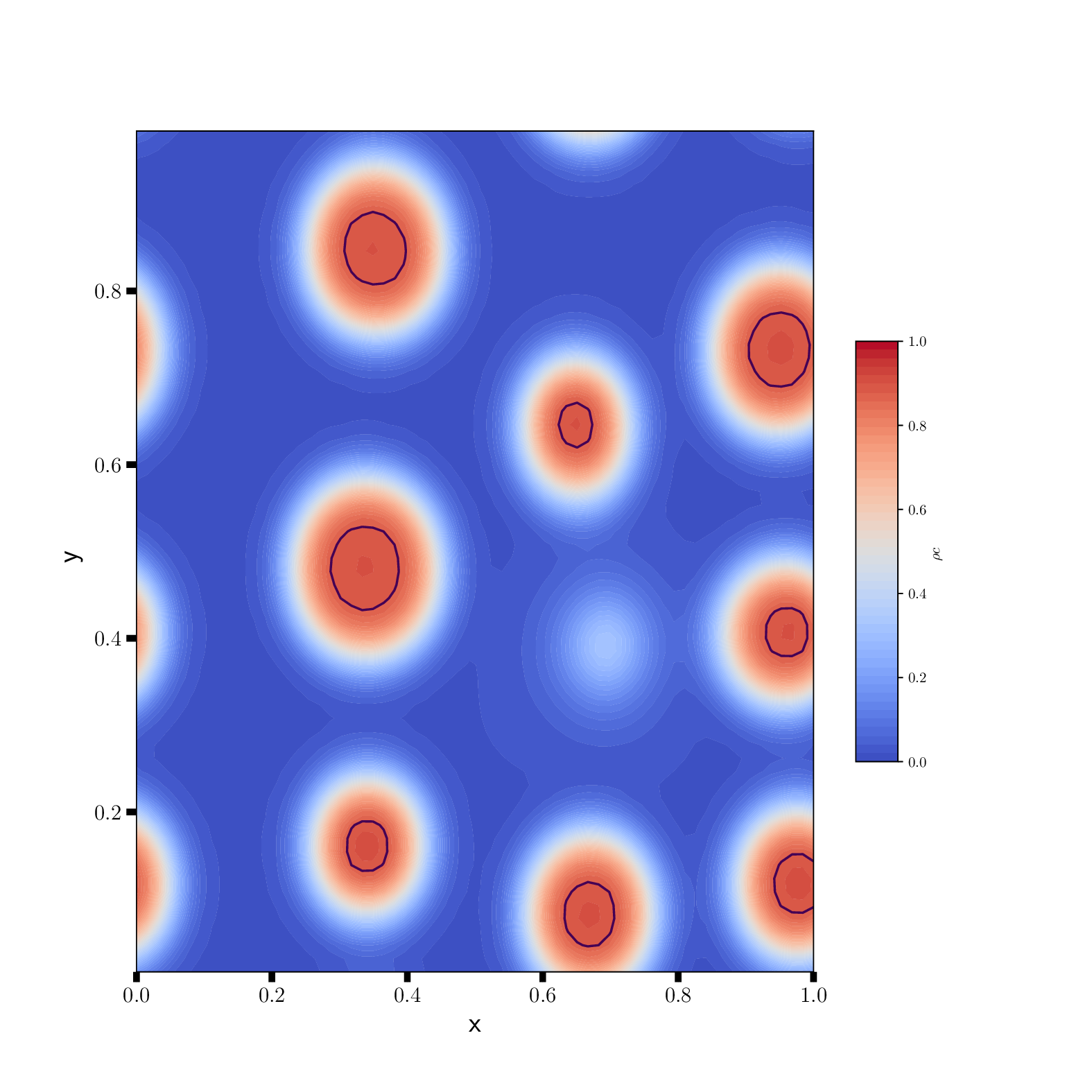}
    \caption{No contrast of friction strengths.}
    \end{subfigure}
    \begin{subfigure}[b]{0.40\textwidth}
    \includegraphics[width=0.99\linewidth]{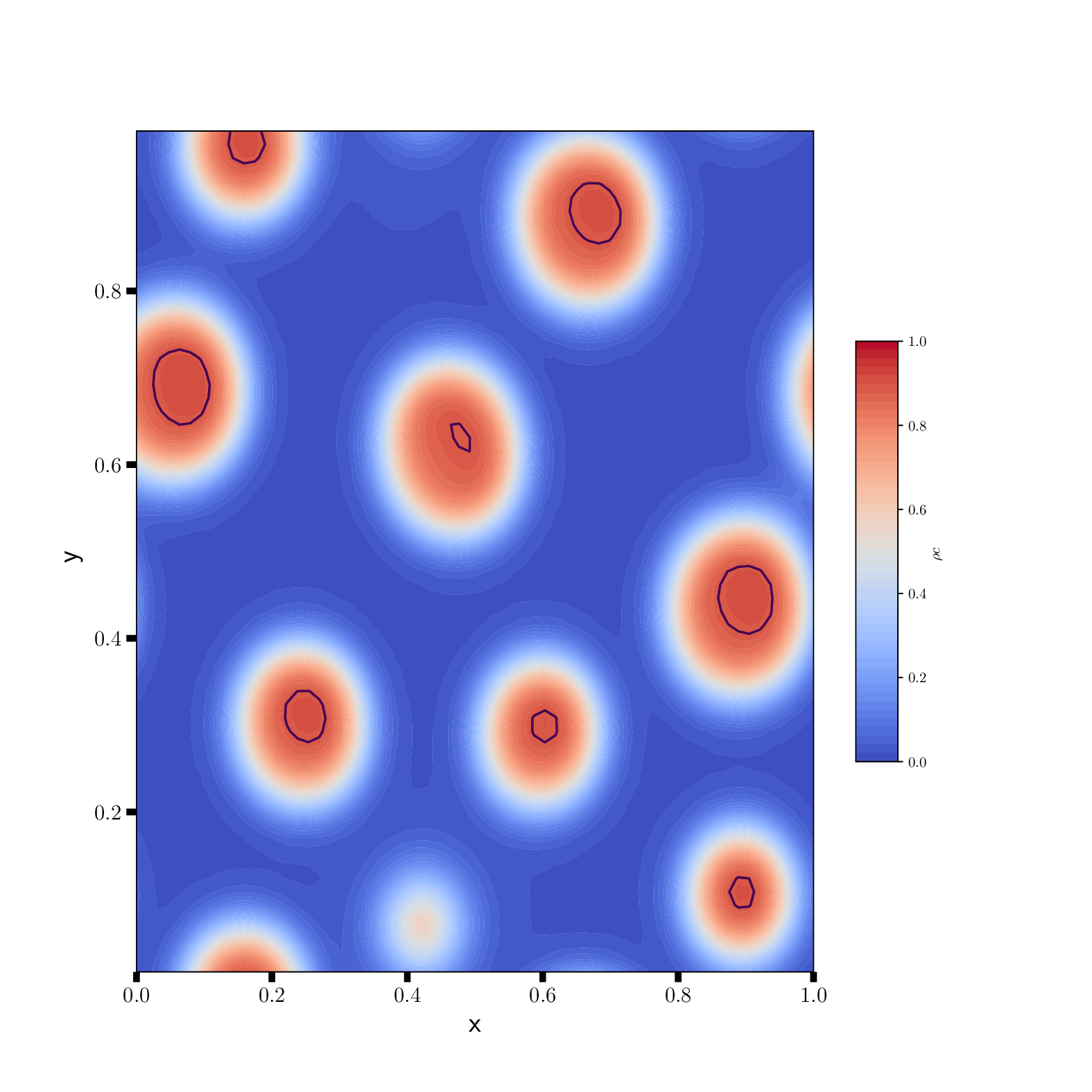}
    \caption{Contrast of friction strengths.}
    \end{subfigure}
    \caption{Relative density of fluid $1$ at time $t=0.25$ considering a contrast of friction strengths between the two fluids. The black circles represent the level $\rho c = \max(\rho c) - 0.03$. This corresponds to the tops of each aggregate.}
    \label{fig:circled}
\end{figure}

\subsection{Convergence tests}

We study the numerical convergence of the one dimensional scheme~\eqref{eq:discrete1D1}--\eqref{eq:discrete1D11} with $\kappa(\rho,c) = 0$ and $F_c(\rho,c) = 0$. The computational domain is $\Omega = (0,1)$. The final time is $T = 0.05$. The other parameters $\gamma$, $\beta$, $\eta$, $\nu$, $a$, $\alpha_1$, $\alpha_2$ are chosen as for the previous 1D test case with non-matching densities (see Subsection~\ref{subsec:testcase1}). The initial condition for the mass fraction is given by 
\[
c^0 = 0.4 + 0.01\cos(6\pi x).  
\]
The initial conditions for the velocity and the total density are chosen as in the previous 1D test cases.

We set the tolerance $\text{rtol}$ of the GMRES algorithm to $\text{rtol} = 10^{-6}$ for the spatial convergence test and $\text{rtol} = 10^{-10}$ for the temporal convergence test. 

\subsubsection{Convergence in space} \label{subsec:spatial-conv}
We fix the time step to $\dt = 1 \times 10^{-5}$ and we vary the grid size. We choose an increasing number of cells $N_x = \{64, 128, 256, 512, 1024, 2048 \}$. 
For each quantity $c,\vecu, \rho$, we compute the discrete errors
\begin{equation}
\begin{aligned}
\text{error}(\rho_\dx, \rho_{\dx/2}) &= \norm{\rho_\dx - \rho_{\dx/2}}_{L^{\infty}(0,T;L^a(\Omega))}, \\
\text{error}(c_\dx, c_{\dx/2}) &= \norm{c_\dx - c_{\dx/2}}_{L^{2}(0,T;L^2(\Omega))}, \\
\text{error}(\vecu_\dx, \vecu_{\dx/2}) &= \norm{\vecu_\dx - \vecu_{\dx/2}}_{L^{2}(0,T;L^2(\Omega))},
\end{aligned}
\label{eq:computerror}
\end{equation}
where $( \rho_{\dx/2},  c_{\dx/2},  \vecu_{\dx/2})$ denotes the solution computed using twice the number of cells of the simulation that computes the solution  $( \rho_{\dx},  c_{\dx},  \vecu_{\dx})$.

To compute the discrete norms, we save the solution every $\dt_\text{save} = 0.001$. The norms in~\eqref{eq:computerror} are computed following
\[
\begin{aligned}
\norm{\rho_\dx - \rho_{\dx/2}}_{L^{\infty}(0,T;L^a(\Omega))} &= \max_{t_\text{save}}\left(\frac{\dx}{2} \sum_{j=1}^{N_x} \left(\rho_\dx(x_j)  - \rho_{\dx/2}(x_j)\right)^a \right)^{1/a} , \\
\norm{c_\dx - c_{\dx/2}}_{L^2(0,T;L^2(\Omega))} &=\left( \sum_{t_\text{save}}\dt_\text{save} \left( \left(\frac{\dx}{2} \sum_{j=1}^{N_x} \left(c_\dx( x_j ) - c_{\dx/2}(x_j)\right)^2 \right)^{1/2} \right)^2 \right)^{1/2},\\
\norm{\vecu_\dx - \vecu_{\dx/2}}_{L^2(0,T;L^2(\Omega))} &=\left( \sum_{t_\text{save}}\dt_\text{save} \left( \left(\frac{\dx}{2} \sum_{j=1}^{N_x} \left(\vecu_\dx( x_j ) - \vecu_{\dx/2}(x_j)\right)^2 \right)^{1/2} \right)^2 \right)^{1/2},
\end{aligned}
\]
with $N_x$ the number of points on the $\dx/2$ grid, and $t_\text{save}$ the array of times at which snapshots of the solutions have been taken. Hence, the solution on the coarse grid $\dx$ is extended on the fine grid $\dx/2$ using the nearest solution from the coarse grid.

We arrive at the results given in Figure~\ref{fig:spatialconv}. As expected by the upwind scheme, the spatial order of convergence is a little less than 1 for the total density $\rho$ (see Figure~\ref{fig:spatialconv-rho}) and the velocity $\vecu$ (see Figure~\ref{fig:spatialconv-u}). We recover first order for the mass fraction. 

\begin{figure}
\centering
\begin{subfigure}[b]{0.32\textwidth}
    \centering
    \includegraphics[width = 0.99\linewidth]{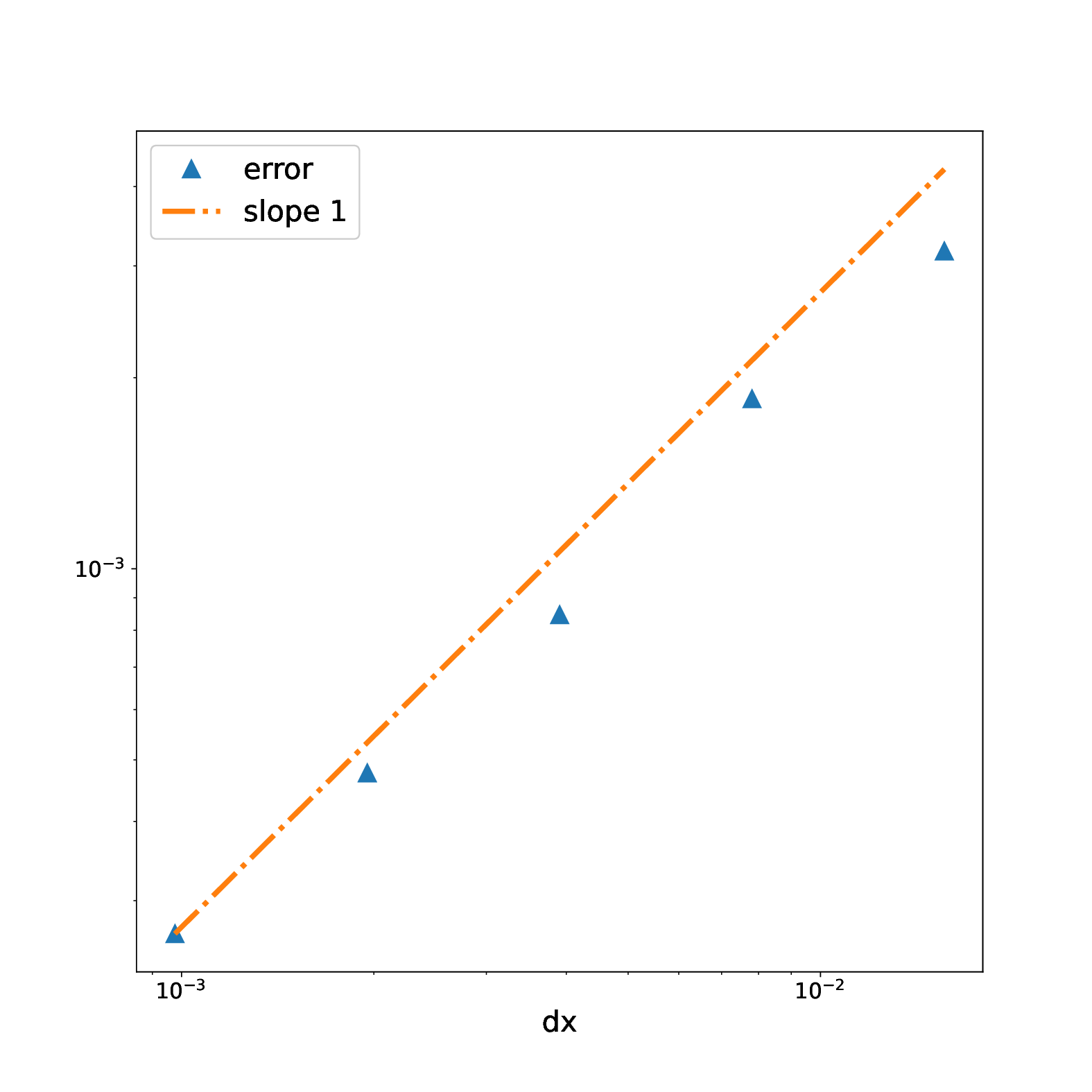}
    \caption{Spatial convergence of density $\rho$.}
    \label{fig:spatialconv-rho}
\end{subfigure}
\hfill
\begin{subfigure}[b]{0.32\textwidth}
    \centering
    \includegraphics[width = 0.99\linewidth]{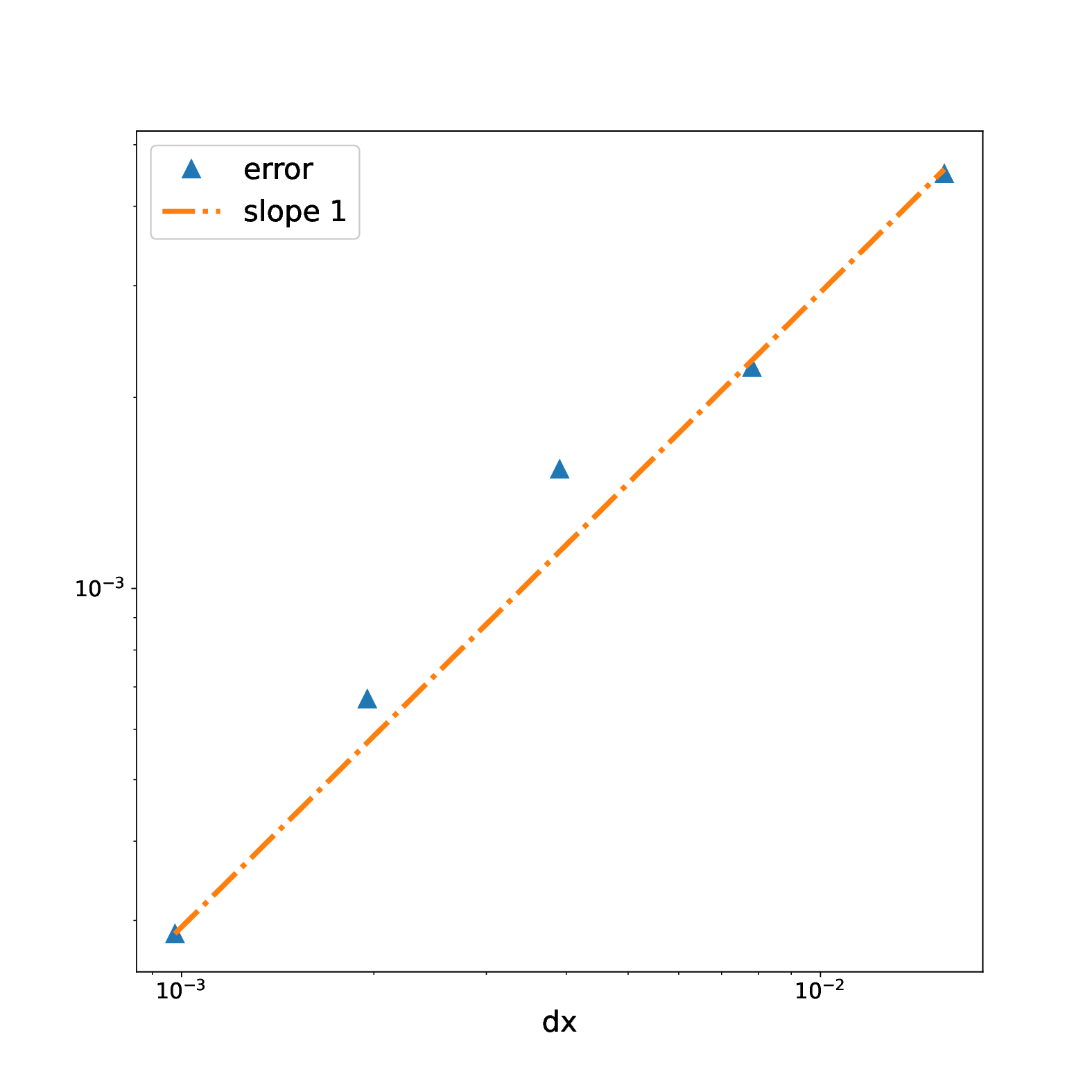}
    \caption{Spatial convergence of mass fraction $c$.}
    \label{fig:spatialconv-c}
\end{subfigure}
\hfill
\begin{subfigure}[b]{0.32\textwidth}
    \centering
    \includegraphics[width = 0.99\linewidth]{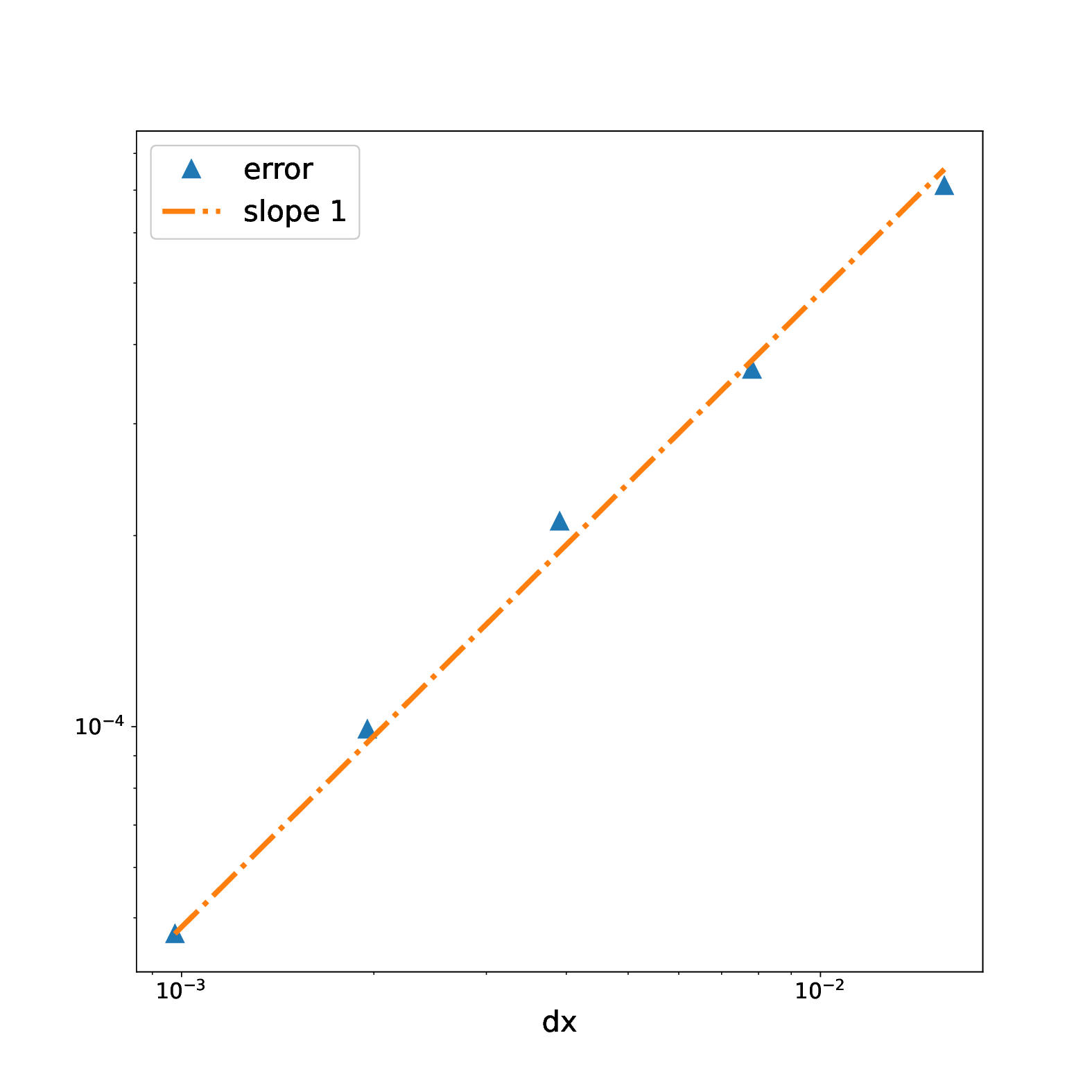}
    \caption{Spatial convergence of velocity $\vecu$.}
    \label{fig:spatialconv-u}
\end{subfigure}
    \caption{Convergence in space for the total density $\rho$, the mass fraction $c$ and the velocity $\vecu$. The orange dashed line represents the slope 1.}
    \label{fig:spatialconv}
\end{figure}

\subsubsection{Convergence in time}
We here fix the grid size and select $N_x = 128$ points. We choose $\dt = 1 \times 10^{-4}$, and decrease the time steps according to $\dt_\text{array} = \{\dt, \f\dt 2, \f \dt 4, \f \dt 8, \f \dt {16}, \f \dt {32}, \f {\Delta t}{64} \}$. We deactivate the time step adaptive strategy from the CFL condition. The other parameters and initial conditions are chosen as in the spatial convergence test (see Subsection~\ref{subsec:spatial-conv}). To check the convergence in time of our scheme, we compute the errors between two solutions computed with two time steps $\dt$ that differ only from a factor $\frac{1}{2}$. We denote these two different solutions by $(\rho_\dt,c_\dt,\vecu_\dt)$ and $(\rho_{\dt/2},c_{\dt/2},\vecu_{\dt/2})$. We use the same method as for the spatial convergence computations, we save the solutions every $\dt_\text{save} = 0.001$. We compute the norms
\begin{equation}
\begin{aligned}
\text{error}(\rho_\dt, \rho_{\dt/2}) &= \norm{\rho_\dt - \rho_{\dt/2}}_{L^{\infty}(0,T;L^a(\Omega))}, \\
\text{error}(c_\dt, c_{\dt/2}) &= \norm{c_\dt - c_{\dt/2}}_{L^{2}(0,T;L^2(\Omega))}, \\
\text{error}(\vecu_\dt, \vecu_{\dt/2}) &= \norm{\vecu_\dt - \vecu_{\dt/2}}_{L^{2}(0,T;L^2(\Omega))}.
\end{aligned}
\end{equation}
with
\[
\begin{aligned}
\norm{\rho_\dt - \rho_{\dt/2}}_{L^{\infty}(0,T;L^a(\Omega))} &= \max_{t_\text{save}}\left(\dx \sum_{j=1}^{N_x} \left(\rho_\dt(x_j)  - \rho_{\dt/2}(x_j)\right)^a \right)^{1/a} , \\
\norm{c_\dt - c_{\dt/2}}_{L^2(0,T;L^2(\Omega))} &=\left( \sum_{t_\text{save}}\dt_\text{save} \left( \left(\dx \sum_{j=1}^{N_x} \left(c_\dt( x_j ) - c_{\dt/2}(x_j)\right)^2 \right)^{1/2} \right)^2 \right)^{1/2},\\
\norm{\vecu_\dt - \vecu_{\dt/2}}_{L^2(0,T;L^2(\Omega))} &=\left( \sum_{t_\text{save}}\dt_\text{save} \left( \left(\dx \sum_{j=1}^{N_x} \left(\vecu_\dt( x_j ) - \vecu_{\dt/2}(x_j)\right)^2 \right)^{1/2} \right)^2 \right)^{1/2},
\end{aligned}
\]
We obtain the results depicted in Figure~\ref{fig:tempconv}
\begin{figure}
\centering
\begin{subfigure}[b]{0.32\textwidth}
    \centering
    \includegraphics[width = 0.99\linewidth]{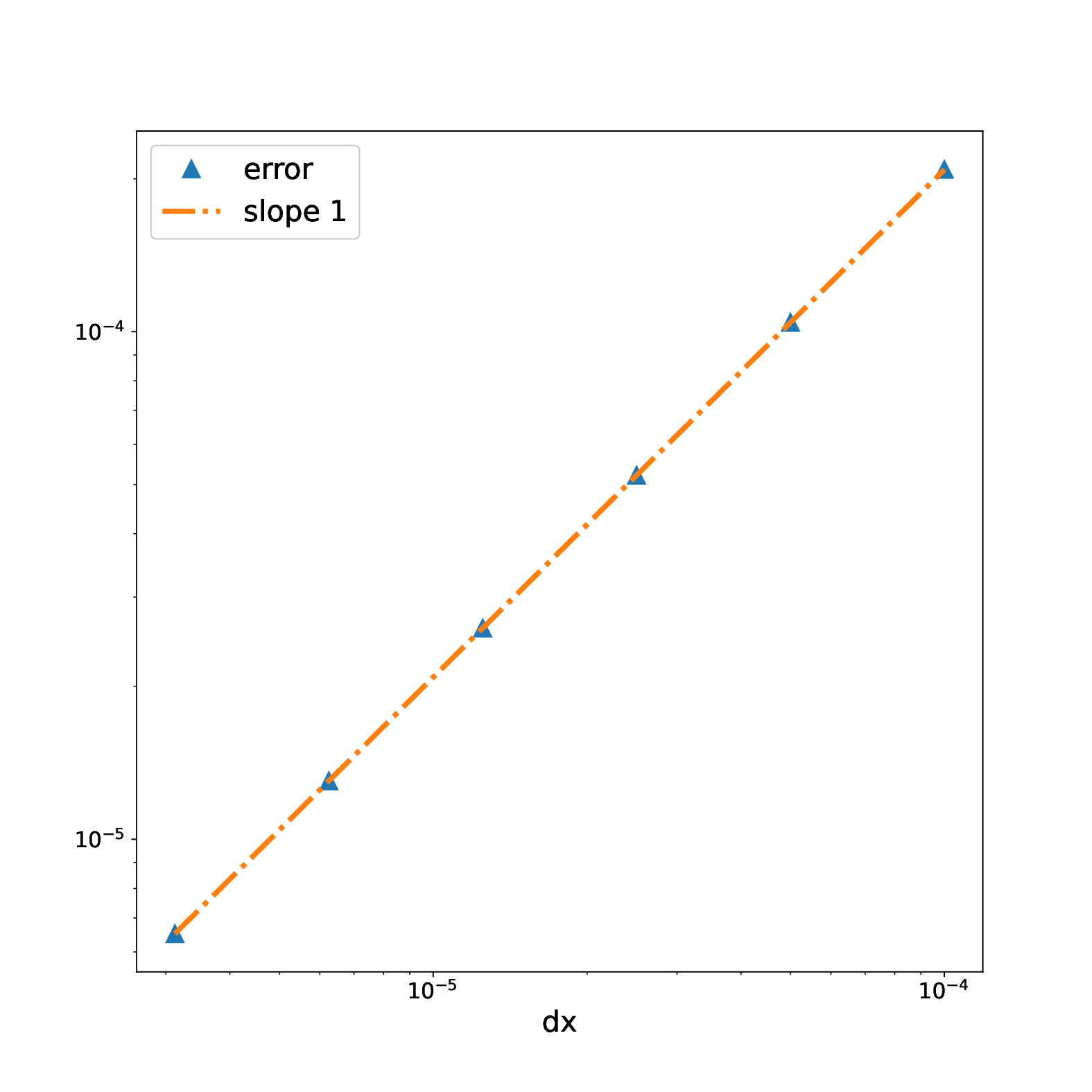}
    \caption{Temporal convergence of density $\rho$.}
    \label{fig:tempconv-rho}
\end{subfigure}
\hfill
\begin{subfigure}[b]{0.32\textwidth}
    \centering
    \includegraphics[width = 0.99\linewidth]{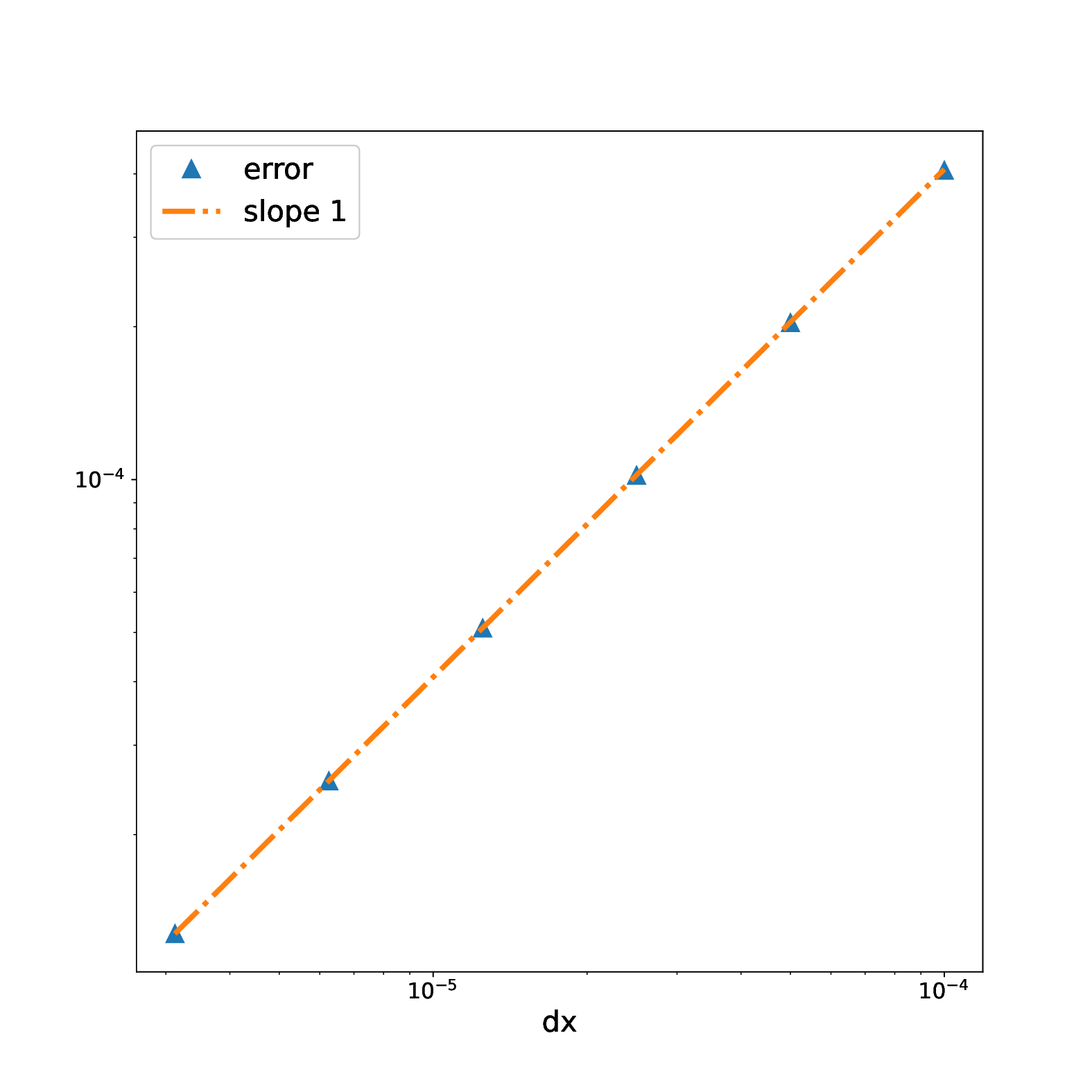}
    \caption{Temporal convergence of mass fraction $c$.}
    \label{fig:tempconv-c}
\end{subfigure}
\hfill
\begin{subfigure}[b]{0.32\textwidth}
    \centering
    \includegraphics[width = 0.99\linewidth]{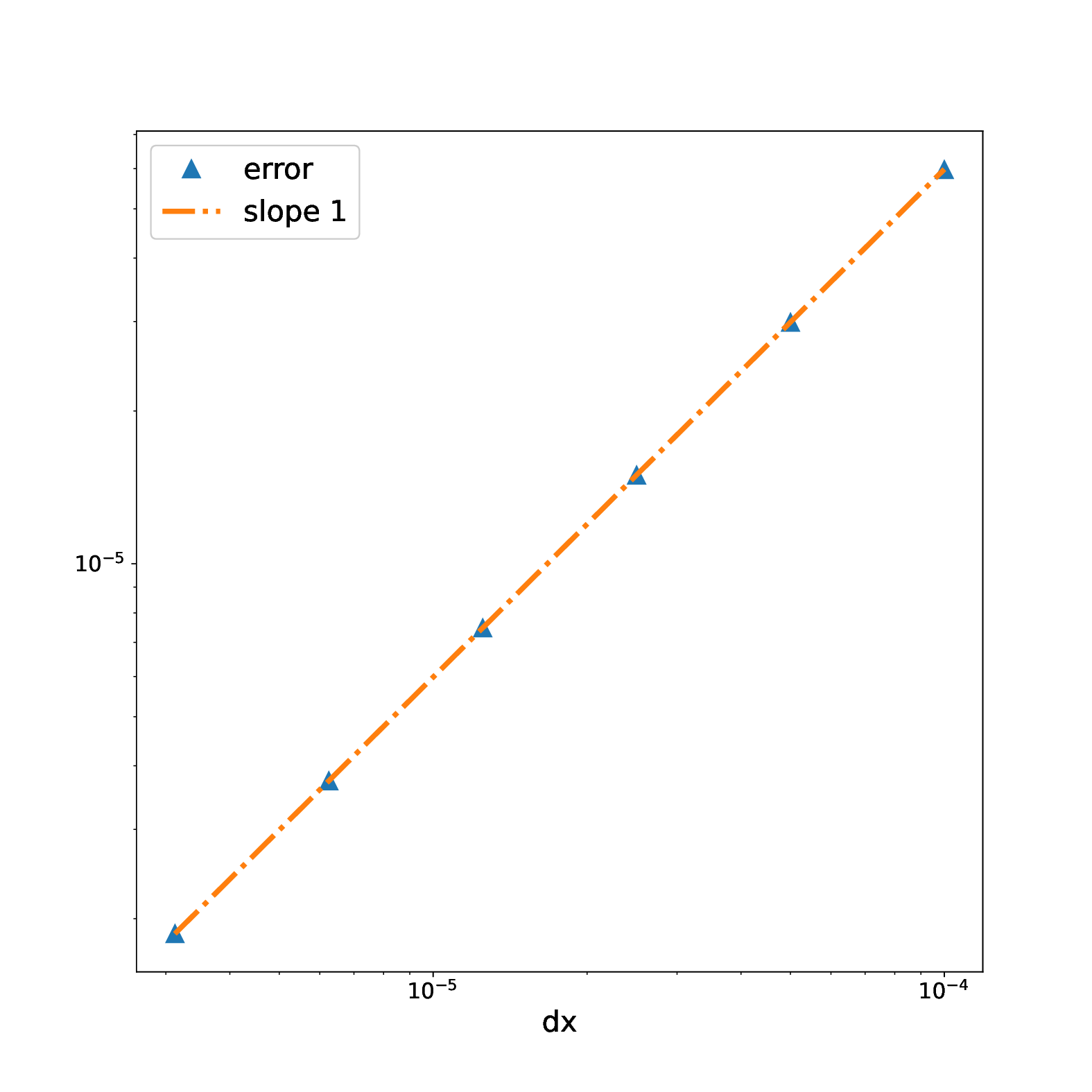}
    \caption{Temporal convergence of velocity $\vecu$.}
    \label{fig:tempconv-u}
\end{subfigure}
    \caption{Convergence in time for the total density $\rho$, the mass fraction $c$ and the velocity $\vecu$. The orange dashed line represents the slope 1. }
    \label{fig:tempconv}
\end{figure}
We observe that the order of convergence in time for our scheme is exactly $1$ for the three quantities.

\section{Conclusion and perspectives}

We presented a generalized model of diphasic compressible fluid termed G-NSCH, that comprises possible mass transfer between the two phases and friction effects. Under simplifying assumptions, summarized in Section~\ref{subsec:assumptions}, we proved the existence of weak solutions of the G-NSCH system. We also proposed a numerical scheme and prove, under the same simplifying assumptions, that it is stable and structure preserving (\ie it ensures the physically relevant bounds for the mass fraction $c$, and it satisfies an energy dissipation inequality). For the numerical simulations, we chose relevant functionals, thus, relaxing the simplifying assumptions that were necessary for the analysis. We presented numerical simulations showing that our numerical scheme possesses the robustness found analytically. 
The numerical simulations allowed us to show the ability of our model to represent diphasic fluids with matching or non-matching densities for the two phases.   
Furthermore, we computed numerically the spatial and temporal convergences our numerical scheme.
 
Our model and numerical scheme allow to use physically relevant choice of functional and to consider contrast of properties between the two phases of the fluid. Our aim is to perform efficient simulation of general compressible diphasic fluids while being able to capture instabilities that could emerge considering contrasts of properties such as Saffman-Taylor or Rayleigh-Taylor instabilities.
However, we emphasize that to achieve this latter goal, we have to be able to capture accurately the possible fine structures appearing during the numerical simulations. 
We plan to improve our numerical scheme in several ways. First, we plan to adapt the Relaxed version of the Generalized SAV method. Indeed, we observed during our numerical simulation that the variable $\xi$ is between $10^{-5}$ and $10^{-3}$. However, with the transfer term $F_c(\rho,c) \neq 0$, as the simulation progresses, the gap between $r$ and the real energy increases, \ie $\xi$ becomes larger. As shown in~\cite{Yanrong_2022_generalizedSAV}, this problem is solved with the relaxed G-SAV method and the use of this method is one of our further developments. In the same work~\cite{Yanrong_2022_generalizedSAV}, it is mentioned that this relaxed method works well even when an \textit{external force} is comprised in the model. In our case, this \textit{external force} will take the form of a mass transfer or mass source term.  
The second improvement concerns the accuracy of our scheme. As shown in our work, the temporal and spatial orders of our scheme do not exceed $1$. 
Thus, we will aim to design a high-order finite element scheme for the generalized compressible NSCH system that will remain structure preserving taking advantage of the flexibility of the relaxed GSAV method. 
On another aspect, we plan to use the reduced version of the G-NSCH model presented in the Appendix~\ref{sec:assumptions-functions-2pop} to represent tumor growth while removing non-necessary effects such as inertia. 
Our goal is to present a model and numerical simulations capturing Saffman-Taylor-like instabilities depicted by the protrusions of the tumor in the healthy tissue and commonly observed in the context of, \eg, skin cancer~\cite{chatelain2011emergence}. 
Furthermore, analytical aspects of this work can also be improved. This direction is challenging because as pointed out in the present work, necessary tools to perform the solutions' existence proof do not work with physically or biologically relevant potentials or mobility functions. In fact, singular potentials, degenerate mobilities and degenerate viscosity functions are not allowed. 
One possible solution is to derive a Bresch-Desjardins entropy estimate~\cite{Bresch_2003_shallow,Bresch_2006_shallowNS} for the compressible NSCH as it has been done recently by Vasseur and Yu~\cite{vasseur_2016_weaksolwithdegvisco}, and Bresch, Vasseur and Yu~\cite{Bresch_2022_NSnonlinearvisco}, for the compressible Navier-Stokes model with degenerate viscosities.   

To conclude, we emphasize that the G-NSCH model is the basis of a reduced system that takes into account only biologically relevant physical effects that play a role in tumor evolution (presented in the present article in Appendix~\ref{sec:assumptions-functions-2pop} as \textit{Problem 2}). Therefore, this work has to be seen as the first part. In a subsequent work, relying heavily on the present one, we will focus on numerical simulations, and sensitivity analysis of the reduced model.

\section*{Acknowledgements}
The authors would like to thank Tommaso Lorenzi for his comments concerning the derivation of our G-NSCH model and for the very interesting discussions we had about the modelling of tumor growth. The authors would like also to thank Alain Miranville for fruitful discussions about the compressible Navier-Stokes-Cahn-Hilliard model and diphasic fluid dynamics.  

\bibliographystyle{siam}
\bibliography{final}

\appendix

\section{Derivation of the model}
\label{sec:derivation-model-2pop}
%Let us describe how the model is obtained and how it is consistent with the second law of thermodynamics and continuum mechanics. In this sections the two cell populations are modelled by the two fluids. We sometime also refer to the two phases of the fluid for the two populations. The mixture is supposed to be viscous with a possible contrast of viscosity between the two phases. 
%Since we aim at representing cells evolving in the extra-cellular matrix (ECM), the mixture moves in a porous domain. Cells of the two populations crawl in the ECM and moves at different speeds depending on their type and metabolism. This situation is represented by the fact that the two fluids exert a friction at the surface of the pores. This effect is taken into account in the equation for the velocity field.  
In this Appendix, we present the rigourous derivation of our G-NSCH model.

\subsection{Notation and definitions}
We formulate our problem in Eulerian coordinates and in a smooth bounded domain $\Omega \subset \R^d$ (where $d=\{1,2,3\}$ is the dimension). The balance laws derived in the following sections are in local form.

We have two fluids in the model where $\rho_1,\rho_2$ are the relative densities of respectively fluid $1$ and $2$. Thus, $\rho_i$ represents the mass $M_i$ of the fluid per volume occupied by the $i$-th phase $V_i$, \ie 
\[
    \rho_i = \f{M_i}{V_i}.
\]  
Then, we define the volume fractions $\vp_1,\vp_2$ which are defined by the volume occupied by the $i$-th phase over the total volume of the mixture
\[
    \vp_i = \f{V_i}{V}.    
\] 
Therefore, the mass density of population $i$ which is the mass of population $i$ in volume $V$ is given by 
\[
    \phi_i = \rho_i\vp_i.  
\]
We further assume that the fluid is saturated, \ie  
\[
    \vp_1 + \vp_2 = 1.
\]
The total density of the mixture is then given by 
\[
    \rho = \phi_1 + \phi_2.  
\]
We also introduce the mass fractions $c_i=M_i/M$ and we have the relations
\begin{equation}
    \rho c_i = \phi_i, \quad \text{and} \quad c_1 = (1-c_2).
    \label{eq:relations-mass-frac}
\end{equation}
We denote by $p$ the pressure inside the mixture and $\vecu_1,\vecu_2$ are the velocities of the different phases. We use a mass-average mixture velocity 
\begin{equation}
    \vecu = \f{1}{\rho}\left(\phi_1 \vecu_1 + \phi_2 \vecu_2\right).
    \label{eq:mass-aver-velo}
\end{equation}
We define the material derivative for a generic function $g$ (scalar or vector-valued) by 
\begin{equation}
    \f{\DD g}{\DD t} = \f{\p g}{\p t}+ \vecu\cdot \nabla g, 
    \label{eq:material-deriv}
\end{equation}
and indicate the definition of the differential operator 
\[
    \vecu \cdot \nabla g= \sum_{j=1}^d \vecu_j \f{\p g}{\p x_j}  .
\]
In the following, we denote vectors by bold roman letters and we use bold Greek letters to denote second-order tensors. 

\subsection{Mass balance equations}
%The constitutive relation for this pressure term will be given in the following sections, however we just specify that $p$ is not an unknown but it is a given function of $\rho$ and $c$. 
We have the mass balance equations
\begin{equation}
    \begin{cases}
        \f{\p \phi_1}{\p t} + \dv \left(\phi_1 \vecu_1  \right) = F_1(\rho,c_1,c_2),\\
        \f{\p \phi_2}{\p t} + \dv \left(\phi_2 \vecu_2  \right) = F_2(\rho,c_1,c_2).
    \end{cases}
    \label{eq:mass-balance-i}
\end{equation}
The functions $F_i(\rho,c_1,c_2)$ ($i=1,2$) act as source or exchange terms of mass.

Summing the two equations, we obtain the continuity equation for the total density of the mixture, using $c=c_1$, and the relations \eqref{eq:relations-mass-frac}, we obtain the balance equation for the total density of the mixture 
\begin{equation}
    \f{\p \rho}{\p t} + \dv \left(\rho \vecu  \right) = F_1 + F_2 =: F_\rho.
    \label{eq:mass-balance-total-dens}
\end{equation} 
To obtain a system analogous to \eqref{eq:mass-balance-i}, we rewrite the first equation of \eqref{eq:mass-balance-i} using the definition of the mass fraction \eqref{eq:relations-mass-frac} to obtain 
\begin{equation}
    \f{\p \rho c}{\p t} + \dv \left(\rho c \vecu_1  \right) = F_1(\rho,c,1-c)=: F_c.
    \label{eq:mass-fraction-original}
\end{equation}
The mass of the component $1$ is transported by the average velocity $\vecu$ and the remaining diffusive flux $\mathbf{J}_1 = \rho c\left(\vecu-\vecu_1\right)$. Therefore, we can replace the previous equation by
\[
    \f{\p \rho c}{\p t} + \dv \left(\rho c \vecu  \right) = \dv \left( \mathbf{J}_1  \right) + F_c.
\]
Then, using the definition of the material derivative \eqref{eq:material-deriv} and the mass balance equation for the total mixture \eqref{eq:mass-balance-total-dens}, the left-hand side of the previous equation reads 
\[
    \f{\p \rho c}{\p t} + \dv \left(\rho c \vecu  \right) = \rho \f{\DD c}{\DD t} + c \left[\f{\p \rho}{\p t} + \dv\left(\rho \vecu\right) \right] = \rho \f{\DD c}{\DD t} + c  F_\rho.
\]
Altogether, we obtain the balance equation for the mass fraction of the component $1$
\begin{equation}
    \rho \f{\DD  c}{\DD t} =  \dv \left( \mathbf{J}_1  \right) + F_c - c F_\rho.
    \label{eq:mass-balance-c}
\end{equation} 
Since $c_2 = 1-c$, solving the equations \eqref{eq:mass-balance-total-dens} and \eqref{eq:mass-balance-c} is equivalent to solving the system \eqref{eq:mass-balance-i}. In the following, we refer to $c$ as the order parameter (terminology often used in the framework of the Cahn-Hilliard model \cite{cahn_free_1958,cahn_spinodal_1961}).

\subsection{Balance of linear momentum}
We write the balance of linear momentum \cite{Eck-2017-Modelling}, which describes the evolution of the velocity $\vecu$ due to internal stresses and external forces. Following continuum mechanics, the Cauchy stress tensor gives the stresses acting inside the mixture due to viscous and non-viscous effects. An additional stress must be taken into account to represent the effect of concentration gradients \cite{Ericksen-1990-Liquid}. Altogether, we assume that the stress tensor is a function of the total density $\rho$, the order parameter $c$ (i.e. the mass fraction of fluid $1$), its gradient $\nabla c$, and the total velocity of the mixture $\vecu$, \ie
\[
    \vecsig = \vecsig(\rho,c,\nabla c, \vecu).  
\]
The friction around the pores of the medium is modeled by a drag term in the balance equation \cite{narsilio_upscaling_2009} with a permeability coefficient $\kappa(\rho,c) = \kappa_1(\rho,c) + \kappa_2(\rho,c)$ (the sum of the two friction coefficients for each component of the mixture). The permeability coefficient relates the properties of the fluid and the porous medium. 

For each dimension (for example if $d=3$, then $j=\{x,y,z\}$), the balance of linear momentum reads \cite{Eck-2017-Modelling}  
\[
    \f{\p \rho \vecu_j}{\p t} + \dv\left(\rho \vecu_j \vecu\right) =  \dv\left(\vecsig\right)_j - \kappa(\rho, c) \vecu_j + F_{\vecu_j},
\]
where $F_{\vecu_j}(\vecu_j,\rho)$ represents the gain or loss of velocity in the $j$-th direction from different effects such as external forces. 
Then, using the continuity equation \eqref{eq:mass-balance-total-dens}, we can rearrange the left-hand side to obtain 
\[
    \f{\p \rho \vecu_j}{\p t} + \dv\left(\rho \vecu_j \vecu\right) =    \rho  \f{\DD \vecu_j}{\DD t} + \vecu_j\left[ \f{\p \rho}{\p t} + \dv\left(\rho \vecu \right)\right] = \rho  \f{\DD \vecu_j}{\DD t} + \vecu_j F_\rho + F_{\vecu_j}.
\]
Therefore, we have
\[
    \rho  \f{\DD \vecu_j}{\DD t} =  \dv(\vecsig)_j - \left(\kappa(\rho, c)+F_\rho\right) \vecu_j + F_{\vecu_j}.
\]
We can rewrite the balance of linear momentum in a more compact form
\begin{equation}
    \rho  \f{\DD \vecu}{\DD t}= \dv(\vecsig) -  \left(\kappa(\rho, c)+F_\rho\right)\vecu + F_\vecu,
    \label{eq:linear-momentum}
\end{equation}
where $F_\vecu(\vecu,\rho)$ is the vector of coordinates $F_{\bv_{j}}$.

\subsection{Energy balance}
The total energy of the mixture is the sum of the kinetic energy $\rho \f12 \abs{\vecu}^2$ and of the internal energy $\rho u$, where $u = u(\rho,c,\nabla c)$ is a specific internal energy. Compared to the classical conservation law for the total energy, we have an additional energy flux $\vectau \f{\DD c}{\DD t}$. Indeed, due to the interface region, surface effects must be taken into account. Following this direction, Gurtin~\cite{gurtin_thermo_1989} proposed to include in the second law of thermodynamics, the effect of an additional force called the \textit{microscopic-stress} which is related to forces acting at the microscopic scale. We denote this supplementary stress by $\vectau$.

Since we assume that the system is maintained in an isothermal state, the balance equation for the energy is given by \cite{Eck-2017-Modelling}
\begin{equation}
\begin{aligned}
    \f{\p}{\p t}\left(\rho \f12 \abs{\vecu}^2 + \rho u\right) &+ \dv\left( \rho\left( \f12 \abs{\vecu}^2 +  u\right) \vecu \right)  \\ 
    &= \dv\left(\vecsig^T \vecu \right) + \dv\left(\vectau \f{\DD c}{\DD t} \right) - \dv\left(\mathbf{q} \right) + \rho g + c_\rho F_\rho + c_c F_c + c_\vecu F_\vecu,
    \end{aligned}
    \label{eq:balance-energy-step1}
\end{equation}
where $\mathbf{q}$ is the heat flux and $\rho g$ is the density of heat sources to maintain the temperature constant. The last three terms in Equation~\eqref{eq:balance-energy-step1} account for the energy supply coming from the mass and velocity sources (see \eg~\cite{gurtin-mechanics-2010, Lam-2018-NSCH}). The prefactors $c_\rho, c_c, c_\vecu$ will be determined later to satisfy the free energy imbalance. 
Then, repeating the same calculations on the left-hand side to use the balance of mass \eqref{eq:mass-balance-total-dens}, we have
\[
    \f{\p}{\p t}\left(\rho \f12 \abs{\vecu}^2 + \rho u\right) + \dv\left( \rho\left( \f12 \abs{\vecu}^2 +  u\right) \vecu \right) = \rho\left[ \f{\DD}{\DD t}\left( \f12 \abs{\vecu}^2 + u\right)\right] + \left( \f12 \abs{\vecu}^2 + u\right) F_\rho.
\]
Applying the chain rule to the kinetic part, we obtain
\[
    \rho \f{\DD }{\DD t} \left(\f12 \abs{\vecu}^2 \right) = \rho \vecu \cdot \f{\DD \vecu}{\DD t},
\] 
and using the balance of linear momentum \eqref{eq:linear-momentum}, we arrive to
\[
    \rho \vecu \cdot \f{\DD \vecu}{\DD t} = \vecu \cdot\dv(\vecsig) - \left(\kappa(\rho,c)+F_\rho\right) \abs{\vecu}^2 + F_{\bv}\cdot \vecu.
\]
Using these previous equations inside \eqref{eq:balance-energy-step1}, we obtain the balance equation for the internal energy 
\[
\begin{aligned}
    \rho \f{\DD u}{\DD t} = \dv\left(\vecsig^T \vecu \right) - \vecu \cdot \dv\left(\vecsig \right) &+ \dv\left(\vectau \f{\DD c}{\DD t} \right)  + \left(\kappa(\rho,c)+F_\rho\right) \abs{\vecu}^2 - F_v \vecu \\
    &- \dv\left(\mathbf{q} \right) + \rho g - \left( \f12 \abs{\vecu}^2 + u\right) F_\rho + c_\rho F_\rho + c_c F_c + c_\vecu F_{\bv} .
\end{aligned}
\]
However, since
\[
    \vecu\cdot \left(\dv\left(\vecsig\right)\right) - \dv\left(\vecsig^T \vecu\right) = - \vecsig\cddot \nabla \vecu,
\] 
where $\nabla \vecu = \left(\p_{x_j}\vecu_i\right)_{i,j=1,\dots,d}$ is the Jacobi matrix and, we have ${A\cddot B = \sum_{i,j}A_{ij}B_{ij}}$, for two matrices $A,B$.
Altogether, we have the balance equation for the internal energy
\begin{equation}
\begin{aligned}
    \rho \f{\DD u}{\DD t} =\vecsig \cddot \nabla \vecu &+ \dv\left(\vectau \f{\DD c}{\DD t} \right) + \left(\kappa(\rho,c)+F_\rho\right) \abs{\vecu}^2 - F_v \vecu \\
    &- \dv\left(\mathbf{q} \right) + \rho g - \left( \f12 \abs{\vecu}^2 + u\right) F_\rho + c_\rho F_\rho + c_c F_c + c_\vecu F_\vecu.
    \end{aligned}
    \label{eq:energy-balance}
\end{equation} 

\subsection{Entropy balance and Clausius-Duhem inequality}
We aim to apply the second law of thermodynamics. To do so, we define the entropy ${s=s(\rho,c,\nabla c)}$ and the Helmholtz free energy $\mathcal{F}=\mathcal{F}(\rho, c,\nabla c)$, both related through the equation
\begin{equation}
    \mathcal{F}=u-T s,
    \label{eq:def-free-energy}
\end{equation}
where $T$ denotes the temperature. 

From the mass balance equation \eqref{eq:mass-balance-total-dens}, we have the entropy balance equation
\begin{equation}
    \f{\p \rho s}{\p t} + \dv(s \rho \vecu) = \rho \f{\DD s}{\DD t} + s\left[\f{\p \rho}{\p t} + \dv\left(\rho \vecu\right)\right] = \rho \f{\DD s}{\DD t} + s F_\rho.
    \label{eq:entropy-step1}
\end{equation}
Then, using the definition of the Helmholtz free energy \eqref{eq:def-free-energy} and the balance of energy \eqref{eq:energy-balance}, we obtain 
\begin{equation}
    \begin{aligned}
    \rho \f{\DD s}{\DD t} &= -\f\rho T \f{\DD \mathcal{F}}{\DD t} + \f\rho T \f{\DD u}{\DD t}\\
    &= -\f\rho T \f{\DD \mathcal{F}}{\DD t} + \begin{multlined}[t][10cm] \f1T\big[\vecsig \cddot \nabla \vecu + \dv\left(\vectau \f{\DD c}{\DD t} \right) + \left(\kappa(\rho,c)+F_\rho\right) \abs{\vecu}^2 - F_v \vecu \\
    - \dv\left(\mathbf{q} \right) + \rho g - \left( \f12 \abs{\vecu}^2 + u\right) F_\rho + c_\rho F_\rho + c_c F_c + c_\vecu F_\vecu\big],
     \end{multlined}
    \end{aligned}
    \label{eq:balance-entropy}
\end{equation}
where we have replaced the material derivative of the internal energy using its balance equation \eqref{eq:energy-balance}.

The constitutive relations for the functions constituting the Navier-Stokes-Cahn-Hilliard model are often derived to satisfy the Clausius-Duhem inequality (Coleman-Noll Procedure) \cite{Eck-2017-Modelling}. Indeed, this inequality provides a set of restrictions for the dissipative mechanisms occurring in the system. However, in our case, due to the presence of source terms, we can not ensure that this inequality holds without some assumptions on the  proliferation and friction of the fluid around the pores. Therefore, we use here a different method: the Lagrange multipliers method. Indeed, the Liu \cite{Liu_lagrange_1972} and M\"uller \cite{muller_thermodynamics_1985} method is based on using Lagrange multipliers to derive a set of restrictions on the constitutive relations that can be applied even in the presence of source terms. 

Following classical Thermodynamics \cite{muller_thermodynamics_1985}, we state the second law as an entropy inequality, \ie, the Clausius-Duhem inequality in the local form \cite{Eck-2017-Modelling}
\begin{equation}
    \rho \f{\DD s}{\DD t} \ge - \dv\left( \f{\bold{q}}{T} \right) +  \f{\rho g}{T}+ \dv\left(\mathcal{J}\right),
    \label{eq:Clausius-Duhem}
\end{equation}
where $\mathcal{J}$ is the entropy flux. The inequality \eqref{eq:Clausius-Duhem} results from the fact that the entropy of the mixture can only increase. Using the equation \eqref{eq:balance-entropy}, we obtain 
\begin{equation}
\begin{aligned}
    \f\rho T \f{D \mathcal{F}}{D t} - \f1T\big[&\vecsig \cddot \nabla \vecu + \dv\left(\vectau \f{\DD c}{\DD t} \right) + \left(\kappa(\rho,c)+F_\rho\right) \abs{\vecu}^2 \\
    &- F_v \vecu - \left( \f12 \abs{\vecu}^2 + u\right) F_\rho + c_\rho F_\rho + c_c F_c + c_\vecu F_\vecu \big] + \dv\left(\mathcal{J}\right) \le 0.
    \end{aligned}
    \label{eq:entropy-inbalance}
\end{equation}
Then, using the chain rule 
\[
   \f{D \mathcal F}{D t} = \f{D \rho}{D t}\f{\p \mathcal F}{\p \rho} + \f{D c}{D t}\f{\p \mathcal F}{\p c} +\f{D \nabla c}{D t}\cdot \f{\p \mathcal F}{\p \nabla c} ,
\]
and 
\[
   \f{D \nabla c}{D t}= \nabla\left[ \f{D  c}{D t}\right] - \left(\nabla \vecu\right)^T \nabla c, \quad \f{\DD \rho}{\DD t} = -\rho \dv(\vecu)+ F_\rho,
\]
in the entropy inequality \eqref{eq:entropy-inbalance}, we obtain 
\begin{equation}
    \begin{aligned}
    \rho& \left[ \left(-\rho \dv(\vecu) + F_\rho\right)\f{\p \mathcal F}{\p \rho} + \f{\DD c}{\DD t}\f{\p \mathcal F}{\p c} +\left(\nabla\left[ \f{D  c}{D t}\right] - \left(\nabla \vecu\right)^T\nabla c \right) \cdot \f{\p \mathcal F}{\p \nabla c} \right] - \dv\left(\vectau \f{\DD c}{\DD t} \right)-\vecsig : \nabla\bv \\
    & - \left[\left(\kappa(\rho,c)+F_\rho\right) \abs{\vecu}^2 - F_v \vecu - \left( \f12 \abs{\vecu}^2 + u\right) F_\rho + c_\rho F_\rho + c_c F_c + c_\vecu F_\vecu\right]  + T\dv\left(\mathcal{J}\right) \le 0.
    \end{aligned}
    \label{eq:entrop-ineq1}
\end{equation}
By the chain rule, we have 
\[
    \dv\left(\vectau \f{\DD c}{\DD t}\right) = \vectau  \nabla\left[ \f{D  c}{D t}\right] +  \f{\DD c}{\DD t} \dv\left(\vectau\right).
\]
Furthermore, we know that
\[
    -\rho^2 \dv\left(\vecu\right)\f{\p \mathcal F}{\p \rho} =  -\rho^2 \f{\p \mathcal F}{\p \rho}\Id \cddot \nabla \vecu,
\]
and
\[
    - \rho\left( \left(\nabla \vecu\right)^T \nabla c \right) \cdot \f{\p \mathcal F}{\p \nabla c} = - \rho \left(\nabla c  \otimes \f{\p \mathcal F}{\p \nabla c} \right)\cddot \nabla \vecu.
\]
Gathering the previous three relations and reorganizing the terms of \eqref{eq:entrop-ineq1}, we obtain
\begin{equation}
    \begin{aligned}
    &\left(  -\rho^2\f{\p \mathcal F}{\p \rho} \Id - \rho \nabla c  \otimes \f{\p \mathcal F}{\p \nabla c}  - \vecsig \right) \cddot \nabla \vecu + \left( \rho \f{\p \mathcal F}{\p c}  - \dv(\vectau)\right)  \f{\DD c}{\DD t} \\
    &+ \left(\rho\f{\p \mathcal F}{\partial \nabla c}-\vectau\right) \nabla \left[\f{\DD c}{\DD t}\right]  + T \dv\left(\mathcal{J}\right)
    \\   &- \left[\left(\kappa(\rho,c)+F_\rho\right) \abs{\vecu}^2 - F_v \vecu - \left( \f12 \abs{\vecu}^2 + u - \rho\f{\p \mathcal F}{\p \rho} \right) F_\rho + c_\rho F_\rho + c_c F_c + c_\vecu F_\vecu\right]\le 0. 
    \end{aligned} 
    \label{eq:entrop-ineq2}
\end{equation}

Then, we use Liu's Lagrange multipliers method \cite{Liu_lagrange_1972}. We denote by $L_c$ the Lagrange multiplier associated with the mass fraction equation \eqref{eq:mass-balance-c}. 
The method of Lagrange multipliers consists in setting the following local dissipation inequality that has to hold for arbitrary values of $(\rho,c,\nabla \rho,\nabla c,\vecu,p)$
\begin{equation}
    \begin{aligned}
    -D_\text{iss}:=&\left(  -\rho^2\f{\p \mathcal F}{\p \rho} \Id - \rho \nabla c  \otimes \f{\p \mathcal F}{\p \nabla c}  - \vecsig \right) \cddot \nabla \vecu \\
    &+ \left( \rho \f{\p \mathcal F}{\p c}  - \dv(\vectau)\right)  \f{\DD c}{\DD t} + \left(\rho\f{\p \mathcal F}{\partial \nabla c}-\vectau\right) \nabla \left[\f{\DD c}{\DD t}\right]  + T\dv\left(\mathcal{J}\right) \\
    &- \left[\left(\kappa(\rho,c)+F_\rho\right) \abs{\vecu}^2 - F_v \vecu - \left( \f12 \abs{\vecu}^2 + u - \rho\f{\p \mathcal F}{\p \rho} \right) F_\rho + c_\rho F_\rho + c_c F_c + c_\vecu F_\vecu\right]\\
    &- L_c \left( \rho \f{\DD  c}{\DD t} -  \dv \left( \mathbf{J}_1  \right) - F_c - c F_\rho  \right)\le 0.
    %& - L_\rho\left( \f{\DD  \rho}{\DD t} + \dv(\vecu)\rho - F_\rho  \right)\\
    %& - L_\vecu\left(\rho \f{\DD  \vecu}{\DD t} - \dv(\vecsig) +  \left(\kappa(\rho, c)+F_\rho\right)\vecu - F_\vecu  \right) \le 0.
    \end{aligned}
    \label{eq:diss-step1}
\end{equation}
Since,
\[
   \dv\left(L_c \mathbf{J}_1\right) = L_c \dv\left(\mathbf{J}_1 \right) + \nabla L_c \cdot \mathbf{J}_1,
\]
we reorganize the terms of \eqref{eq:diss-step1} to obtain
\begin{equation}
    \begin{aligned}
    -D_\text{iss}:=&\left(  -\rho^2\f{\p \mathcal F}{\p \rho} \Id - \rho \nabla c  \otimes \f{\p \mathcal F}{\p \nabla c}  - \vecsig \right) \cddot \nabla \vecu \\
    &+ \left( \rho \f{\p \mathcal F}{\p c}  - \dv(\vectau) - \rho L_c  \right)  \f{\DD c}{\DD t} + \left(\rho\f{\p \mathcal F}{\partial \nabla c}-\vectau\right) \nabla \left[\f{\DD c}{\DD t}\right]  + \dv\left( T \mathcal{J} + L_c \mathbf{J}_1\right)\\ &  - \nabla L_c \cdot \mathbf{J}_1 \\
    & \begin{multlined}[10cm]- 
    \big[\left(\kappa(\rho,c)+F_\rho\right) \abs{\vecu}^2 - F_v \vecu - \left( \f12 \abs{\vecu}^2 + u - \rho\f{\p \mathcal F}{\p \rho} \right) F_\rho \\+ c_\rho F_\rho + c_c F_c + c_\vecu F_\vecu - L_c (F_c+cF_\rho)\big] \le 0.
    \end{multlined}
    \end{aligned}
    \label{eq:entropy-ineq}
\end{equation}

\subsection{Constitutive assumptions and model equations}
First of all, we assume that the free energy density $\mathcal{F}$ is of Ginzburg-Landau type  and has the following form  \cite{cahn_free_1958,cahn_spinodal_1961}
\begin{equation}
    \mathcal{F}(\rho,c,\nabla c) \coloneqq  \psi_0(\rho,c) + \f\gamma{2} \abs{\nabla c}^2,
    \label{eq:constit-F}
\end{equation}
where $\psi_0$ is the homogeneous free energy accounting for the processes of phase separation and the gradient term $\f\gamma{2} \abs{\nabla c}^2$ represents the surface tension between the two phases. This free energy is the basis of the Cahn-Hilliard model which describes the phase separation occurring in binary mixtures. 
Furthermore, as obtained in Wise~\etal~\cite{wise_three-dimensional_2008}, the adhesion energy between different cell species is indeed well represented by such a choice of the free energy functional.

To satisfy the inequality \eqref{eq:entropy-ineq}, we first choose
\[
    \vectau \coloneqq  \rho \f{\p \mathcal F}{\partial \nabla c}= \gamma \rho \nabla c.
\] 
Then, we define the chemical potential $\mu(\rho, c, \nabla c)$ by 
\[
   \mu \coloneqq   \f{\p \mathcal F}{\p c}  - \f{1}{\rho}\dv(\vectau) = \f{\p \mathcal F}{\p c}  - \f{1}{\rho}\dv(\rho \f{\p \mathcal F}{\partial \nabla c}) = \f{\p \psi_0}{\p c} - \f\gamma\rho\dv\left(\rho \nabla c\right),
\]
which in turn gives a condition for the Lagrange multiplier 
\begin{equation}
    L_c = \mu.   
    \label{eq:constutive-lagrange1}
\end{equation}
Using these previous constitutive relations, we have already canceled some terms in the entropy inequality, \ie 
\[
    \left( \rho \f{\p \mathcal F}{\p c}  - \dv(\vectau) - \rho L_c  \right)  \f{\DD c}{\DD t} + \left(\rho\f{\p \mathcal F}{\partial \nabla c}-\vectau\right) \nabla \left[\f{\DD c}{\DD t}\right] = 0.
\]
Then, using classical results on isothermal diffusion \cite{Lowengrub-quasi-incompressible-1998,Eck-2017-Modelling}, we have
\begin{equation}
    \mathcal{J} \coloneqq - \f{\mu \mathbf{J}_1}{T},
    \label{eq:constit-Jentropy}
\end{equation}
and, using a generalized Fick's law, we have 
\begin{equation}
    \mathbf{J}_1 \coloneqq b(c)\nabla \mu,
    \label{eq:constit-J_1}
\end{equation}
where $b(c)$ is a nonnegative mobility function that we will specify in the following.  
The two constitutive relations for the diffusive fluxes \eqref{eq:constit-Jentropy} and \eqref{eq:constit-J_1} together with \eqref{eq:constutive-lagrange1}, we obtain
\[
    \dv\left(T \mathcal J +L_c \mathbf{J}_1\right) - \nabla L_c\cdot \mathbf{J}_1 = -b(c) \abs{\nabla \mu}^2 \le 0.
\]

Following \cite{Lowengrub-quasi-incompressible-1998,abels_diffuse_2008}, we define the pressure inside the mixture 
\begin{equation}
    p \coloneqq  \rho^2 \f{\p \psi_0}{\p \rho}.
    \label{eq:constit-press}
\end{equation}
% Mechanical part 
From standard rheology, we assume that the fluid satisfies Newton's rheological laws.
The stress tensor is composed of two parts for the viscous $ \tilde{\mathbf{P}} $ and non-viscous $\mathbf P$ contributions of stress
\begin{equation}
    \vecsig \coloneqq  \mathbf{P} + \tilde{\mathbf{P}},
    \label{eq:constit-sig}
\end{equation} 
and we have by standard continuum mechanics (see \eg~\cite{Anderson-1998-Diffuse,Eck-2017-Modelling, abels_diffuse_2008})
\begin{equation}
    \begin{cases}
        \mathbf P = -\left(p - \frac\gamma 2\abs{\nabla c}^2\right) \Id - \gamma  \rho \nabla c \otimes \nabla c ,\\
        \tilde{\mathbf P}  = \nu(c) \left(\nabla \vecu + \nabla \vecu^T\right) + \lambda(c)\dv\left( \vecu\right)\Id.
    \end{cases}
    \label{eq:constit-P}
\end{equation}
In~\eqref{eq:constit-P}, $\nu(c)$ denotes the shear viscosity and $\lambda(c) = \eta(c) - \frac{2}{3}\nu(c)$ where $\eta(c)$ is the dilatational viscosity that encodes the response of the fluid to volume changes.
The second term in the non-viscous part of the stress (namely $- \gamma  \left(\rho\nabla c \otimes \nabla c\right)$) represents capillary stresses that act at the interface of the two populations. 

Using \eqref{eq:constit-P}, we can cancel terms in \eqref{eq:entropy-ineq}
\[
    \left(  -\rho^2\f{\p \mathcal F}{\p \rho} \Id - \rho \nabla c  \otimes \f{\p \mathcal F}{\p \nabla c}  - \vecsig \right) \cddot \nabla \vecu  = 0.
\]
The remaining terms of the entropy inequality are the ones associated with proliferation and friction. The last step to satisfy the entropy inequality is to choose arbitrarily a value for $c_\rho$, such that 
\[
\begin{aligned}
       -\big[\left(\kappa(\rho,c)+F_\rho\right) \abs{\vecu}^2 &- F_v \vecu - \left( \f12 \abs{\vecu}^2 + u - \rho\f{\p \mathcal F}{\p \rho} \right) F_\rho \\
   &+ c_\rho F_\rho + c_c F_c + c_\vecu F_\vecu - L_c (F_c+cF_\rho)\big] \le 0.
   \end{aligned}
\]
Reorganizing the terms we have
\[
-\kappa(\rho,c) \abs{\vecu}^2 - F_\rho\left[c_\rho + \abs{\vecu}^2 - \left( \f12 \abs{\vecu}^2 + u - \rho\f{\p \mathcal F}{\p \rho} \right) - \mu c \right] - F_\vecu\left[ c_\vecu -\vecu  \right] - F_c \left[ c_c - \mu \right] \le 0.
\]
The obvious choices are
\[
\begin{cases}
    c_\rho = -\abs{\vecu}^2 + \left( \f12 \abs{\vecu}^2 + u - \rho\f{\p \mathcal F}{\p \rho} \right) + \mu c,\\
    c_\vecu = \vecu,\\
    c_c =  \mu.
    \end{cases}
\]
From the previous constitutive relations, we satisfy the dissipation inequality \eqref{eq:entropy-ineq}.

\subsection{Summary of the model's equations}
Using the previous constitutive relations our general model is the following compressible Navier-Stokes-Cahn-Hilliard system 
\begin{equation}
\begin{aligned}
        \f{\p \rho}{\p t}  &= - \dv\left(\rho \vecu\right)+ F_\rho,\\
        \rho\f{\DD c}{\DD t} &= \dv\left(b(c) \nabla \mu \right) + F_c - c F_\rho,\\ 
        \rho \mu &= -\gamma \dv \left(\rho \nabla c \right) + \rho\f{\p \psi_0}{\p c},\\
        \rho \f{\DD \vecu}{\DD t} &= \begin{multlined}[t][10cm] -\left[\nabla p+ \gamma\dv\left(\rho \nabla c \otimes \nabla c \right) \right] + \dv\left(\nu(c)\left(\nabla \vecu + \nabla \vecu^T \right) \right)\\
        -\f23 \dv\left(\nu(c)\dv\left(\vecu\right)\Id\right) + \dv\left(\eta(c) \dv\left(\vecu\right) \Id \right)  -  \left(\kappa(\rho, c)+F_\rho\right)\vecu + F_\vecu,
        \end{multlined}
        \end{aligned}
    \label{eq:summary-eqs}
\end{equation}

with $p$ defined in~\eqref{eq:constit-press}.

\section{Model reductions, general assumptions and biologically relevant choices of the model's functions}
\label{sec:assumptions-functions-2pop}
\subsection{Specific choices of functionals and model reductions}
\paragraph{Problem 1: General compressible NSCH with friction term and mass transfer.}
Assuming no creation of mass nor transfer of mass from the exterior of the system we have 
\begin{equation}
    F_c = -F_{1-c},
\end{equation}
leading to mass conservation
\begin{equation}
    F_\rho = 0. 
\end{equation}
Furthermore, we assume no external source of velocity and energy, leading to 
\begin{equation}
    F_\vecu = 0,\text{ and } F_u = 0.
\end{equation}

Furthermore, using the same simplifying assumption as in Abels and Feireisl~\cite{abels_diffuse_2008} to avoid vacuum zones, our final reduced system of equations is 
\begin{align}
    &\f{\p \rho}{\p t}  + \dv\left(\rho \vecu\right)\corr{}{=0},\label{eq:problem1-rho}
\\
    &\frac{\p \rho c}{\p t} + \dv\left(\rho c \vecu \right) = \dv\left(b( c) \nabla \mu \right) + F_c,\label{eq:problem1-c}
\\ 
    &\rho \mu = -\gamma\Delta c  + \rho\f{\p \psi_0}{\p c}, \label{eq:problem1-mu}
\\
     &\f{\p \rho \vecu}{\p t} + \dv\left(\rho \vecu\otimes \vecu\right) = \begin{multlined}[t][10cm] -\left[\nabla p+ \gamma\dv\left( \nabla c \otimes \nabla c - \frac{1}{2}\abs{\nabla c}^2 \Id\right) \right] + \dv\left(\nu(c)\left(\nabla \vecu + \nabla \vecu^T \right) \right)\\
   -\f23 \dv\left(\nu(c)\dv\left(\vecu\right)\Id\right) + \dv\left(\eta(c) \dv\left(\vecu\right) \Id \right)  - \kappa(\rho,c) \vecu,
    \end{multlined}\label{eq:problem1-u}
\end{align}

\paragraph{Problem 2: Biologically relevant variant of the system. }
We here refer to the two phases of the mixture as cell populations and not fluids. 
For this variant of the system, we assume the production of mass and neglect certain effects. Namely, we neglect inertia effects, and the viscosity of the fluid, and assume no external source of velocity. This leads to the momentum equation
\[
\nabla p + \kappa(\rho,c)\vecu = -\gamma  \dv\left( \nabla c \otimes \nabla c - \frac{1}{2}\abs{\nabla c}^2 \Id\right) - F_\rho \vecu.
\]

Assuming that one cell population proliferates while the other does not leads to 
\[
F_c = F_\rho = \rho c P_c(p),\quad \text{and}\quad F_{1-c} = 0, 
\]
with a pressure-dependent proliferation rate $P_c(p)\ge 0$. 
The growth function $P_c(p)$ is used to represent the capacity of cells to divide accordingly to the pressure exerted on them. It is well known that cells are able to divide as long as the pressure is not too large. Once a certain pressure $p_\text{max}$ is reached cells enter a quiescent state. Therefore, we assume that 
\begin{equation}
    P_c^\prime(p)\le 0, \quad \text{and} \quad P_c(p) = 0\quad \text{for}\quad p> p_\text{max}.
    \label{eq:assume-growth}
\end{equation}

Combining these changes, the model becomes
\begin{equation}
\begin{cases}
    \f{\p \rho}{\p t}  + \dv\left(\rho \vecu\right) = \rho c P_c(p),\\
    \frac{\p \rho c}{\p t} + \dv\left(\rho c \vecu \right) = \dv\left(b( c) \nabla \mu \right) + \rho c P_c(p),\\ 
    \rho \mu = -\gamma\Delta c  + \rho\f{\p \psi_0}{\p c},\\
     \nabla p + \kappa(\rho,c)\vecu = -\gamma  \dv\left( \nabla c \otimes \nabla c - \frac{1}{2}\abs{\nabla c}^2 \Id\right) - \rho c P_{c}(p) \vecu.
\end{cases}
\end{equation}

\subsection{Biologically consistent choices of functions}
As said in the derivation of the model, the free energy density $\mathcal{F}$ is the sum of two terms: $\f{\gamma}{2}\abs{\nabla c}^2$ taking into account the surface tension effects existing between the phases of the mixture and the potential $\psi_0(\rho,c)$ representing the cell-cell interactions and pressure. 
Thus, we choose 
\begin{equation}
    \psi_0(\rho,c) = \psi_e(\rho) + \psi_\text{mix}(\rho,c),
\end{equation}
with $ \psi_\text{mix}(\rho,c)= H(c)\log \rho + Q(c)$. Then, using the constitutive relation for the pressure we have
\begin{equation}
    p(\rho,c) = \rho^2\f{\p \psi_0}{\p\rho} = p_e(\rho) + \rho H(c).
\end{equation}

The function $b(c)$ is the active mobility of the cells. 

Let us explain how the choices of functions for the free energy density and mobility are motivated by biological observations.

To satisfy the conditions \eqref{eq:cond-mot}, we propose to choose
\begin{equation}
    b(c) = C_b c(1-c)^\alpha, \quad \alpha \ge 1,
    \label{eq:mob}
\end{equation}
where $C_b$ is a positive constant.

We use for the pressure a power law such that 
\begin{equation}
    p_e(\rho) = \f{1}{a-1}\rho^{a-1}.    
\end{equation}

For $H(c)$ and $G(c)$, two choices can be considered depending on the behavior of the cells we want to represent. If the two cell populations exert attractive forces when they recognize cells of the same type and repulsion with the other type, the potential has to take a form of a double-well for which the two stable phases are located at the bottom of the two wells (see e.g. Figure \ref{fig:double-well}). This is a situation close to the phase separation in binary fluids. 
Thermodynamically consistent potentials are of Ginzburg-Landau type with the presence of logarithmic terms. An example of double-well potential is given by 
\begin{equation}
    \psi_\text{mix} = \f12 \left(\alpha_1(1-c)\log(\rho(1-c)) + \alpha_2c\log(\rho c)\right) - \f\theta 2 (c-\frac 1 2 )^2 + k,
    \label{eq:potential-double-well}
\end{equation}
thus giving 
\[
    H(c) = \frac{1}{2}\left(\alpha_1(1-c) +  \alpha_2c\right), \quad Q(c) = \f12 \left(\alpha_1(1-c)\log(1-c) + \alpha_2c\log(c)\right) - \f\theta 2 (c-\frac 1 2 )^2 + k,
\]
where $\theta>1$, and $k,\alpha_1, \alpha_2 >0$ are an arbitrary constants.

\begin{figure}
    \begin{subfigure}[b]{0.49\textwidth}
    \centering
      \includegraphics[width=0.99 \linewidth]{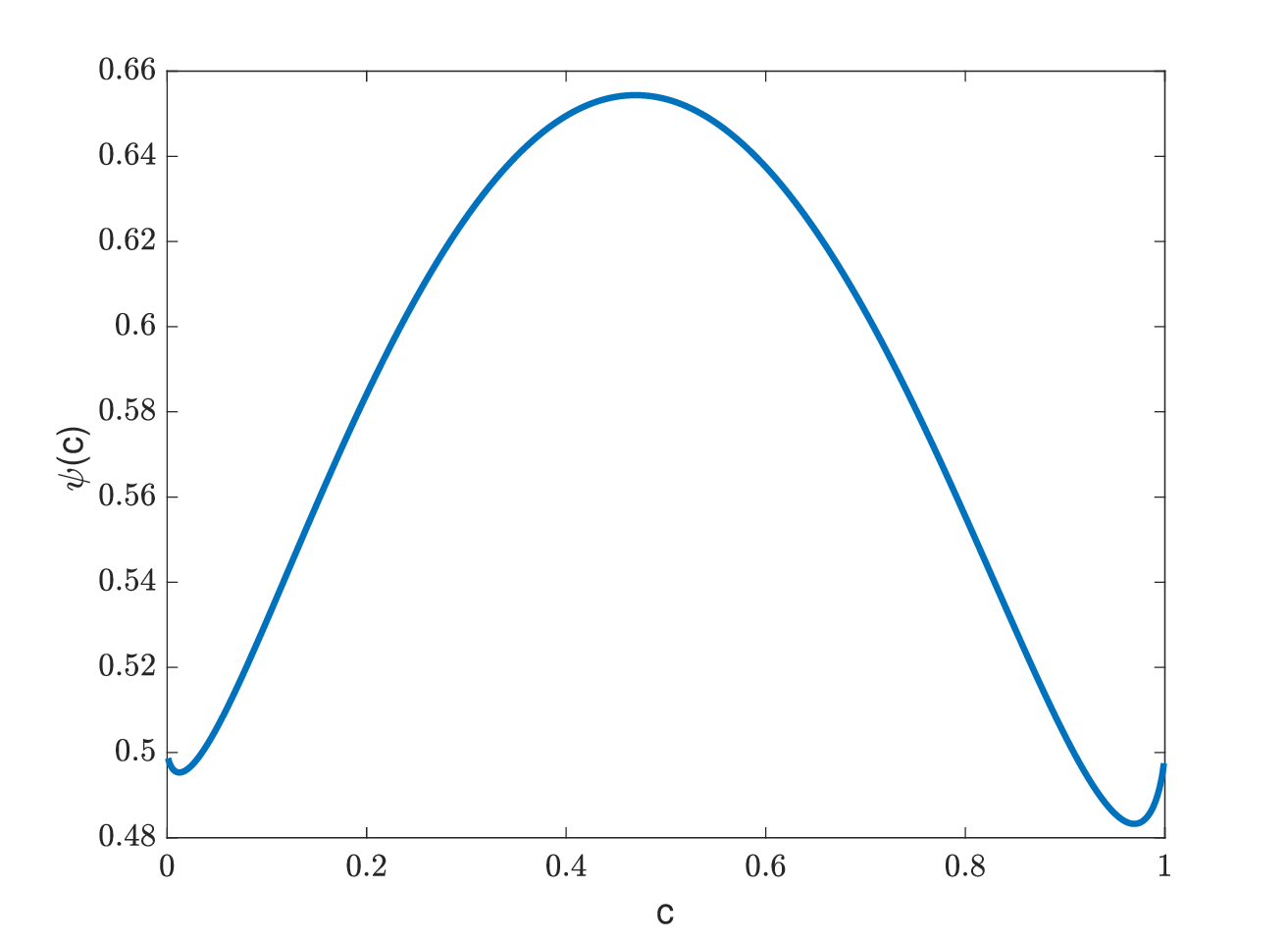}
      \caption{Double-well potential}
       \label{fig:double-well}
    \end{subfigure}
    \hfill  % maximize separation between the subfigures
    \begin{subfigure}[b]{0.49\textwidth}
    \centering
        \includegraphics[width=0.99 \linewidth]{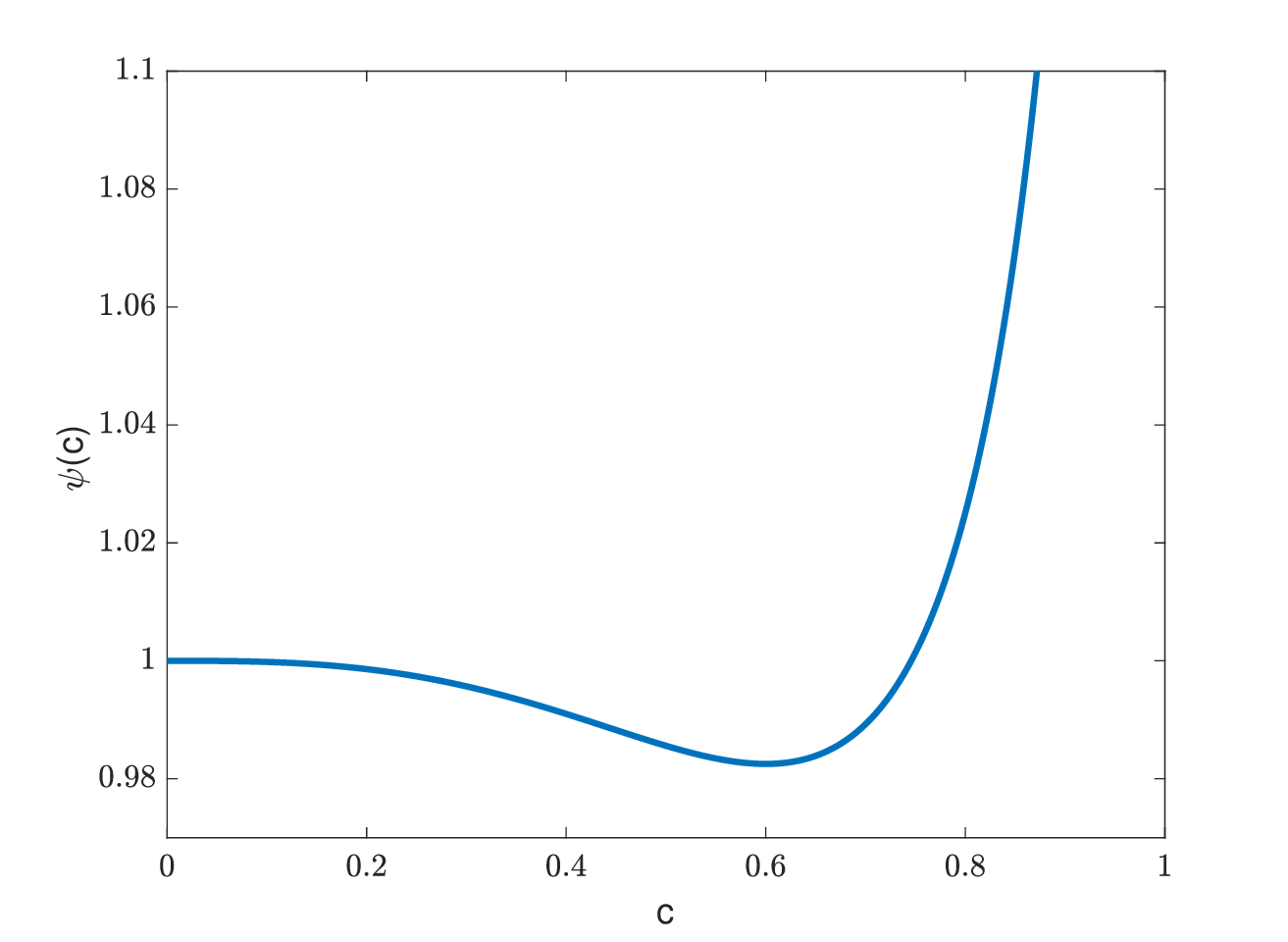}
        \caption{Single-well potential}
       \label{fig:single-well}
    \end{subfigure}
  \caption{For both figures $\rho =1$. Double-well logarithmic potential (left) with $\alpha_1 = 1.2$ and $\alpha_2 = 0.8$ and single-well logarithmic potential (right)} \label{fig:single-well-potential}
\end{figure}

To meet the phenomenological observations of the interaction between cells when the mixture is composed of only one cell population, a single-well potential seems more appropriate \cite{byrne_modelling_2004,chatelain_morphological_2011}. 

Indeed, when the distance between cells falls below a certain value (i.e. if the cell density is large enough), cells are attracted to each other. Then, it exists a threshold value called the mechanical equilibrium for which $ \partial_c \psi_0=0$ i.e. there is an equilibrium between attractive and repulsive forces. For larger cell densities, cells are packed too close to each other, they thus experience a repulsive force. When cells are so packed that they fill the whole control volume, then the repulsive force becomes infinite due to the pressure.  
The representation of such functional is depicted in Figure \ref{fig:single-well}.
A typical example of single-well potential which has been used for the modeling of living tissue and cancer \cite{chatelain_morphological_2011,Agosti-CH-2017} is 
\begin{equation}
    \psi_\text{mix}(\rho,c) = -(1-c_e)\log(\rho(1-c)) - \f{c^3}{3} - (1-c_e)\f{c^2}{2}-(1-c_e)c + k,
\end{equation}
thus giving
\begin{equation}
    H(c) = -(1-c_e),\quad Q(c) = -(1-c_e)\log(1-c)-\f{c^3}{3} - (1-c_e)\f{c^2}{2}-(1-c_e)c + k,
\end{equation}
where $k$ is an arbitrary constant. 

%\section{Analysis of the G-NSCH model}
% assumptions on the functions and results of existence  

\section{Description of the two-dimensional numerical scheme}
\label{app:2D_scheme}
We describe the two-dimensional scheme. This scheme possesses the same properties as the one-dimensional scheme.

We write the velocity field $\vecu = (u_x, u_y)$. 
System~\eqref{eq:main1}--\eqref{eq:main4} with the transformation proposed at the beginning of this section, reads
\begin{equation}
    \p_t \rho + \p_x(\rho u_x) + \p_y(\rho u_y) = 0,
\end{equation}
\begin{equation}
\begin{aligned}
    \p_t \left(\rho \begin{bmatrix}u_x\\u_y \end{bmatrix}\right) &+\begin{bmatrix}
        \p_x(\rho u_x^2 +p)\\ \p_y(\rho u_y^2 +p)
    \end{bmatrix} + \begin{bmatrix}
        \p_y(\rho u_x u_y)\\ \p_x(\rho u_x u_y)
    \end{bmatrix} = 2 \begin{bmatrix}
        \p_{x} \left(\nu(c) \p_x u_x \right)\\ \p_{y}\left(\nu(c)\p_y u_y\right) 
    \end{bmatrix} + \begin{bmatrix}
        \p_y\left(\nu(c) (\p_yu_x + \p_x u_y)\right) \\ \p_x\left(\nu(c) (\p_y u_x + \p_x u_y) \right)
    \end{bmatrix} \\
    &+  \begin{bmatrix}
        \p_x \left( \left(\eta(c)- \frac{2}{3}\nu(c)\right)( \p_x u_x + \p_y u_y)\right) \\ \p_y\left(\left(\eta(c)- \frac{2}{3}\nu(c)\right)(\p_x u_x + \p_y u_y) \right)
    \end{bmatrix} - \f\gamma2 \begin{bmatrix} \p_x ((\p_x c)^2 - (\p_yc)^2 ) \\
    \p_y ((\p_y c)^2 - (\p_xc)^2 ) 
    \end{bmatrix} -\gamma \begin{bmatrix}
        \p_y(\p_x c \p_y c)\\
        \p_x (\p_x c \p_y c)
    \end{bmatrix} \\ &- \kappa(\rho,c)\begin{bmatrix}u_x\\u_y \end{bmatrix},
\end{aligned}
\end{equation}
\begin{align}
    \rho \left( \p_t v + u_x \p_x v + u_y \p_y v \right) &= \frac{1}{T'(v)} \left(\p_x(b(c) \p_x \mu ) + \p_y(b(c) \p_y \mu) \right)  + \frac{1}{T'(v)} F_c,\\
    \rho \mu &= -\gamma T'(v) (\p_{xx} c  + \p_{yy} c ) - \gamma T''(v) \left((\p_x v)^2 + (\p_y v)^2 \right) + \rho \f{\p \psi_0}{\p c},\\
    \frac{\dd r}{\dd t} &= - \frac{r(t)}{E[t] +C_0} \int_\Omega b(c) \abs{\nabla \mu}^2 - \mu F_{c}\,\dd x.    
\end{align}

We introduce the notations $U = (\rho, \rho u_x, \rho u_y)$, $G(U) = (0 , -\kappa u_x, -\kappa u_y)$ and
\[
\begin{aligned}
F(U) = \Huge(\rho u_x, \rho u_x^2 + p - 2\nu(c) \p_x u_x & +\left(\f23\nu(c)-\eta(c)\right) (\p_x u_x +\p_y u_y) +\f12 \gamma \left((\p_x c)^2 -(\p_y c)^2\right), \\
&\rho u_x u_y - \nu(c)\left(\p_y u_x + \p_x u_y  \right) + \gamma \p_x c\p_y c \Huge),
\end{aligned}
\]
\[
\begin{aligned}
K(U) = \Huge(\rho u_y, \rho u_x u_y &- \nu(c)\left(\p_y u_x + \p_x u_y  \right) + \gamma \p_x c\p_y c, \\
&\rho u_y^2 + p - 2\nu(c) \p_y u_y +\left(\f23\nu(c)-\eta(c)\right) (\p_x u_x +\p_y u_y) +\f12 \gamma \left((\p_y c)^2 -(\p_x c)^2\right) \Huge).
\end{aligned}
\]

The stabilization (see~\cite{jin-1995-relaxation,Quaolin-2020-compressible})  of the Navier-Stokes part of our system reads, with $\iota >0$ the relaxation parameter,
\begin{equation}
    \begin{cases}
        \p_t U + \p_x V + \p_y W = G(U),\\
        \p_t V + A \p_x U = -\f1\iota (V - F(U)),\\
        \p_t W + B \p_y U = -\f1\iota(W - K(U)),
    \end{cases}
\end{equation}
in which $A = \text{diag}(a_1,a_2,a_3)$ and $B = \text{diag}(b_1,b_2,b_3)$. In the following, we choose 
\[
\begin{aligned}
    a_1=a_2 = a_3  = \max\{\sup\left(u_x + \sqrt{\partial_\rho p} \right)^2, \sup u_x^2 , \sup \left(u_x - \sqrt{\partial_\rho  p} \right)^2\},\\
    b_1= b_2 = b_3 = \max\{\sup\left(u_y + \sqrt{\partial_\rho  p} \right)^2, \sup u_y^2 , \sup \left(u_y - \sqrt{\partial_\rho p} \right)^2\}.
\end{aligned}
\]

We assume that our two-dimensional domain is a square $[0,L] \times [0,L]$. We discretize the domain using square control volumes of size $\Delta x \times \Delta y$. The cell centers are located at positions $\left(x_{j},y_{j}\right)$, and we approximate the value of a variable at the cell center by its mean, \eg
\[
\rho_{j,i} = \frac{1}{\Delta x \Delta y} \int_{x_{j-\f12}}^{x_{j+\f12}} \int_{y_{j-\f12}}^{y_{j+\f12}} \rho(\mathbf{x},t) \, \dd \mathbf{x} .
\]

Simply employing a first-order time discretization, the numerical scheme becomes 
\begin{align}
    U^*_{j,i} &= U^n_{j,i}, \label{eq:discrete2D1}\\
    V^*_{j,i} &= V^n_{j,i} - \frac{\Delta t}{\eta}\left(V^*_{j,i} - F(U^*_{j,i}) \right),\\
    W^*_{j,i} &= W^n_{j,i} - \frac{\Delta t}{\eta}\left(W^*_{j,i} - K(U^*_{j,i}) \right),\\
    U^{n+1}_{j,i} &= U^*_{j,i} - \frac{\Delta t}{\Delta x}\left(V^*_{j+\frac{1}{2},i} - V^*_{j-\frac{1}{2},i}\right) - \frac{\Delta t}{\Delta y}\left(W^*_{j,i+\frac{1}{2}} - W^*_{j,i-\frac{1}{2}}\right)  + \Delta t G(U_{i,j}^{n+1}),\label{eq:discrete2D3}\\
    V^{n+1}_{j,i} &= V^*_{j,i} - \frac{\Delta t}{\Delta x}A \left(U^*_{j+\frac{1}{2},i} - U^*_{j-\frac{1}{2},i}  \right), \label{eq:discrete2D4}\\
    W^{n+1}_{j,i} &= W^*_{j,i} - \frac{\Delta t}{\Delta y}B\left(U^*_{j,i+\frac{1}{2}} - U^*_{j,i-\frac{1}{2}}  \right), \label{eq:discrete2D4bis}\\
    \frac{\overline{v}^{n+1}_{j,i} - v^n_{j,i}}{\Delta t} &+ \vecu^{n+1}_{j,i}\cdot (\nabla \overline{v}^{n+1})_{j,i} = g(c^n,\mu^{n+1}, \rho^{n+1})_{j,i},\label{eq:discrete2D5}\\
    g(c^n,\mu^{n+1}, \rho^{n+1})_{j,i} &= \begin{multlined}[t][10cm]
    \frac{1}{T'(v^n_{j,i}) \rho^{n+1}_{j,i} \Delta x} \left((b(c^n)\nabla \mu^{n+1})_{j+\frac{1}{2},i} -(b(c^n)\nabla \mu^{n+1})_{j-\frac{1}{2},i}  \right) \\+ \frac{1}{T'(v^n_{j,i}) \rho^{n+1}_{j,i} \Delta y} \left((b(c^n)\nabla \mu^{n+1})_{j,i+\frac{1}{2}} -(b(c^n)\nabla \mu^{n+1})_{j,i-\frac{1}{2}}  \right) \\
    + \frac{F_c(\rho^n_{j,i},c^n_{j,i}) }{T'(v^n_{j,i})\rho^{n+1}_{j,i} },\end{multlined}\\
    \mu^{n+1}_{j,i} &= \frac{1}{\rho^{n}_{j,i}}\left(-\gamma T'(v^n_{j,i})(\Delta \bar{v}^{n+1})_{j,i} -\gamma T''(v^n_{j,i}) \abs{(\nabla v^n)_{j,i}}^2 \right)+ \left(\frac{\p \psi_0}{\p c}\right)_{j,i}^n ,\label{eq:discrete2D7}\\
    \int_\Omega  T(\lambda \overline{v}^{n+1}) \, \mathrm{d}x &= \int_\Omega c^n + \dt F_c \, \mathrm{d}x ,\label{eq:2DTdiscreteinteg} \\
    \overline{c}^{n+1}_{j,i} &= T( \lambda_j \overline {v}^{n+1}_{j,i}),\\
    \frac{1}{\Delta t}\left(r^{n+1} - r^n\right) &= - \frac{r^{n+1}}{E(\overline{c}^{n+1}) + C_0}    \int_\Omega b(\overline{c}^{n+1}) \abs{\nabla \mu^{n+1}}^2\,\dd \mathbf{x}+\nonumber\\
    &+\frac{{\color{blue}r^{n+1}}}{E(\overline{c}^{n+1}) + C_0}    \int_\Omega\mu^{n+1}F_{c}(\rho^{n+1},\overline{c}^{n+1})\,\dd \mathbf{x},\label{eq:discrete2D8}\\
    \xi^{n+1} &= \frac{r^{n+1}}{E(\overline{c}^{n+1}) + C_0},\label{eq:discrete2D9}\\  
    c^{n+1}_{j,i} &= \nu^{n+1} \overline{c}^{n+1}_{j,i}, \quad \text{with}\quad \nu^{n+1} = 1-(1-\xi^{n+1})^2,\label{eq:discrete2D10}\\
    v^{n+1}_{j,i} &= \nu^{n+1}\overline{v}^{n+1}_{j,i}. \label{eq:discrete2D11}
\end{align}

\end{document}